\pdfoutput=1 % for submission to arxiv.org

\documentclass[reqno,11pt]{amsart}

\usepackage[dvipsnames]{xcolor}
\usepackage{amssymb,amsmath,amsthm,amsbsy,
graphicx,
multicol,mathtools,
tikz, 
microtype, % to improve appearance
float, 
chngcntr, 
etoolbox, 
marvosym, % to use @ symbol
colonequals, % to use \equalscolon symbol
}

\usepackage{mathdots} % to use non-horizontal dots like \iddots

\usepackage[margin=1in,letterpaper,portrait]{geometry}

\usepackage[shortlabels]{enumitem} % used

\usepackage[utf8]{inputenc}

\usepackage[colorlinks=true, pdfstartview=FitV, linkcolor=red!60!gray, citecolor=blue,
filecolor=violet, urlcolor=violet]{hyperref}

\usepackage{colortbl}
\usepackage{comment}

\theoremstyle{definition} 
\newtheorem{theorem}{Theorem}[section]

\newtheorem{lemma}[theorem]{Lemma}
\newtheorem{lem}[theorem]{Lemma}

\newtheorem{proposition}[theorem]{Proposition}
\newtheorem{prop}[theorem]{Proposition}
\newtheorem{corollary}[theorem]{Corollary}
\newtheorem{algorithm}[theorem]{Algorithm}

\newtheorem{definition}[theorem]{Definition}
\newtheorem{notation}[theorem]{Notation}
\newtheorem{example}[theorem]{Example}
\newtheorem{remark}[theorem]{Remark}
\numberwithin{equation}{section}
\counterwithout{figure}{section}
\newtheorem{question}[theorem]{Question}

\definecolor{forGreen}{RGB}{0,120,0}

\usepackage{circledsteps} % to use command \Circled

\newcommand*\circled[1]{\tikz[baseline=(char.base)]{
\node[blue,shape=circle,draw,inner sep=0.7pt] (char) {#1};}} %blue circled

\newcommand*\circledplain[1]{\tikz[baseline=(char.base)]{
\node[shape=circle,draw,inner sep=0.7pt] (char) {#1};}} % circled no color

\usetikzlibrary{shapes.geometric}

\newcommand*\starred[1]{\tikz[baseline=(char.base)]{
\node[violet,shape=diamond,draw,inner sep=0.7pt] (char) {#1};}}

\newcommand*\rectangled[1]{\tikz[baseline=(char.base)]{
\node[red,shape=rectangle,draw,inner sep=1.1pt] (char) {#1};}}

\newcommand{\rw}[1]{\left[{#1}\right]} % notation for reduced words
\newcommand{\Coxeterlength}{\ell}

\newcommand{\pidown}{\pi_\downarrow}
\DeclareMathOperator{\sort}{{\sf sort}_{\it c}}
\newcommand{\SpecialSort}[1]{{\sf sort}_{#1}}
\DeclareMathOperator{\supp}{supp}
\DeclareMathOperator{\tamari}{Tamari}
\DeclareMathOperator{\cTamari}{c_{\tamari}}
\DeclareMathOperator{\DiagonalReadingWord}{{\mathcal R}_c}
\newcommand{\DiagWordSubseq}[1]{{\mathcal R(#1)}}
\newcommand{\HeapDiagonalReading}{H_c}
\DeclareMathOperator{\MapPermutation}{Perm}

\DeclareMathOperator{\reverse}{rev}

\newcommand{\RC}{\mathcal{R}}

\newcommand{\HeapLEQ}{\preccurlyeq}
\newcommand{\HeapLessThan}{\prec}

\newcommand{\bir}{{\sf Birk}}
\newcommand{\aff}{{\sf Aff}}
\newcommand{\ord}{\mathcal{O}}
\newcommand{\heap}{{\sf Heap}}
\newcommand{\nulabel}{\nu}

\newcommand{\upperbarletter}{u}
\newcommand{\lowerbarletter}{d}

\newcommand{\Heapwnot}{\heap(\sort(w_0))}

\newcommand\LatticeC{{\mathcal L}({\text{$c$-singletons})}}
\newcommand\SetOfCSingletons{\{\text{$c$-singletons}\}}
\newcommand{\prefixR}[1]{b_{#1}}

\newcommand{\X}{\color{red}\mathsf{X}}

\newcommand{\swpi}{\Pi^{\mathsf{SW}}_c}
\newcommand{\nepi}{\Pi^{\mathsf{NE}}_c}

\newcommand{\PermMatrix}[1]{X(#1)}

\newcommand{\wsuff}{w_{\text{suf}}}

\usepackage
[sorting=nty, % Sort by name, title, year
style=alphabetic,
maxbibnames=1000, % show all authors' names instead of et al
backend=biber,
isbn=false,
url=false,
eprint=true % to show arxiv info for unpublished preprints
]
{biblatex}

\title{
 $c$-Birkhoff polytopes
}
\author
{Esther Banaian}\address{Department of Mathematics, University of California, Riverside, Riverside, CA (USA)}

\author
{Sunita Chepuri}\address{Department of Mathematics and Computer Science, University of Puget Sound, Tacoma, WA (USA)}

\author
{Emily Gunawan}\address{Department of Mathematics and Statistics, University of Massachusetts Lowell, Lowell, MA (USA)}

\author
{Jianping Pan}\address{School of Mathematical and Statistical Sciences, Arizona State University, Tempe, AZ (USA)}

  	%Combinatorics of partially ordered sets

\begin{document}

\begin{abstract}
In a 2018 paper, Davis and Sagan studied several pattern-avoiding polytopes.  They found that a particular pattern-avoiding Birkhoff polytope had the same normalized volume as the order polytope of a certain poset, leading them to ask if the two polytopes were unimodularly equivalent.  Motivated by Davis and Sagan's question, in this paper we define a pattern-avoiding Birkhoff polytope called a $c$-Birkhoff polytope for each Coxeter element $c$ of the symmetric group.  We then show that the $c$-Birkhoff polytope is unimodularly equivalent to the order polytope of the heap poset of the $c$-sorting word of the longest permutation.  When $c=s_1s_2\dots s_{n}$, this result recovers an affirmative answer to Davis and Sagan's question.  Another consequence of this result is that the normalized volume of the $c$-Birkhoff polytope is the number of the longest chains in the (type A) $c$-Cambrian lattice.
\end{abstract}


\maketitle

\section{Introduction}
The \emph{Birkhoff--von Neumann polytope} 
$\mathcal{B}_n$ is the convex polytope of $n\times n$ doubly stochastic matrices, i.e., non-negative matrices whose rows and columns all sum to $1$.
It is often known as the assignment polytope, the perfect matching polytope of the complete bipartite graph, or simply the Birkhoff polytope. The Birkhoff polytope is in several well-studied classes of polytopes; it is a transportation polytope, a matching polytope, and a flow polytope.
Due to its rich geometric and combinatorial structure, the Birkhoff polytope shows up in many branches of mathematics, including combinatorics~\cite{Igor,BPD03,A05,P13}, representation theory~\cite{BHNP09}, optimization~\cite{M88,BS96,BS03} and statistics~\cite{BRT97,DE06,PSU19}.
The combinatorial properties of $\mathcal{B}_n$ have been thoroughly studied. For example, the Birkhoff--von Neumann Theorem~\cite{birkhoff1946three} shows that the vertices of $\mathcal{B}_n$ are the permutation matrices and that $\mathcal{B}_n$ has $n^2$ facets, and the faces of $\mathcal{B}_n$ are studied in~\cite{P13}.
There has been substantial work on finding the volume for Birkhoff polytopes~\cite{CRD99,BPD03,A05}.  The first exact formula was given by~\cite{DLY} via a combinatorial model called arborescences.  Their formula involves a double summation over permutations and arborescences.

It is natural to consider a subpolytope of a Birkhoff polytope given by a subset of the permutation matrices; for example, see  \cite{onn1993geometry, burggraf2013volumes, BC24}. In a 2018 paper, Davis and Sagan studied Birkhoff subpolytopes whose vertices correspond to pattern-avoiding permutations \cite{DS18}. They noticed that the sequence of normalized volumes of Birkhoff subpolytopes whose vertices avoid 132 and 312 is the sequence \cite[A003121]{oeis}. 
This sequence also counts 
shifted standard tableaux of staircase shape~\cite{FN14}, 
longest chains in the Tamari lattice, 
and the number of reduced words in a certain commutation class of the longest permutation $w_0$. 
In addition, since the set of 132 and 312 avoiding permutations forms a distributive sublattice of the right weak order, the number of $132,312$-avoiding permutations equals the number of order ideals of the poset of the join irreducibles of $132,312$-avoiding permutations.
As the $132,312$-avoiding Birkhoff subpolytope and the order polytope of the join irreducibles of $132,312$-avoiding permutations have the same volume and number of vertices, Davis and Sagan asked whether the two polytopes might be unimodularly equivalent.

In the same paper, Davis and Sagan pointed out~\cite[Remark~3.6]{DS18} that the $132,312$ avoiding permutations
are known to be the \emph{$c$-singletons} for a certain Coxeter element $c$ for the symmetric group.
These $c$-singleton permutations appear in several areas of representation theory and cluster algebras.  Given any Coxeter element $c$ of a finite Coxeter group $W$, the $c$-singletons are a subset of the \emph{$c$-sortable elements}, a special subset of $W$ introduced by Reading in \cite{reading07-clusters-paper4} to study the relationship between $W$-noncrossing partitions and generalized associahedra, as in \cite{chapoton2002polytopal}. The $c$-sortable elements of $W$ also have a connection to the theory of cluster algebras and tilting theory \cite{amiot2012preprojective,gyoda2022lattice}. 
The restriction of the weak order onto the $c$-sortable elements yields the $c$-Cambrian lattice, which is a generalization of the Tamari lattice \cite{Reading07-paper5}.

When constructing polytopes whose normal fan coincides with the $c$-Cambrian fan, Hohlweg, Lange, and Thomas introduced $c$-singletons \cite{HLT11}. These can be seen as the $c$-sortable elements which sit on the longest chains in the $c$-Cambrian lattice. These also have a formulation in terms of the commutation class of the $c$-sorting word of $w_0$ \cite{HLT11} and in terms of pattern avoidance \cite{reading2006cambrian}. The $c$-singletons for general $W$ and $c$ were enumerated in \cite{LL20}. 

The $c$-Cambrian lattice and its $c$-singletons have connection to maximal green sequences, an important concept in the theory of cluster algebras~\cite{CAI}. The $c$-Cambrian lattice is the oriented exchange graph of the cluster algebra whose initial quiver comes from~$c$. 
A maximal chain in the $c$-Cambrian lattice corresponds to clusters on a maximal green sequence of a quiver corresponding to~$c$; 
every longest chain in the $c$-Cambrian lattice corresponds to a longest maximal green sequence, and so our $c$-singletons correspond to clusters on the longest maximal  green sequences. For more information on maximal green sequences, see for example \cite{BDP14, Muller16, GarverMusiker17}.

In this paper, we study and define \emph{$c$-Birkhoff polytopes}, $\bir(c)$, which are Birkhoff subpolytopes whose vertices are the $c$-singletons for a fixed Coxeter element $c$ in $W = A_n$. The vertex set of these polytopes coincides with the set of order ideals of the heap of a certain reduced word of $w_0$ associated to $c$ (see \cite[Proposition 3]{LL20} and Proposition~\ref{prop:vertex-bijection}). Our main result, Theorem~\ref{thm:main}, is the following:

\vspace{0.5\baselineskip}
\noindent\textbf{Main Theorem.}
The $c$-Birkhoff polytope~$\bir(c)$ is unimodularly equivalent to the order polytope $\ord(H)$ where $H$ is the heap poset of the $c$-sorting word of $w_0$.

\vspace{0.5\baselineskip}

When $c = s_1s_2 \cdots s_n$, our result gives an affirmative answer to \cite[Question 5.1]{DS18}. Our result also has immediate corollaries regarding the volume of $c$-Birkhoff polytopes in terms of poset-theoretic information (see 
Corollary~\ref{cor:volume}).

In order to prove our main result, we must explicitly understand the affine span of the vertices of $\bir(c)$.
We do this in Section~\ref{Sec:Relations} by studying the consequences of the characterization of $c$-singletons in terms of pattern avoidance from \cite{reading2006cambrian}. The fact that $\ord(H)$ is full-dimensional implies that our relations provide a complete description of this affine space (see Remark~\ref{rmk:RelationsAreEverything}). When $c = s_1s_2\cdots s_n$, that is, when the $c$-singletons coincide with 132 and 312 avoiding permutations, the relations have a simple reformulation which may be of independent interest (see Corollary \ref{cor:SummingRelationsTamari}).

The paper is organized as follows. In Section~\ref{sec: sec 2}, we review necessary background on heap posets and $c$-singletons. Section~\ref{sec:backgroud.polytopes} defines the order polytopes (Section~\ref{sec:order polytope}), the $c$-Birkhoff polytopes (Section~\ref{subsec.cbir}), and unimodular equivalence of polytopes (Section~\ref{subsec.unimodular}).
In Section~\ref{Sec:Relations}, we describe linear relations on the affine span of the $c$-Birkhoff polytopes.
Section~\ref{sec.lattice} constructs a unimodular transformation of the $c$-Birkhoff polytope to a polytope living in a lower-dimensional ambient space via a lattice-preserving projection $\Pi_c$ (Definition~\ref{defn.projection}).
Finally in Section~\ref{sec:UnimodularEquivalence} we prove our main theorem (Theorem~\ref{thm:main}) by defining a unimodular transformation from $\Pi_c(\bir(c))$ to $\ord(H)$.  
In Section~\ref{Sec:ex}, we provide some examples including the Tamari orientation (Section~\ref{subsec.tamari}) and bipartite orientation (Section~\ref{subsec.bipartite}).
We end in Section~\ref{Sec:future} with some discussions on potential future directions.

Note that an extended abstract for this paper appeared in the FPSAC 2024 proceedings~\cite{BCGP}.

\section{Background: Heaps of $c$-singletons}\label{sec: sec 2}

Below we give some necessary notation for type A Coxeter group. For general theory on Coxeter groups, see for example~\cite{BB05}.
Denote the symmetric group on $n+1$ elements by $A_n$.
We can represent a permutation $w\in A_n$ in \emph{one-line notation} as $w = w(1)w(2)\cdots w(n+1)$. 
For each $i \in \{1, \ldots, n\}$, we write $s_i \in A_n$ to denote the \emph{simple reflection} (or \emph{adjacent transposition})
that swaps $i$ and $i+1$ and fixes all other letters. 
Every permutation can be expressed as a product of simple reflections. Given $w\in A_n$, the minimum number of simple reflections among all such expressions for $w$ is called the \emph{(Coxeter) length} of $w$, and is denoted by $\ell(w)$. 
A \emph{reduced decomposition} of $w$ is an expression $w=s_{i_1} \cdots s_{i_{\ell(w)}}$ realizing the Coxeter length of $w$. To simplify notation, we refer to such a  decomposition via its \emph{reduced word} $\rw{i_1 \cdots i_{\ell(w)}}$. 
The \emph{support} $\supp(w)$ of a permutation $w$ is the set of letters that appear in a reduced words of $w$; this set only depends on $w$ and not the choice of reduced words. 
For example, consider  $w = 51342 \in A_4$. 
One of its reduced decompositions is 
$s_{4}s_{2}s_{3}s_{2}s_{4}s_{1} $ 
with $\rw{423241}$ as the corresponding reduced word, its length is  $\Coxeterlength(w)=6$, and $\supp(w)=\{1,2,3,4\}$.

A \emph{Coxeter element} $c$ in 
$A_n$ is a product of all $n$ simple reflections in any order, where each reflection appears exactly once.  
The \emph{longest permutation} of $A_n$ is the permutation $w_0= (n+1)n \dots 321$ and 
$\ell(w_0)=\binom{n+1}{2}$.

Simple reflections satisfy \emph{commutation relations} of the form $s_i s_j=s_j s_i$ for $|i-j| > 1$. 
An application of a commutation relation to a product of simple reflections is called a \emph{commutation move}. 
When referring to reduced words, 
we will say adjacent letters $i$ and $j$ in a reduced word \emph{commute} when $|i - j| > 1$. 
Given a reduced word $\rw{u}$ of a permutation, the equivalence class consisting of all words that can be obtained from $\rw{u}$ by a sequence of commutation moves is the \emph{commutation class} of $\rw{u}$.

\subsection{Heaps}

We begin by reviewing the classical theory of heaps~\cite{Viennot86lecturenotes}, which was used in~\cite{Ste96} to study fully commutative elements of a Coxeter group. 
Heaps also appeared as ``the natural partial orders" and were used to study certain acyclic domains
in \cite[Definition 6]{GR08} and \cite[Definition 1]{LL20}. 
For a detailed list of attributions on the theory of heaps, see~\cite[Solutions to Exercise 3.123(ab)]{Sta12}.

\begin{definition}
\label{defn:heap of a reduced word}
Given a reduced word $\rw{a}=\rw{a_1 \cdots a_\ell}$ of a permutation, 
consider the
partial order $\HeapLEQ$ on the set $\{ 1, \dots, \ell\}$ obtained via the transitive closure of the relations 
\[x
\HeapLessThan
y\]
for $x < y$ such that
$|a_x-a_y| \leq 1$. 
For each $1 \leq x \leq \ell$, the \emph{label} of the poset element $x$ is  $a_x$. This labeled poset is called the \emph{heap} for $\rw{a}$, denoted $\heap(\rw{a})$. We identify the elements of $\heap(\rw{a})$ with the elements of its underlying poset, and refer to them accordingly. The Hasse diagram for this poset with elements $\{1,\ldots, \ell\}$ replaced by their labels is called the \emph{heap diagram} for $\rw{a}$. 
The labels in the heap diagram are drawn in increasing order
from left to right.  
\end{definition}

Note that a label $j$ corresponds to the simple reflection $s_j$. In our figures, we represent each label~$j$ by $s_j$ for clarity.

\newcommand{\posetelt}[1]{{#1}}
\newcommand{\posetLABEL}[1]{{#1}}
\newcommand{\posetLABELwithS}[1]{s_{#1}}
\begin{figure}[htb!]
\begin{center}
\begin{tikzpicture}[scale=0.7]
\node (s1a) at (0,0) {$\posetelt{1}$}; 
\node (s2a) at (1,1) {$\posetelt{2}$}; 
\node (s3a) at (2,2) {$\posetelt{5}$}; 
\node (s4a) at (3,1) {$\posetelt{4}$}; 
\node (s1b) at (0,2) {$\posetelt{3}$}; 
\node (s2b) at (1,3) {$\posetelt{6}$}; 
\node (s3b) at (2,4) {$\posetelt{9}$}; 
\node (s4b) at (3,3) {$\posetelt{8}$}; 
\node (s1c) at (0,4) {$\posetelt{7}$}; 
\node (s2c) at (1,5) {$\posetelt{10}$};

\draw [
black, thick, shorten <=-2pt, shorten >=-2pt] (s1a) -- (s2a);
\draw [
black, thick, shorten <=-2pt, shorten >=-2pt] (s2a) -- (s3a);
\draw [
black, thick, shorten <=-2pt, shorten >=-2pt] (s3a) -- (s4a);

\draw [
black, thick, shorten <=-2pt, shorten >=-2pt] (s1b) -- (s2b);
\draw [
black, thick, shorten <=-2pt, shorten >=-2pt] (s2b) -- (s3b);
\draw [
black, thick, shorten <=-2pt, shorten >=-2pt] (s3b) -- (s4b);

\draw [black, thick, shorten <=-2pt, shorten >=-2pt] (s1c) -- (s2c);

\draw [thick, shorten <=-2pt, shorten >=-2pt] (s2a) -- (s1b); 
\draw [thick, shorten <=-2pt, shorten >=-2pt] (s3a) -- (s4b);
\draw [thick, shorten <=-2pt, shorten >=-2pt]  (s3a) -- (s2b) -- (s1c);
\draw [thick, shorten <=-2pt, shorten >=-2pt] (s3b) -- (s2c);
\end{tikzpicture}
\qquad
\qquad
\qquad
\begin{tikzpicture}[scale=0.7]

\filldraw [rounded corners=10pt, pink!20]
(0-.6,0+.2)  
-- (2,2+.6) 
-- (3+.6,1+.1) 
-- (3,1-.6)  
-- (2,2-.6) 
-- (0,0-.6) 
-- cycle;

\filldraw [rounded corners=10pt,yellow!20]
(0-.6,0+.2 +2)  
-- (2,2+.6 +2) 
-- (3+.6,1+.1 +2) 
-- (3,1-.6 +2)  
-- (2,2-.6 +2) 
-- (0,0-.6 +2) 
-- cycle;

\filldraw [rounded corners=10pt, forGreen!20] 
(0-.6,0+.2 +4)  
-- (1,1+.6 +4)  
-- (1+.6,1-.2 +4) 
-- (0,0-.6 +4) 
-- cycle;

\node (s1a) at (0,0) {$\posetLABELwithS{1}$};
\node (s2a) at (1,1) {$\posetLABELwithS{2}$};
\node (s3a) at (2,2) {$\posetLABELwithS{3}$};
\node (s4a) at (3,1) {$\posetLABELwithS{4}$};
\node (s1b) at (0,2) {$\posetLABELwithS{1}$};
\node (s2b) at (1,3) {$\posetLABELwithS{2}$};
\node (s3b) at (2,4) {$\posetLABELwithS{3}$};
\node (s4b) at (3,3) {$\posetLABELwithS{4}$};
\node (s1c) at (0,4) {$\posetLABELwithS{1}$};
\node (s2c) at (1,5) {$\posetLABELwithS{2}$};

\draw [
black, thick, shorten <=-2pt, shorten >=-2pt] (s1a) -- (s2a);
\draw [
black, thick, shorten <=-2pt, shorten >=-2pt] (s2a) -- (s3a);
\draw [
black, thick, shorten <=-2pt, shorten >=-2pt] (s3a) -- (s4a);
\draw [
black, thick, shorten <=-2pt, shorten >=-2pt] (s1b) -- (s2b) node[pos=0.5, left=5mm]{};
\draw [
black, thick, shorten <=-2pt, shorten >=-2pt] (s2b) -- (s3b);
\draw [
black, thick, shorten <=-2pt, shorten >=-2pt] (s3b) -- (s4b);

\draw [
black, thick, shorten <=-2pt, shorten >=-2pt] (s1c) -- (s2c);

\draw [
black, thick, 
shorten <=-2pt, shorten >=-2pt] (s2a) -- (s1b); 
\draw [thick, shorten <=-2pt, shorten >=-2pt] (s3a) -- (s4b);
\draw [thick, shorten <=-2pt, shorten >=-2pt]  (s3a) -- (s2b) -- (s1c);
\draw [thick, shorten <=-2pt, shorten >=-2pt] (s3b) -- (s2c);
\end{tikzpicture}
\caption{Left: Hasse diagram of the underlying poset of $\heap([1214321432])$. Right: the heap diagram of $\heap([1214321432])$, with each label $j$ replaced by $s_j$. }
\label{fig:longest element heap A4 for c 1243}
\end{center}
\end{figure}

\begin{example}\label{ex:fig:longest element heap A4 for c 1243}
Consider 
a reduced word $\rw{a}=\rw{a_1 \dots a_{10}}=\rw{1 21 4321 432}$  
of the longest element $w_0$ in $A_4$. 
\begin{enumerate}
\item 
Figure~\ref{fig:longest element heap A4 for c 1243} (left) shows a Hasse diagram of the poset $\HeapLEQ$ corresponding to $\rw{a}$.
Here $\ell=10$ and so the elements of the heap poset $\heap(\rw{a})$ 
are $\{1,2,\dots, 10 \}$.
\item 
Figure~\ref{fig:longest element heap A4 for c 1243} (right) shows the heap diagram for $\heap(\rw{a})$. The possible labels of the poset elements are $\{1,2,3,4\}$.
\end{enumerate}
\end{example}

\begin{remark}\label{rem:poset elements with the same label are comparable}
If two elements $i$ and $j$ of the heap of a reduced word have the same label, then $i$ and $j$ are comparable.
\end{remark}

We now explain how linear extensions of $\heap(\rw{a})$ relate to 
the commutation class of $\rw{a}$.

\begin{definition}\label{def:Labeled linear extensions}
A \emph{linear extension} 
$\pi = \pi(1) \cdots \pi(\ell)$ 
of a partial order $\HeapLEQ$ on $\{1,\ldots, \ell\}$
is a total order on the poset elements that is consistent with the structure of the poset, that  is, $\pi(x) \HeapLessThan \pi(y)$ implies $x < y$.
A \emph{labeled linear extension} of the heap of a reduced word $\rw{a}=\rw{a_1 \cdots a_\ell}$
is a word $\rw{a_{\pi(1)} \cdots a_{\pi(\ell)}}$, 
where 
$\pi = \pi(1) \cdots \pi(\ell)$ 
 is a linear extension of the heap. 
\end{definition}

\begin{proposition}
{\cite[Proof of Proposition~2.2]{Ste96} and~\cite[Solutions to Exercise 3.123(ab)]{Sta12}}
\label{prop:set of labeled linear extensions is commutativity class Stembridge}
Given a reduced word $\rw{a}$,  the set of labeled linear extensions of the heap for $\rw{a}$ is the commutation class of $\rw{a}$.
\end{proposition}

\begin{example}\label{ex:linear ext A4 for c 1243}
Labeled linear extensions of $\heap(\rw{a})$ from Example~\ref{ex:fig:longest element heap A4 for c 1243} include 
$\rw{a}$ itself,
$\rw{1243 1243 12}$, 
and 
$\rw{4123  4123  12}$. Notice that these reduced words all belong to the same commutation class due to Proposition~\ref{prop:set of labeled linear extensions is commutativity class Stembridge}. 
\end{example}

We consider order ideals of $\heap(\rw{a})$ to retain their labels.  
Note that if $I$ is an order ideal of $\heap(\rw{a})$ then $I$ is the heap of a prefix of some labeled linear extension of $\rw{a}$.

\subsection{$c$-sorting words and $c$-sortable permutations}\label{sec:subwords K1 K2 etc}
In this section, we review $c$-sorting words and $c$-sortable elements, which were introduced in~\cite{reading07-clusters-paper4}. 
Fix a reduced word $\rw{a_1 a_2 \dots a_n}$ for a Coxeter element $c$, and define an infinite word
\[
c^{\infty} \coloneqq a_1 a_2 \dots a_n \mid a_1 a_2 \dots a_n\mid \cdots
\]
consisting of repeated copies of the given reduced word for $c$. 
The symbols ``$\mid$" are ``dividers" which facilitate the definition of sortable elements. 
The \emph{$c$-sorting word} of $w\in A_n$ is 
the lexicographically first (as a sequence of positions in $c^\infty$) subword of $c^\infty$ that is a reduced word for $w$. 
We denote this word by $\sort(w)$. 
If a word $\rw{u}=\rw{u_1 \dots u_\ell}$ is the $c$-sorting word of a permutation, we refer to $\rw{u}$ as a \emph{$c$-sorting word}.

Section 5 of the survey chapter \cite{ReadingBeyond} describes the following useful algorithm to find the $c$-sorting word of a permutation.

\begin{algorithm}[Finding the $c$-sorting word]\label{algorithmForSortC} 
Fix a Coxeter element $c \in A_n$ with reduced word $\rw{a_1 a_2 \dots a_n}$. 
Let $w \in A_n$ with length $\ell \coloneqq \ell(w)$. Set $w' \coloneqq w$ and $\rw{u}$ to the empty word.
\begin{enumerate}[(Step 1)]
\item\label{step.1} If $w'$ is the identity permutation, the algorithm will terminate here. Otherwise, we continue with~\ref{step.2}.

\item\label{step.2} 
For $i=1,2,...,n$, we perform the following: 
\indent 
\begin{itemize}
\item 
Let $q=a_i$. 
\item If $q$ is initial in $w'$ (that is, if $w'$ has a reduced word starting with $q$), then append $q$ to $\rw{u}$ and set $w'\coloneqq s_q w$.
\end{itemize}

\noindent 
At the end of the loop, continue with~\ref{step.1}. 
\end{enumerate}
The reduced word $\rw{u_1\dots u_{\ell}}$ is the $c$-sorting word $\sort(w)$ for $w$. 
\end{algorithm}

If $w$ is not the identity permutation, 
we can think of $\sort(w)$ as a concatenation of nonempty subwords of $\rw{a_1 \ldots a_n}$:
\[
\sort(w)= \rw{K_1 \mid K_2  \mid ... \mid K_p}. 
\]
Each of the subwords $K_1, K_2, \dots,K_p$ occurs between two adjacent dividers, 
so we have $x \in K_j$ if $x$ is in the $j\textsuperscript{th}$ copy of $\rw{a_1 \ldots a_n}$ inside $c^\infty$. 
We say that the identity permutation is \emph{$c$-sortable}, and  
a non-identity permutation $w$ is \emph{$c$-sortable} if $K_1 \supseteq K_2 \supseteq \cdots \supseteq K_p$ as sets.

\begin{remark}
The definition of a $c$-sorting word requires a choice of reduced word for $c$. However, note that the $c$-sorting words for $w$ arising from different reduced words for $c$ are related by commutation of letters, with no commutations across dividers. 
For this reason, the set of $c$-sortable permutations does not depend on the choice of reduced word for $c$.  
\end{remark}

\begin{example}
Consider the Coxeter element $c = s_1 s_2 s_3 s_4=\rw{1234}$  of $A_4$. 
Then the $c$-sorting word of the permutation $42351$ is 
$\rw{1 2 3 4 \mid  2 \mid 1}$. 
Our subwords are $K_1 = \rw{1234}$, $K_2 = \rw{2}$, and $K_3 = \rw{1}$.  
Since $K_2\not\supseteq  K_3$ as sets, these subwords do not form a nested sequence and therefore $42351$ is not $c$-sortable.
On the other hand, the permutation $43215$ has $c$-sorting word $\rw{123 \mid 12 \mid 1}$ and is $c$-sortable. 
\end{example}

\begin{proposition}[Corollary 4.4 of \cite{reading07-clusters-paper4}]
Given any Coxeter element $c$, the longest permutation $w_0$ is $c$-sortable.
\end{proposition}

\begin{example}\label{eg.tamari.sort}
Let $c=\rw{12 ... n}$. Then the $c$-sorting word for $w_0$ is 
\[\sort(w_0)=\rw{1 ... (n-1)n \mid 1 ... (n-1) \mid ... \mid 12 \mid 1 }.
\]
We can see in this case that $w_0$ is indeed $c$-sortable.
\end{example}

\begin{remark}
\label{remark:layers of H are subwords of c}
\begin{enumerate}[(1)] 
\item \label{itm:remark:layers of H are subwords of c:1}
We can construct the heap diagram $H\coloneqq\Heapwnot$ by gluing ``layers" labeled by the letters of subwords $K_1, K_2, \dots, K_p$. 
For example, consider Figure~\ref{fig:longest element heap A4 for c 1243} (right) which illustrates the heap diagram for $\sort(w_0)$ for $c=\rw{1423}$. 
We have the subwords $K_1=\rw{1423}$, $K_2=\rw{1423}$, and $K_3=\rw{12}$ of $c$. The bottom-most layer of $H$ corresponds to $K_1$, the second layer to $K_2$, and the top layer to $K_3$.
In general, $H$ can be partitioned into these layers of $\heap(K_1), \heap(K_2), \dots, \heap(K_p)$.

\item \label{itm:2:remark:layers of H are subwords of c}
Conversely, given the heap 
$H\coloneqq\Heapwnot$, we can obtain $\sort(w_0)$ by reading $H$ one layer of $\heap(\rw{a_1 \dots a_n})$ at a time, from bottom to top, so that $K_1=\rw{a_1 \dots a_n}$ corresponds to the bottom layer, and $K_2$ is the second subheap of $\heap(c)$, and so on. We can think of this linear extension of $H$ obtained by reading one layer of $c$ at a time as the \emph{$c$-reading word} of~$H$.
\end{enumerate}
\end{remark}

We will describe the precise construction of $\Heapwnot$ in Algorithm \ref{alg:algorithm for heap using diagonal reading word}.

Reading showed in~\cite[Theorem 1.2]{Reading07-paper5}
that the restriction of the right weak order to $c$-sortable elements is a lattice which is isomorphic to an important quotient of the right weak order called the $c$-Cambrian lattice~\cite{reading2006cambrian}; see also a representation-theoretic proof in \cite[Theorem 7.8]{LatticeTheoryOfTorsionClasses23}. 
For the Coxeter element $c=s_1 s_2 \dots s_n$, the $c$-sortable elements form the Tamari lattice. For this reason, we refer to this Coxeter element as the ``Tamari" Coxeter element of $A_n$. 
Cambrian lattices and $c$-sortable elements  have strong connections
to cluster algebras, representation theory, and many areas of combinatorics, and they are widely studied; see for example \cite{gyoda2022lattice, ingalls2009noncrossing, 
reading2009cambrian,BGMS23}.
We will be interested in a subclass of $c$-sortable elements, called $c$-singletons, which we describe next.

\subsection{$c$-singleton permutations}
\label{sec:c-singleton}

There is an order-preserving projection $\pidown^c$ from the right weak order on $A_n$ to itself which sends an element $w$ to the largest $c$-sortable element that is weakly below $w$~\cite[Proposition 3.2]{Reading07-paper5}. 
In \cite{HLT11}, Hohlweg, Lange, and Thomas used this map to introduce an important subclass of $c$-sortable elements:  
A $c$-sortable $w$ is called a 
\emph{$c$-singleton} if the preimage of $\{ w\}$ under $\pidown^c$
is the singleton $\{ w\}$ itself.

\begin{remark}\label{rem:emily temp}
It follows from the definition of $c$-singletons that  $w$ is a $c$-singleton if and only if $w$ lies in a longest chain in the $c$-Cambrian lattice.
\end{remark}

In~\cite{Tho06AnalogueOfDistributivity,ingalls2009noncrossing}, it is shown that the $c$-Cambrian lattice is \emph{trim}. 
The \emph{spine} of a trim lattice, introduced in~\cite{Tho06AnalogueOfDistributivity}, consists of elements that lie in some longest chain of the lattice.  
Therefore, the $c$-singletons form the spine of the $c$-Cambrian lattice.

We will use the following equivalent characterization of $c$-singletons.

\begin{theorem}{\cite[Theorem 2.2]{HLT11}}\label{thm:c-sing-prefix}
A permutation $w$ is a $c$-singleton 
if and only if some reduced word of $w$ is a prefix of a word in the commutation class of $\sort(w_0)$, the $c$-sorting word of the longest permutation $w_0$.
\end{theorem}

\begin{corollary}[{Corollary of Theorem \ref{thm:c-sing-prefix} and Proposition \ref{prop:set of labeled linear extensions is commutativity class Stembridge}}]
\label{corollary:c-singleton iff linear extension of order ideal}
A permutation $w$ is a $c$-singleton if and only if there exists a reduced word $\rw{u}$ 
of $w$ and an order ideal $I$ of $\Heapwnot$ such that $I=\heap(\rw{u})$. 

\end{corollary}

In view of Corollary~\ref{corollary:c-singleton iff linear extension of order ideal}, we can write an algorithm for finding the $c$-sorting word of a  $c$-singleton using heaps; it is equivalent to Algorithm~\ref{algorithmForSortC} when the permutation is a $c$-singleton. In this algorithm, we will consider the infinite heap $H^\infty$ of the  word $c^\infty$ and then we will iterate through the elements in $H^\infty$ layer by layer.

Fix a Coxeter element $c \in A_n$ with reduced word $\rw{a_1 a_2 \dots a_n}$. 
Following Definition~\ref{defn:heap of a reduced word}, we can construct a heap, which we denote as $H^\infty$, for the infinite word $c^\infty$.  Recall that $c^\infty=ccc\ldots =\rw{a_1 a_2 \dots }=\rw{a_i}_{i\in \mathbb{N}}$ where $a_{kn+i} = a_i$ for $k\in \mathbb{Z}_{\geq0}$ and $1\leq i \leq n$.  As in Definition~\ref{defn:heap of a reduced word}, the underlying poset of $H^\infty$ is the partial order $\HeapLEQ$ on the set of positive integers obtained via the transitive closure of the relations $x
\HeapLessThan
y$ for $x < y$ such that
$|a_x-a_y| \leq 1$, and the label of the poset element $x$ is $a_x$.

\begin{algorithm}[Finding a $c$-sorting word using heaps]\label{algorithmForSortC_c-singletons} 
First, we identify $\Heapwnot$ with an order ideal $H$ of $H^\infty$ (see Remark~\ref{remark:layers of H are subwords of c}\ref{itm:remark:layers of H are subwords of c:1}).   Note that the cardinality of $H$ is $\Coxeterlength(w_0)$, but in general the set of elements of $H$ is not $\{1,2,\dots,\Coxeterlength(w_0) \}$.  In contrast, the set of elements of $\Heapwnot$, defined as per Definition~\ref{defn:heap of a reduced word}, is $\{1,2,\dots,\Coxeterlength(w_0) \}$.

Let $I$ be an  order ideal of $H$.  
Set $I'\coloneqq I$ and let $\rw{u}$ be the empty word.

\noindent For $z=1,2,3,\dots$:
\begin{itemize}
\item
Let $q$ be the label of the poset element $z$. 
\item
If $z$ is a minimal element of $I'$, then append $q$ to $\rw{u}$ and set $I'\coloneqq I'  \setminus \{z\}$.
\item
If $I'$ is empty, terminate the loop.
\end{itemize}
\end{algorithm}

\begin{remark}
Notice that Algorithm~\ref{algorithmForSortC_c-singletons} always outputs a reduced word: since $I$ is an order ideal of $H$, $\rw{u}$ is a prefix of some reduced word of $w_0$.
\end{remark}

\begin{prop}\label{prop:u_is csorting_for_w_and_a_is_csorting for_w0}
Let $H$ be as defined in Algorithm~\ref{algorithmForSortC_c-singletons}, and let $I$ be an order ideal of $H$.  
Let $\rw{u}$ be the word constructed in Algorithm~\ref{algorithmForSortC_c-singletons} and let $w$ be the permutation obtained from $\rw{u}$. 
Then $\rw{u}$ is the $c$-sorting word of $w$.
\end{prop}
\begin{proof}
First note that $I$ is an order ideal of $H^\infty$. 
For each element $z$ of $H^\infty$, $z=kn+i$ for some $k\in\mathbb{Z}$ and $1\leq i\leq n$.  We will show by induction that the $\rw{u}$ we get after applying the $z\textsuperscript{th}$ loop of Algorithm~\ref{algorithmForSortC_c-singletons} on $I$, which we will call $\rw{u}_z$, is equivalent to the $\rw{u}$ we get after applying the $i\textsuperscript{th}$ loop of \ref{step.2} of Algorithm~\ref{algorithmForSortC}  on $w$ for the $(k+1)\textsuperscript{st}$ time, which we will call $\rw{u}_{\tilde{z}}$, if both algorithms go through this many loops.

Both algorithms initialize with the empty word.  Thus we will define $\rw{u}_0=\rw{u}_{\tilde{0}}$.
We will now assume that for some $z\geq 0$, $\rw{u}_z=\rw{u}_{\tilde{z}}$ and that neither algorithm terminates at this stage.
\begin{itemize}
\item First consider Algorithm~\ref{algorithmForSortC_c-singletons}.  Since $z+1$ is the smallest element remaining in the linear extension $1,2,3,\dots$ of $H^\infty$, $z+1$ is either a minimal element in $I'$ or it is not in $I'$.  Let $q$ be the label of $z+1$ in $H^\infty$.  This means $\rw{u}_{z+1}=\rw{u}_{z}[q]$ if $z+1$ is in $I$ and $\rw{u}_{z+1}=\rw{u}_z$ if $z+1$ is not in $I$.
\item Now consider Algorithm~\ref{algorithmForSortC}.  Let $z+1=kn+i$ with $k\in\mathbb{Z}$ and $1\leq i\leq n$.  At this step in Algorithm~\ref{algorithmForSortC} we set $q=a_i$ which is the same as the label of $z+1$ in $H^\infty$.  If $\rw{u} = \rw{u_1 \dots u_\ell}$ then $\rw{u}_{\tilde{z}}=\rw{u}_z = \rw{u_1 \dots u_j}$ for some $1\leq j\leq \ell$.  By construction, at this stage $w'$ is the permutation obtained from $\rw{u_{j+1} \dots u_\ell}$,
which is the suffix of $\rw{u}$ that contains the letters added in the $(z+1)\textsuperscript{st}$ and later loops of Algorithm~\ref{algorithmForSortC_c-singletons}.  Thus, if $z+1$ is in $I'$ then $q$ is initial in $w'$.  In this case, this step of Algorithm~\ref{algorithmForSortC} updates $\rw{u}_{\widetilde{z+1}}=\rw{u}_{\tilde{z}}\rw{q}$. If $z+1$ is not in $I'$ then by Remark~\ref{rem:poset elements with the same label are comparable} there is no element labeled $q$ in $I'$.  This means $q$ is not in the support of $w'$ and so $q$ is not initial in $w'$.  Thus in this case, $\rw{u}_{\tilde{z}}=\rw{u}_{\tilde{z}}$ remains the same.
\end{itemize}
Therefore the output of the two algorithms after this loop match.

If Algorithm~\ref{algorithmForSortC} does not output $\rw{u}$ then this is because it terminated before Algorithm~\ref{algorithmForSortC_c-singletons} updated $\rw{u}_z$ to $\rw{u}$ or because it changed $\rw{u}_{\tilde{z}}$ after $\rw{u}_{\tilde{z}}$ was set to $\rw{u}$.  In the first situation, Algorithm~\ref{algorithmForSortC} outputs $\rw{u}_{\tilde{z}}=\rw{u}_z$ for some $z$ value where the length of $\rw{u}_z$ is less than that of $\rw{u}$.  In the second situation Algorithm~\ref{algorithmForSortC} outputs a word of length longer than $\rw{u}$.  Neither of these are possible because $w$ has the same length as $\rw{u}$.  

Thus, we can conclude that Algorithm~\ref{algorithmForSortC_c-singletons} and Algorithm~\ref{algorithmForSortC} give the same output.  This means~$\rw{u}$ is the $c$-sorting word for $w$.
\end{proof}

For a poset $P$, we define $J(P)$ to be the lattice of order ideals of $P$.  The following  tells us that every order ideal of $\Heapwnot$ is the heap of a $c$-sorting word. 

\begin{proposition}
\label{prop:I=Heap([r])}
Let $\rw{v}=\rw{v_1 v_2 \dots v_{\Coxeterlength(w_0)}}$ denote $\sort(w_0)$, and consider the labeled poset $H=\heap(\rw{v})$ on $\{1,2,...,\Coxeterlength(w_0) \}$, following Definition~\ref{defn:heap of a reduced word}. 
Given an order ideal $I$ of $H$, 
let $\rw{v}_I=\rw{v_i}_{i\in I}$ denote the subword of $\rw{v}$ at positions $I$.
\begin{enumerate}[(1)]
\item\label{prop:I=Heap([r]):itm:1} Then $\rw{v}_I$ is a $c$-sorting word. 
\item \label{prop:I=Heap([r]):itm:2}
The map 
\begin{align}
\label{eq:MapPermutation}
\MapPermutation \colon 
J(H)
& \to  
\SetOfCSingletons 
\\
\nonumber I &\mapsto w,
\end{align}
where $w$ is the permutation obtained from the reduced word $\rw{v}_I$, is well-defined. 
Furthermore, $\heap(\sort(w))=I$.
\end{enumerate}
\end{proposition}

\begin{proof}
\ref{prop:I=Heap([r]):itm:1} The statement follows from Proposition~\ref{prop:u_is csorting_for_w_and_a_is_csorting for_w0}.

\ref{prop:I=Heap([r]):itm:2}
By construction of $\rw{v}_I$, we have $I=\heap(\rw{v}_I)$.   Corollary~\ref{corollary:c-singleton iff linear extension of order ideal} then tells us that $w$ is a $c$-singleton.
Furthermore, since $\rw{v}_I$ is a $c$-sorting word by part \ref{prop:I=Heap([r]):itm:1}, we have $\rw{v}_I=\sort(w)$, and thus  $\heap(\sort(w))=I$.
\end{proof}

The set of $c$-singletons forms a distributive sublattice of the right weak order due to \cite[Proposition 2.5]{HLT11}. We denote this lattice by $\LatticeC$.
We have that $v \leq w$ in $\LatticeC$ if and only if $\heap(\sort(v))$ is an order ideal of $\heap(\sort(w))$.

A proof appeared in \cite[Proposition 3]{LL20} that $\LatticeC$ is isomorphic to the lattice of order ideals of $\heap(\sort(w_0))$. 
In the following proposition, we add the additional detail that each order ideal of $H$ is the heap of the $c$-sorting word for the corresponding $c$-singleton from~Proposition~\ref{prop:I=Heap([r])}.

\begin{proposition}\label{prop:vertex-bijection}
Let $H\coloneqq\Heapwnot$. Then 
the map 
\begin{align*}
f: 
\LatticeC 
& \to J(H) 
\\
w & \mapsto \heap(\sort(w))
\end{align*}
is a poset isomorphism. 
\end{proposition}
\begin{proof}
First, we show that  $f(w)$ is indeed an order ideal in $J(H)$ when $w$ is a $c$-singleton. 
By Corollary~\ref{corollary:c-singleton iff linear extension of order ideal}, there is a reduced word $\rw{u}$ of $w$ and an order ideal $I$ of $H$, such that $\heap(\rw{u})=I$. 
Proposition~\ref{prop:I=Heap([r])} tells us that $I=\heap(\sort(w))$, and so $f(w)=\heap(\sort(w))$ is in $J(H)$. 

The inverse map of $f$ sends each order ideal $I$ of $H$ to the permutation $\MapPermutation(I)$ defined in~\eqref{eq:MapPermutation}. 
Moreover, the map $f$ is a poset isomorphism because $v \leq w$ in $\LatticeC$ if and only if $f(v)=\heap(\sort(v)) \subseteq \heap(\sort(w))=f(w)$. 
\end{proof}

\subsection{A pattern-avoidance criterion for $c$-singletons}\label{subsec:PatternAvoidanceCsingletons}

We present here an alternate classification of $c$-singletons via a pattern-avoidance property, following \cite{reading2006cambrian}. 
Let $c$ be a Coxeter element in $A_n$. 
We partition the set $[2,n]:= \{2,3,\ldots,n-1,n\}$ into two sets, $\overline{[2,n]}$ and $\underline{[2,n]}$: if $i$ appears to the right of $i-1$ in a reduced word of $c$, then $i \in \underline{[2,n]}$; otherwise $i \in \overline{[2,n]}$. If $i \in \underline{[2,n]}$, we will call $i$ a ``lower-barred number'' and sometimes emphasize this by writing $\underline{i}$. Similarly if $i \in \overline{[2,n]}$ we will call $i$ an ``upper-barred number'' and sometimes emphasize this by writing $\overline{i}$. We denote the lower-barred numbers by $\lowerbarletter_1< \dots < \lowerbarletter_r$ and 
the upper-barred numbers by $\upperbarletter_1 < \dots < \upperbarletter_s$. Occasionally we will also include a lower or upper bar on this notation for emphasis.  For convenience, we will define $\lowerbarletter_{0} = 1$ and $\lowerbarletter_{r+1} = n+1$ (note $1$ and $n+1$ are not lower-barred numbers).

\begin{remark}\label{rem:Coxeter element in cycle notation} 
It is well-known that we can write the 
Coxeter element $c$ in $A_n$ in cycle notation as \[c=
(1~  \underline{\lowerbarletter_1} ~ ...~  \underline{\lowerbarletter_r} ~ (n+1) ~ \overline{\upperbarletter_s} ~ ... ~ \overline{\upperbarletter_1})
=
(\overline{\upperbarletter_s} ~ ... ~\overline{\upperbarletter_1} ~1~ \underline{\lowerbarletter_1} ~ ... ~ \underline{\lowerbarletter_r} ~(n+1)).\] 
See, for example, \cite[Section 3]{reading07-clusters-paper4}. Therefore we can compute $c^k(u_t)$ as 
\[
c^k(u_t) = \begin{cases}
    u_{t-k}, & \text{if } t>k\,,\\
    d_{k-t}, & \text{if } t\leq k\,.
 \end{cases}
\]
\end{remark}

We say $w \in A_n$ \emph{avoids the pattern $31\underline{2}$} if there is no $i<j<k$ such that $w(j) < w(k) < w(i)$ and $w(k) \in \underline{[2,n]}$. That is, $w$ avoids all patterns $312$ where the ``2'' is a lower-barred number. One can define avoidance of the patterns $13\underline{2}, \overline{2}13,$ and $ \overline{2}31$ analogously. 

The following result characterizes 
$c$-sortable and $c$-singleton permutations using pattern-avoidance.

\begin{proposition}{\cite[Proposition 5.7]{reading2006cambrian}}
\label{prop:pattern avoidance characterization type A}
A permutation $w \in A_n$ is $c$-sortable if and only if $w$ avoids the patterns $31\underline{2}$
and 
$\overline{2}31$. 
Furthermore, 
a $c$-sortable permutation $w$ is a $c$-singleton if and only if $w$ avoids the patterns $13\underline{2}$
and 
$\overline{2}13$. 
\end{proposition}

The pattern-avoidance criteria for $c$-singletons can be rephrased as the following.

\begin{corollary}\label{cor:PatternAvoidanceRephrased}
 A permutation $w \in A_n$ is a $c$-singleton if and only if, when written in one-line notation, we have the following.
 \begin{itemize}
     \item For each lower-barred number $\lowerbarletter$, all numbers in $[1,\lowerbarletter-1]$ appear after  $\lowerbarletter$ or all numbers in $[\lowerbarletter+1,n+1]$ appear after $\lowerbarletter$.
     \item For each upper-barred number $\upperbarletter$, all numbers in $[1,\upperbarletter-1]$ appear before $u$ or all numbers in $[\upperbarletter+1,n+1]$ appear before $u$. 
 \end{itemize}
\end{corollary}

\begin{example}
Let $c = s_1s_2s_5s_4s_3s_6 = (1 \, \underline{2} \, \underline{3} \, \underline{6} \, 7 \, \overline{5} \, \overline{4})$. The heap diagram for 
$\heap(\rw{125436})$ is given below.

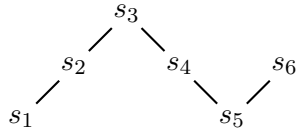
\begin{center}
\begin{tikzpicture}[scale=0.7]
\node (1) at (0,0) {$\posetLABELwithS{1}$}; 
\node (2) at (1,1) {$\posetLABELwithS{2}$}; 
\node (3) at (2,2) {$\posetLABELwithS{3}$}; 
\node (4) at (3,1) {$\posetLABELwithS{4}$}; 
\node (5) at (4,0) {$\posetLABELwithS{5}$}; 
\node (6) at (5,1) {$\posetLABELwithS{6}$}; 

\draw [
black, thick, shorten <=-2pt, shorten >=-2pt] (1) -- (2);
\draw [
black, thick, shorten <=-2pt, shorten >=-2pt] (3) -- (2);
\draw [
black, thick, shorten <=-2pt, shorten >=-2pt] (3) -- (4);
\draw [
black, thick, shorten <=-2pt, shorten >=-2pt] (4) -- (5);
\draw [
black, thick, shorten <=-2pt, shorten >=-2pt] (6) -- (5);
\end{tikzpicture}
\end{center}

We have $\underline{[2,6]} = \{2,3,6\}$ and $\overline{[2,6]} = \{4,5\}$. We see that the permutation $2167345$ is not a $c$-singleton because it contains several instances of $13\underline{2}$, for example $\mathbf{\circledplain{2}}1\mathbf{\circledplain{6}}7\mathbf{\circledplain{3}}45$. 
One can check that the permutation $3672145$ avoids all required patterns so it is a $c$-singleton. 
\end{example}

\subsection{Longest chains in the $c$-Cambrian lattice}

In this section, we describe a generalization of some of the objects counted by the sequence \cite[A003121]{oeis}.

\begin{lemma}\label{lem:threesets}
The following sets are in one-to-one correspondence.
\begin{enumerate}[(1)]
\item \label{lem:threesets:itm:linearextensions} linear extensions of $\Heapwnot$
\item 
\label{lem:threesets:itm:commutationclass}
words in the commutation class of $\sort(w_0)$
\item 
\label{lem:threesets:itm:latticeC}
maximal chains in 
$\LatticeC$ 
\item 
\label{lem:threesets:itm:Cambrian}
longest chains in the  $c$-Cambrian lattice
\item 
\label{lem:threesets:itm:patterns}
maximal chains in the lattice of permutations which avoid the four patterns $31\underline{2}$,  
$\overline{2}31$,  
$13\underline{2}$, 
and 
$\overline{2}13$, as a sublattice of the weak order on the symmetric group $A_n$
\end{enumerate}
\end{lemma}
\begin{proof} 
Let $H=\Heapwnot$. 
By Proposition~\ref{prop:set of labeled linear extensions is commutativity class Stembridge}, the commutation class of $\sort(w_0)$ correspond to the linear extensions of $H$, so the sets 
\ref{lem:threesets:itm:linearextensions} and \ref{lem:threesets:itm:commutationclass} are in bijection.

The poset isomorphism given in Proposition~\ref{prop:vertex-bijection} tells us that the maximal chains in $\LatticeC$ are in one-to-one correspondence with words in the commutation class of $\sort(w_0)$, proving that the sets
\ref{lem:threesets:itm:commutationclass} and \ref{lem:threesets:itm:latticeC} are in bijection.

Recall that $\LatticeC$ is a sublattice of the $c$-Cambrian lattice.  Remark~\ref{rem:emily temp} tells us that a permutation $w$ is a $c$-singleton if and only if $w$ lies in a longest chain in the $c$-Cambrian lattice, so it follows that the longest chains in the $c$-Cambrian lattice are precisely the longest chains in $\LatticeC$. 
Since $\LatticeC$ is a distributive lattice, it is a graded lattice, that is, all maximal chains have the same length; thus the longest chains in $\LatticeC$ are the same as maximal chains in $\LatticeC$.  
Thus, the sets \ref{lem:threesets:itm:latticeC} and \ref{lem:threesets:itm:Cambrian} are equal.

Proposition~\ref{prop:pattern avoidance characterization type A}
tells us that 
a permutation $w$ is a $c$-singleton if and only if $w$ avoids the patterns $31\underline{2}$,  
$\overline{2}31$,  
$13\underline{2}$, 
and 
$\overline{2}13$, so the sets \ref{lem:threesets:itm:latticeC} and \ref{lem:threesets:itm:patterns} are equal.
\end{proof}

\begin{remark}
For
$\cTamari=\rw{12 ... n}$, each set in Lemma~\ref{lem:threesets} is enumerated by 
the sequence \cite[A003121]{oeis}. 
The four patterns given in \ref{lem:threesets:itm:patterns} collapse to two patterns $132$ and $312$, so the set \ref{lem:threesets:itm:patterns} is the set of maximal chains in the lattice of $132,312$-avoiding permutations, as a sublattice of the weak order on the symmetric group. For example, there are exactly 12 maximal chains in this sublattice of the weak order on $A_4$. 
\end{remark}

\section{Background: polytopes}
\label{sec:backgroud.polytopes}

In this section, we provide background on order polytopes, $c$-Birkhoff polytopes, and unimodular equivalence of polytopes.
For more general treatment on polytopes, see~\cite{Z}.

\subsection{Order polytopes}\label{sec:order polytope}

We define an order polytope following \cite{Sta86}; however, it will be most convenient for us to use an opposite convention.

\begin{definition}
We define the \emph{order polytope} of a poset $P$ on the set $[n] = \{1,2,\dots,n\}$ with inequality denoted $\HeapLEQ$ by \[\ord(P) = \{\mathbf{x} \in \mathbb{R}^{n} \mid  0 \leq x_i \leq 1 \text{ for all } i \in [n]\text{ and } x_i \geq x_j \text{ whenever } i \HeapLEQ j\}\]
\end{definition}

The following facts about $\ord(P)$ will be useful for us later.

\begin{theorem}[\cite{Sta86}]\label{thm:BasicFactsOrderPolytope}
Let $P$ be a poset on the set $[n]$.
\begin{enumerate}[(1)]
\item (Section 1) The polytope $\ord(P)$ is $n$-dimensional.
\item (Corollary 1.3) The vertices of $\ord(P)$ are in bijection with order ideals of $I$. In particular, for each order ideal $I$, the corresponding vertex is $\sum_{i \in I} \mathbf{e}_i$. Note that $\sum_{i \in I} \mathbf{e}_i$ is the indicator function of~$I$.
\item \label{itm:thm:BasicFactsOrderPolytope:volume}
(Corollary 4.2) The volume of $\ord(P)$ is given by $e(P)/n!$ where $e(P)$ is the number of linear extensions of $P$.
\end{enumerate}
\end{theorem}

\begin{example}
Let $P$ be the poset whose Hasse diagram is drawn below.

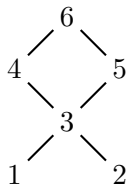
\begin{center}
\begin{tikzpicture}[scale=0.7]
\node (1) at (0,0) {$\posetelt{1}$}; 
\node (2) at (2,0) {$\posetelt{2}$}; 
\node (3) at (1,1) {$\posetelt{3}$}; 
\node (4) at (0,2) {$\posetelt{4}$}; 
\node (5) at (2,2) {$\posetelt{5}$}; 
\node (6) at (1,3) {$\posetelt{6}$}; 

\draw [
black, thick, shorten <=-2pt, shorten >=-2pt] (1) -- (3);
\draw [
black, thick, shorten <=-2pt, shorten >=-2pt] (3) -- (2);
\draw [
black, thick, shorten <=-2pt, shorten >=-2pt] (3) -- (4);
\draw [
black, thick, shorten <=-2pt, shorten >=-2pt] (3) -- (5);
\draw [
black, thick, shorten <=-2pt, shorten >=-2pt] (6) -- (5);
\draw [
black, thick, shorten <=-2pt, shorten >=-2pt] (6) -- (4);
\end{tikzpicture}
\end{center}

The order polytope $\ord(P)$ is all points of the form $(x_1,\ldots,x_6)$ in $\mathbb{R}^6$ where $0 \leq x_i \leq 1$ for all $x_i$ and coordinates have the following relations:
\[
x_1 \geq x_3, \quad x_2 \geq x_3, \quad  x_3 \geq x_4, \quad x_3 \geq x_5, \quad x_4 \geq x_6, \quad x_5 \geq x_6\
\]

The vertices of $\ord(P)$ are listed below. 
\begin{multicols}{3}
\begin{itemize}
    \item $(0,0,0,0,0,0)$
    \item $(1,0,0,0,0,0)$
    \item $(0,1,0,0,0,0)$
    \item $(1,1,0,0,0,0)$
    \item $(1,1,1,0,0,0)$
    \item $(1,1,1,1,0,0)$
    \item $(1,1,1,0,1,0)$
    \item $(1,1,1,1,1,0)$
    \item $(1,1,1,1,1,1)$
\end{itemize}
\end{multicols}

There are four linear extensions of $P$ so by Stanley's result the volume of $\ord(P)$ is $4/6!$. 

\end{example}

\begin{remark}\label{rmk:vertex-bij}
Recall that Proposition~\ref{prop:vertex-bijection} defined a poset isomorphism between the lattice of $c$-singletons and the lattice of order ideals of $\Heapwnot$. As a consequence, the $c$-singletons are in bijection with the vertices of $\ord(\heap(\sort(w_0)))$.
\end{remark}

\subsection{$c$-Birkhoff polytopes}
\label{subsec.cbir}

The Birkhoff polytope is the set of doubly stochastic matrices in $\mathbb{R}^{(n+1)\times (n+1)}$, that is, \[
\left\{X \in \mathbb{R}^{(n+1)\times (n+1)} \mid x(i,j) \geq 0, \sum_{k=1}^{n+1}x(i,k) = 1, \sum_{k=1}^{n+1}x(k,j) = 1 \text{ for all } i,j \in [n+1]\right\}.
\]

Given $w \in S_{n+1}$, let $X(w)$ be the corresponding permutation matrix. 
Specifically, let $X(w)$ be the matrix with $1$'s in row $i$ and column $w(i)$ for all $i \in [n+1]$ and $0$'s everywhere else.
The Birkhoff polytope can be defined equivalently as the convex hull of all permutation matrices \cite{birkhoff1946three}.
We define a subpolytope of the Birkhoff poytope coming from the $c$-singletons.

\begin{definition}\label{def:cBirkhoff}
The \emph{$c$-Birkhoff polytope}, denoted $\bir(c)$, is the convex hull of $$\{X(w)\ \vert\ w \text{ is a }c\text{-singleton}\}.$$
We also define $\aff(c)$ to be the affine span of $\{X(w)\ \vert\ w\text{ is a }c\text{-singleton}\}$.
\end{definition}

\begin{example}\label{ex:cBirk132}
If $c = s_1s_3s_2$, $\bir(c)$ is the convex hull of the following 9 points in $\mathbb{R}^{16}$. For each $c$-singleton $w$, we list $w$ beneath $X(w)$.

\begin{center}
\begin{tabular}{ccccc}
$\begin{pmatrix}
1&0&0&0\\
0&1&0&0\\
0&0&1&0\\
0&0&0&1\\
\end{pmatrix}$&
$\begin{pmatrix}
0&1&0&0\\
1&0&0&0\\
0&0&1&0\\
0&0&0&1\\
\end{pmatrix}$&
$\begin{pmatrix}
1&0&0&0\\
0&1&0&0\\
0&0&0&1\\
0&0&1&0\\
\end{pmatrix}$&
$\begin{pmatrix}
0&1&0&0\\
1&0&0&0\\
0&0&0&1\\
0&0&1&0\\
\end{pmatrix}$&
$\begin{pmatrix}
0&1&0&0\\
0&0&0&1\\
1&0&0&0\\
0&0&1&0\\
\end{pmatrix}$\\
Id & $s_1$ & $s_3$&$s_1s_3$&$s_1s_3s_2$\\
$\begin{pmatrix}
0&0&0&1\\
0&1&0&0\\
1&0&0&0\\
0&0&1&0\\
\end{pmatrix}$&
$\begin{pmatrix}
0&1&0&0\\
0&0&0&1\\
0&0&1&0\\
1&0&0&0\\
\end{pmatrix}$&
$\begin{pmatrix}
0&0&0&1\\
0&1&0&0\\
0&0&1&0\\
1&0&0&0\\
\end{pmatrix}$&
$\begin{pmatrix}
0&0&0&1\\
0&0&1&0\\
0&1&0&0\\
1&0&0&0\\
\end{pmatrix}$&\\
$s_1s_3s_2s_1$ & $s_1s_3s_2s_3$ & $s_1s_3s_2s_1s_3$ &  $s_1s_3s_2s_1s_3s_2$ & \\
\end{tabular}
\end{center}

\end{example}

\begin{remark}\label{rem:vertices of birk c}
Note the Birkhoff polytope does not have any interior lattice points, and so neither do its subpolytopes. This means the vertices of $\bir(c)$ are exactly the permutation matrices $\{ X(w) \ \vert\ w \text{ is  a }c\text{-singleton}\}$.
\end{remark}

\subsection{Unimodular equivalence of polytopes} 
\label{subsec.unimodular}

Given a polytope $P$, let $\aff(P)$ denote the affine span of $P$. 
For two integral polytopes $P$ in $\mathbb{R}^m$ and $Q$ in $\mathbb{R}^n$,
we say that $P$ and $Q$ are \emph{unimodularly equivalent} if there exists an affine  transformation $T\colon \mathbb{R}^m \to \mathbb{R}^n$ whose restriction to $P$ is a bijection $P \to Q$ and whose restriction to
$\aff(P) \cap \mathbb{Z}^m$ is a bijection 
$\aff(P) \cap \mathbb{Z}^m\to \aff(Q)\cap \mathbb{Z}^n$.
We refer to such a map $T$ as a \emph{unimodular transformation on $P$}. 
In particular, when $m = n$, two integral polytopes $P$ and $Q$ are unimodularly equivalent if and only if we can find an affine map  with determinant $\pm1$ that sends $P$ to $Q$.

Our main result given in Theorem~\ref{thm:main} is that the $c$-Birkhoff polytope $\bir(c)$ is unimodularly equivalent to the order polytope $\ord(\heap(\sort(w_0)))$ of $\heap(\sort(w_0))$. 

Polytopes that are unimodularly equivalent share many similar properties, including the same volume and Ehrhart polynomial. 
We also remark that unimodular equivalence is often referred to as integral equivalence in the literature
(for example, see \cite{MMS19,AStoryOfFlowAndOrderPolytopes19}).

\section{Relations on $\aff(c)$}\label{Sec:Relations}

The goal of this section is to describe linear relations on elements of $\aff(c)$, which will later be useful to describe a projection on $\aff(c)$. We will begin with a set of relations which also holds in the affine span of all permutation matrices. 

\begin{lemma}[Row and column relations]\label{lem:RowAndCol}
Let $X$ be a point in $\aff(c)$. Then, for each $i \in [n+1]$, we have \[
\sum_{j=1}^{n+1} X(i,j) = 1
\]
and for each $j \in [n+1]$, we have \[
\sum_{i=1}^{n+1}X(i,j) = 1.
\]
\end{lemma}

\begin{proof}
These relations hold for each $X(w), w \in A_n$, so in particular they hold for the $c$-singletons as well. It then follows that the relations will hold for any point in $\aff(c)$.
\end{proof}

We will now produce further relations which come from the  pattern-avoidance characterization of $c$-singletons. We will do this by exhibiting properties of $c$-singletons and then translating these into relations on $\aff(c)$.

Recall from Section \ref{subsec:PatternAvoidanceCsingletons} that, from $c$, we partition of $\{2,\ldots,n\}$ into lower-barred numbers $\underline{[2,n]} = \{d_1 < \cdots <d_r\}$ 
and upper-barred numbers $\overline{[2,n]} = \{u_1 <\cdots <u_s\}$. 
In general, we define $\overline{[a,b]}:=\overline{[2,n]}\cap[a,b]$ and $\underline{[a,b]}:= \underline{[2,n]}\cap[a,b]$.

\begin{lemma}\label{lem:zero_relations_on_permutation}
Let $w$ be a $c$-singleton.  Given an upper-barred number $u$, if $1 \leq i \leq \min(u-1,n+1-u)$, then we cannot have $w(i) = u$. Given a lower-barred number $d$, if $\max(d+1,n+3-d) \leq i \leq n+1$, then we cannot have $w(i) = d$.  
\end{lemma}

\begin{proof}
The statement follows immediately from Corollary \ref{cor:PatternAvoidanceRephrased}. 
\end{proof}

\begin{example}\label{ex:zero_relations1}
We demonstrate  Lemma~\ref{lem:zero_relations_on_permutation} for some specific Coxeter elements.
\begin{enumerate}[(a)]
\item\label{ex:zero_relations:itm:a} If $c =  s_1s_3s_2$, then $2$ is a lower-barred number. Since $\max(2+1,6-2) = 4$, we see that we cannot have $w(4) = 2$. Indeed, if the one-line notation of a permutation in $S_4=A_3$ ends with $2$, then since both 1 and 3 will appear before $2$, we cannot avoid both $13\underline{2}$ and $31\underline{2}$. Similarly, since $3$ is an upper-barred number, a $c$-singleton $w$ will never have $w(1)  = 3$.    
    \item If $c = s_1s_2 \cdots s_n$, then for all $i > \frac{n+2}{2}$, we will never have $w(i) = j$ where $i > j$ and $i+j > n+2$. The fact that all numbers in $[2,n]$ are lower-barred for this choice of $c$ explains why all of our restrictions concern $w(i)$ for $i > \frac{n+2}{2}$. 
\end{enumerate}
\end{example}

Lemma~\ref{lem:zero_relations_on_permutation} extends to the following statement about $\aff(c)$.

\begin{proposition}[Zero relations]\label{prop:ZeroRelationsOnMatrix}
Let $X$ be a point in $\aff(c)$.  Then we have the following.
 \begin{itemize} 
 \item For each upper-barred number $u$, we have $X(i,u) = 0$ for all $1 \leq i \leq \min(u-1,n+1-u)$.
 \item For each lower-barred number $d$, we have $X(i,d) = 0$ for all $\max(d+1,n+3-d) \leq i \leq n+1$.
\end{itemize}
\end{proposition}

\begin{proof}
The statement is true for all $X(w)$ where $w$ is a $c$-singleton by  Lemma~\ref{lem:zero_relations_on_permutation}. It follows that it is true for all points in $\aff(c)$, the affine span of these $X(w)$'s.
\end{proof}

\begin{example}
Consider $c=s_1s_3s_2$ as in Example~\ref{ex:cBirk132} and Example~\ref{ex:zero_relations1}\ref{ex:zero_relations:itm:a}.  Since 2 is a lower-barred number, the above proposition tells us that $X(4,2)=0$ for any point $X\in \aff(c)$.  Similarly, since 3 is an upper-barred number, $X(1,3) = 0$ for any point $X\in \aff(c)$.  Note that for any matrix $X(w)$ in Example~\ref{ex:cBirk132} those two entries are always 0. Moreover, for any other pair $(i,j)$, we can find a matrix $X(w)$ in Example~\ref{ex:cBirk132} such that $X(w)(i,j) = 1$.
\end{example}

\begin{example}\label{ex:zero_relations}
Let $c=\rw{1 432 65 7} \in A_7$, 
with $\underline{\lowerbarletter_1}, \underline{\lowerbarletter_2}, \underline{\lowerbarletter_3} = \underline{2}, \underline{5}, \underline{7}$ and $\overline{\upperbarletter_1} , \overline{\upperbarletter_2}, \overline{\upperbarletter_3} = \overline{3}, \overline{4}, \overline{6}$.    
In Figure~\ref{fig.zero_relations} (left), we cross out the entries guaranteed to be $0$ by 
Proposition~\ref{prop:ZeroRelationsOnMatrix} in the permutation matrix of a $c$-singleton for $c = [1432657]$.
Figure 2 (right) shows the permutation matrix for $s_1 s_4 s_3 s_2$, a $c$-singleton for this $c$. Note that the crossed out entries on the left are all $0$ on the right.
\end{example}

\begin{figure}[h]
\begin{minipage}{0.45\textwidth}
\begin{tabular}{|c|c|c|c|c|c|c|c|}
\hline
  \text{      } &  & $\X$  & $\X$ &  & $\X$ &  &  \text{      } \\ \hline
  &    & $\X$  & $\X$ &  & $\X$ & &  \\ \hline
  &    &    & $\X$ &  &  &  &  \\ \hline
  &    &    &   &  &  &  &  \\ \hline
  &    &    &   &    &  &  &  \\ \hline
  &   &    &   &  $\X$  &   &  &  \\ \hline
 &  & &   &  $\X$  &   &    &  \\ \hline
 & $\X$  &  &  &  $\X$  &   & $\X$   &    \\ \hline
\end{tabular}
\end{minipage}\hfill
\begin{minipage}{0.45\textwidth}
\begin{tabular}{|c|c|c|c|c|c|c|c|}
\hline
$0$ & ${1}$ & $0$ & $0$ & ${0}$ & $0$ & ${0}$ & ${0}$\\ \hline
$0$ & $0$ & $0$ & $0$ & ${1}$ & $0$ & ${0}$ & ${0}$ \\ \hline
$1$ & $0$ & $0$ & $0$ & ${0}$ & ${0}$ & ${0}$ & ${0}$ \\ \hline
$0$ & $0$ & $1$ & $0$ & ${0}$ & ${0}$& ${0}$ & ${0}$ \\ \hline
$0$ & $0$ & $0$ & $1$ & $0$ & ${0}$ & ${0}$ & ${0}$ \\ \hline
$0$ & ${0}$ & $0$ & $0$ & $0$ & $1$ & ${0}$ & ${0}$\\ \hline
${0}$ & ${0}$ & ${0}$& $0$ & $0$ & $0$ & $1$ & ${0}$\\ \hline
${0}$ & $0$ & ${0}$ & ${0}$ & $0$ & $0$ & $0$ & $1$ \\ \hline
\end{tabular}
\end{minipage}
\caption{Left: 
The zero relations of  Proposition~\ref{prop:ZeroRelationsOnMatrix} in the permutation matrix of a $c$-singleton for $c = [1432657]$. 
Right: The permutation matrix for $s_1 s_4 s_3 s_2$, a $c$-singleton for this $c$.
}\label{fig.zero_relations}
\end{figure}

Let $\overline{\upsilon}_c$ denote the cardinality of the upper-barred numbers in 
$\overline{[2, \frac{n+1}{2}}]$. 
We define $\nulabel_c: [1,n+1] \to [-\overline{\upsilon}_c, n-\overline{\upsilon}_c]$ as the bijection given in the table below. 

\begin{center}
\begin{tabular}{*{14}{|c}|}
\hline
$i$ & $u_{\overline{\upsilon}_c}$ & $u_{\overline{\upsilon}_c-1}$  & \dots & $u_1$ & $d_0$ & $d_1$ & \dots & $d_r$ & $d_{r+1}$ & $u_s$ & $u_{s-1}$ & \dots & $u_{\overline{\upsilon}_c+1}$\\\hline
$\nulabel_c(i)$& $-\overline{\upsilon}_c$ & $1-\overline{\upsilon}_c$ &\dots&-1&0&1&\dots&$r$&$r+1$&$r+2$&$r+3$ & \dots & $n-\overline{\upsilon}_c$\\ \hline
\end{tabular}
\end{center}

Note that if $j > \overline{\upsilon}_c$, then $\nulabel_c(u_j) = n+1-j = s+r+2 - j$. These are equivalent since $s+r = n-1$.

\begin{example}\label{ex:nulabel}
$ $
\begin{enumerate}[(1)]
\item\label{ex:nulabel:itm1} First, let $c = \rw{1432657}$, so that $\underline{[2,7]} = \{2,5,7\}$, $\overline{[2,7]} = \{3,4,6\}$, and $\overline{\upsilon}_c = 2$. We demonstrate the application of $\nulabel_c$ to $[1,8]$.
\begin{center}
\begin{tabular}{*{9}{|c}|}
\hline
$i$&4&3&1&2&5&7&8&6\\\hline
$\nulabel_c(i)$&-2&-1&0&1&2&3&4&5\\ \hline
\end{tabular}
\end{center}
In the usual ordering for $i$, the table for $\nulabel_c$ is the following.
\begin{center}
\begin{tabular}{*{9}{|c}|}
\hline
$i$&1&2&3&4&5&6&7&8\\\hline
$\nulabel_c(i)$&0&1&-1&-2&2&5&3&4\\ \hline
\end{tabular}
\end{center}
\item Now we consider a case with $n$ even. 
Let 
 $c = \rw{1\,4\,3\,2\,5\,7\,6\,9\,8(10)}$,   
so that $\overline{[2,10]} = \{3,4,7,9\}$ and $\underline{[2,10]} = \{2,5,6,8,10\}$. 
We see $\overline{\upsilon}_c = 2$. We demonstrate the application of $\nulabel_c$ to $[1,11]$.
\begin{center}
\begin{tabular}{*{12}{|c}|}
\hline
$i$&4&3&1&2&5&6&8&10&11&9&7\\\hline
$\nulabel_c(i)$&-2&-1&0&1&2&3&4&5&6&7&8\\ \hline
\end{tabular}
\end{center}
In the usual ordering for $i$, the table for $\nulabel_c$ is the following.
\begin{center}
\begin{tabular}{*{12}{|c}|}
\hline
$i$&1&2&3&4&5&6&7&8&9&10&11\\\hline
$\nulabel_c(i)$&0&1&-1&-2&2&3&8&4&7&5&6\\ \hline
\end{tabular}
\end{center}
\end{enumerate}
\end{example}

We begin by collecting some small, technical results which will aid our proof of the next family of relations. 

\begin{lemma}\label{lem:content_of_interval}
Given $1 \leq i \leq r$, the interval $(1,d_i)$, as a set, is equal to $\{d_1,\ldots,d_{i-1}\} \cup \{u_1,\ldots,u_{d_i-i-1}\}$ and the interval $(d_i,n+1)$ is equal to $\{d_{i+1},\ldots,d_r\} \cup \{u_{d_i-i},\ldots,u_{s}\}$. 

Similarly, given $1 \leq i \leq s$, the interval $(1,u_i)$, as a set, is equal to $\{u_1,\ldots,u_{i-1}\} \cup \{d_1,\ldots,d_{u_i-i-1}\}$ and the interval $(u_i,n+1)$ is equal as a set to $\{u_{i+1},\ldots,u_s\} \cup \{d_{u_i-i},\ldots,d_r\}$.
\end{lemma}

\begin{proof}
The interval $(1,d_i)$ must contain the $i-1$ lower-barred numbers $d_1,\ldots,d_{i-1}$. There are $d_i-2$ total numbers in this interval, so the remaining $d_i-i-1$ must be upper-barred numbers. The proof for the other three cases is similar.
\end{proof}
 
\begin{lemma}\label{lem:how_first_columns_are_labeled}
For all $y \leq \frac{n+1}{2}$, the set $\{\nulabel_c(1),\ldots,\nulabel_c(y)\}$ is a connected interval of size $y$.
Similarly, if $y \geq \frac{n+1}{2}$, the set $\{\nulabel_c(y),\ldots,\nulabel_c(n+1)\}$ is a connected interval of size $n+2-y$.

If $\frac{n+2}{2}$ is an integer and is lower-barred, then both statements are true for $y=\frac{n+2}{2}$. 
\end{lemma}

\begin{proof}
If $y \leq \frac{n+1}{2}$, then it follows from the definition of $\nulabel_c$ that $\{\nulabel_c(1),\ldots,\nulabel_c(y)\} = [-a,b]$ where $a= \vert \overline{[1,y]} \vert$ and $b = \vert \underline{[1,y]} \vert$. Since $a+b = y-1$, the claim follows. 

Now suppose $y \geq \frac{n+1}{2}$. Then the set $\{\nu_c(y),\ldots,\nu_c(n+1)\}$ is equal to the interval $[(r+1)-b, (r+1)+a]$ where $a= \vert \overline{[y,n+1]} \vert$ and $b = \vert \underline{[y,n+1]} \vert$. Since $a+b = (n+1)-y$ here, the claim follows. 

If $n+1$ is odd and $\frac{n+2}{2}$ is a lower-barred number, then the same arguments hold for $\{\nulabel_c(1),\ldots$, $\nulabel_c(\frac{n+2}{2})\}$ and $\{\nulabel_c(\frac{n+2}{2}),\ldots$, $\nulabel_c(n+1)\}$.
\end{proof}

We are now ready to prove our second main result of this section, which shows restrictions on where a $c$-singleton can send the first $y$ numbers in $[n+1]$, for small $y$. This shows that a $c$-singleton $w$ rearranges the numbers in $[1,n+1]$ in such a way that respects the first condition in Lemma \ref{lem:how_first_columns_are_labeled}. 

\begin{lemma}\label{lem:first_entries_distinct_equiv_classes}
Let $w$ be a $c$-singleton. For $2 \leq y \leq \frac{n+1}{2}$ and any integer $z$, there is exactly one value in $\{\nulabel_c(w(1)),\ldots,\nulabel_c(w(y))\}$ which is equivalent to $z$ modulo $y$. If $\frac{n+2}{2}$ is a lower-barred number, the same statement holds for $y=\frac{n+2}{2}$.
\end{lemma}

\begin{proof}

\textbf{Part 1: For $y \leq \frac{n+1}{2}$ and any integer $z$, there are no two values in the set $\{\nulabel_c(w(1)),\ldots,\nulabel_c(w(y))\}$ equivalent to $z$ modulo $y$.}

Throughout the proof, let $1 \leq a < b \leq y$. We will assume for sake of contradiction that $\nulabel_c(w(a)) \equiv \nulabel_c(w(b)) \pmod{y}$. 
We split this into three cases based on whether $w(a)$ and $w(b)$ are lower-barred or upper-barred. 

\textbf{Case i: $w(a)$ and $w(b)$ are both lower-barred.} 

Since $\nulabel_c(d_i) = i$, if $w(a)$ and $w(b)$ are both lower-barred, then $w(a) \equiv w(b) \pmod{y}$ only if $w(a) = d_i$ and $w(b) = d_{i+ k y}$ for an integer $k \neq 0$. Since $w$ must avoid patterns $13\underline{2}$ and $31\underline{2}$, none of the lower-barred numbers between $d_i$ and $d_{i+ky}$ can appear after both $d_i$ and $d_{i+k y}$ in the one-line notation for $w$. However, this means that we must fit $(\vert k\vert y-1)+2$ numbers in the first $y$ positions of the one-line notation for $w$, which is impossible for any $k \neq 0$.  Therefore, if $w(a)$ and $w(b)$ are lower-barred, $\nulabel_c(w(a)) \not\equiv \nu_c(w(b))\pmod{y}$. 

\textbf{Case ii: $w(a)$ is lower-barred and $w(b)$ is upper-barred or vice versa.} 

Now suppose $\{w(a),w(b)\} = \{u_i,d_j\}$ such that $\nulabel_c(u_i) \equiv \nulabel_c(d_j) \pmod{y}$.  Lemma~\ref{lem:zero_relations_on_permutation} already rules out certain upper-barred numbers from appearing in $\{w(1),\ldots,w(y)\}$. In particular, we know $\min(u_i-1,n+1-u_i) < y$, so there are two cases based on whether $u_i-1$ or $n+1-u_i$ is smaller. 

Suppose first $u_i-1 \leq n+1-u_i$ and therefore $u_i \leq y$ . Since $u_i \leq y \leq \frac{n+1}{2}$, we have $\nulabel_c(u_i)=-i$. As we also have $\nulabel_c(d_j)=j$, we know $j=ky-i$ for some $k \geq 1$. By Lemma \ref{lem:how_first_columns_are_labeled}, there is exactly one number in $\{\nulabel_c(1),\ldots,\nulabel_c(y)\}$ which is equivalent to $-i$ modulo $y$. Since $u_i \leq y$, this number must be $\nulabel_c(u_i)$, implying $d_{ky-i} > y \geq u_i$. Now, by Lemma \ref{lem:content_of_interval},  there are $u_i-i-1$ lower-barred numbers in $(1,u_i)$ and $ky-i-1$ in $(1,d_{ky-i})$. Therefore,  there are $ky-u_i$ lower-barred numbers in $(u_i,d_{ky-i})$.  As before, since $w$ avoids patterns $13\underline{2}$ and $31\underline{2}$, we cannot have any of these values appear after both $u_i$ and $d_{ky-i}$ in the one-line notation of $w$.  So these numbers must also occur in the first $y$ positions of $w$.  Combining this analysis with Corollary~\ref{cor:PatternAvoidanceRephrased}, this implies we must include at least $(ky-u_i) + (u_i-1) + 2 = ky+1$ numbers in the first $y$ positions,  which is impossible for any $k \geq 1$. 
The argument in the case where $u_i > (n+1)-y$ is parallel.

\textbf{Case iii: $w(a)$ and $w(b)$ are both upper-barred.} 

Suppose $\{w(a),w(b)\} = \{u_i,u_j\}$ and without loss of generality that $i<j$.
Again, by  Lemma~\ref{lem:zero_relations_on_permutation}, we must have $u_i < y+1$ or $u_i > n+1-y$ and similarly for $u_j$. Based on the definition of $\nulabel_c$,  it must be that $u_i < y+1$ and $u_j > n+1-y$. Then, since $\nulabel_c(u_i) = -i$, we must have $\nulabel_c(u_j) = n+1-j = ky-i$ for some positive integer $k$. This means we have $n+1-ky = j-i$. 

As stated in Corollary \ref{cor:PatternAvoidanceRephrased}, any $c$-singleton will have all numbers smaller than $u_i$ or all numbers larger than $u_i$ appear before $u_i$ in $w$ and similarly for $u_j$. It is impossible for all numbers greater than $u_i$ to appear in the first $y \leq \frac{n+1}{2}$ positions and similarly for all numbers less than $u_j$. This means all values in $\{1,\ldots,u_i-1\}$ must appear in $w$ before $u_i$, and similarly all values in $\{u_j+1,\ldots,n+1\}$ must appear before $u_j$. Moreover, by Lemma \ref{lem:content_of_interval}, we can deduce that there are $(u_j-u_i) - (j-i)$ lower-barred numbers in the interval $(u_i,u_j)$. 
We have thus listed $2 + (u_i-1) + (n+1-u_j) + (u_j-u_i) - (j-i) = n+2 - (n+1-ky) = ky+1$ numbers which we would be required to include in the first $y$ positions in order to include both $u_i$ and $u_j$. This is impossible since $k$ is a positive integer.

\textbf{Part 2: For $y \leq \frac{n+1}{2}$ and $0 \leq z \leq y-1$ there is at least one value in $\{\nulabel( w(1)_,\ldots,\nulabel_c(w(y))\}$ equivalent to $z$ modulo $y$.}

There are $y$ elements in $\{\nulabel( w(1)_,\ldots,\nulabel_c(w(y))\}$.  From part 1 we know that no two of them are in the same equivalence class modulo $y$. Since the set of equivalence classes modulo $y$ has $y$ elements, by the Pigeonhole Principle, there must be exactly one value in each equivalence class.

\textbf{Part 3: If $\frac{n+2}{2}$ is a lower-barred integer, the statement holds for $y = \frac{n+2}{2}$.} Suppose $y = \frac{n+2}{2}$ is a lower-barred integer. We can apply Lemma \ref{lem:how_first_columns_are_labeled} to both $\{\nulabel_c(1),\ldots,\nulabel_c(\frac{n+2}{2})\}$ and $\{\nulabel_c(\frac{n+2}{2}),\ldots,\nulabel_c(n+1)\}$, implying that both sets contain values in distinct equivalence classes modulo $\frac{n+2}{2}$. Consequently, given $i < \frac{n+2}{2}$, there exists $j > \frac{n+2}{2}$ such that $\nulabel_c(i) \equiv \nulabel_c(j) \pmod{\frac{n+2}{2}}$. The existence of at least two numbers with equivalent values under $\nulabel_c$ means for $z = \nulabel_c(i)$, $i < \frac{n+2}{2}$ we can use the same arguments as above to show our claim. Finally, there is one special value $z = \nulabel_c(\frac{n+2}{2})$ such that there is no other $i \in [n+1], i \neq \frac{n+2}{2}$ such that $\nulabel_c(i) \equiv z \pmod{\frac{n+2}{2}}$. However, it is a consequence of  Lemma~\ref{lem:zero_relations_on_permutation} that $w^{-1}(\frac{n+2}{2}) \in [1,\frac{n+2}{2}]$ when $\frac{n+2}{2}$ is lower-barred, so our claim remains true in this case. 
\end{proof}

We note that the relations given in Lemma \ref{lem:first_entries_distinct_equiv_classes} can be phrased in a simpler manner in the case when $c = s_1s_2 \cdots s_n$ since the function $\nulabel_c$ is simply $\nulabel_c(i) = i-1$. This is stated in Corollary \ref{cor:SummingRelationsTamari}. 

We now use Lemma \ref{lem:first_entries_distinct_equiv_classes} to describe relations on the affine span of the $c$-singleton permutation matrices. 

\begin{theorem}[Top sum relations]\label{thm:first_half_relations}
If $\frac{n+2}{2}$ is a lower-barred number let $2 \leq y \leq \frac{n+2}{2}$ and otherwise let $2 \leq y < \frac{n+2}{2}$. Let $z$ be an integer.  Then for $X \in \aff(c)$, \[
\sum_{\{j:\nulabel_c(j) \equiv z\}} \ \sum_{i=1}^y X(i,j) = 1
\]
where our equivalence in the first sum is modulo $y$.
\end{theorem}

\begin{proof}
Lemma \ref{lem:first_entries_distinct_equiv_classes} tells us that the statement is true for all $X(w)$, which allows us to extend the statement to all points in $\aff(c)$. 
\end{proof}

Given a Coxeter element $c = s_{i_1}\cdots s_{i_n}$, its inverse is $c^{-1} = s_{i_n}\cdots s_{i_1}$. Note that if $c$ partitions $\{2,\ldots,n\}$ into $\{d_1,\ldots,d_r\} \sqcup \{u_1,\ldots,u_s\}$, then $c^{-1}$ corresponds to the partition $\{d'_1,\ldots,d'_s\}\sqcup \{u'_1,\ldots,u'_r\}$ where $d'_i = u_i$ and $u'_i = d_i$. 
The following lemma can be deduced from \cite[Proposition 2.1]{HLT11}.

\begin{lemma}\label{lem:temp Esther}
Let $w^{\reverse}$ denote the permutation whose one-line notation is the reverse of the one-line notation of $w$, in other words, $w^{\reverse}=w w_0$. 
We have that $w$ is a $c$-singleton if and only if $w^{\reverse}$ is a $c^{-1}$-singleton.
\end{lemma}

Notice that the permutation matrix for $w^{\reverse}$ is the result of reflecting the permutation matrix for~$w$ across a horizontal axis. We give further relations on $\aff(c)$ by looking at the pattern avoidance of $c^{-1}$-singletons.

\begin{theorem}[Bottom sum relations]\label{thm:second_half_relations}
If $\frac{n+2}{2}$ is an upper-barred number, let $2 \leq y \leq \frac{n+2}{2}$, and otherwise let $2 \leq y < \frac{n+2}{2}$.  Let $z$ be an integer.  Then for $X \in \aff(c)$,
\[
\sum_{\{j:\nulabel_{c^{-1}}(j) \equiv z\}}\  \sum_{i=n+2-y}^{n+1} X(i,j) = 1
\]
where our equivalence in the first sum is modulo $y$.
\end{theorem}

\begin{proof}
From Theorem \ref{thm:first_half_relations}, we know that for $X = X(i,j)$ in $\aff(c^{-1})$, $y \leq \frac{n+2}{2}$ with equality only allowed when $\frac{n+2}{2}$ is a lower-barred number (with respect to $c^{-1}$), we have \[
\sum_{\{j:\nulabel_{c^{-1}}(j) \equiv z\}}\  \sum_{i=1}^{y} X(i,j) = 1.
\]

By Lemma \ref{lem:temp Esther}, the vertices of $\bir(c)$ and $\bir(c^{-1})$ are related by sending $X(w) \to X(w^{\reverse})$. Therefore, if $X \in \aff(c^{-1})$ then the point $X'$ defined by $X'(i,j) = X(n+2-i,j)$ is in $\aff(c)$. The claim follows from applying the relation for $\aff(c^{-1})$ to $X'$.
\end{proof}

\begin{example}\label{ex:SummingRelations}
Consider the Coxeter element $c=\rw{1432657}$ from Example~\ref{ex:zero_relations} and Example~\ref{ex:nulabel} part \ref{ex:nulabel:itm1}. 
We begin by writing all  relations described by Theorem \ref{thm:first_half_relations} for an arbitrary point $X \in \aff(c)$. 

\[
\sum_{j\in \{1,4,5,8\}} \sum_{i=1}^2 X(i,j) = 1 \qquad \sum_{j\in \{2,3,6,7\}} \sum_{i=1}^2 X(i,j) = 1 
\]

\[
\sum_{j\in \{1,7\}} \sum_{i=1}^3 X(i,j) = 1 \qquad \sum_{j\in \{2,4,8\}} \sum_{i=1}^3 X(i,j) = 1 \qquad \sum_{j\in \{3,5,6\}} \sum_{i=1}^3 X(i,j) = 1 
\]

\[
\sum_{j\in \{1,8\}} \sum_{i=1}^4 X(i,j) = 1 \quad \sum_{j\in \{2,6\}} \sum_{i=1}^4 X(i,j) = 1 \quad  \sum_{j\in \{4,5\}} \sum_{i=1}^4 X(i,j) = 1 \quad \sum_{j\in \{3,7\}} \sum_{i=1}^4 X(i,j) = 1   
\]

Now, we will record the relations described by Theorem \ref{thm:second_half_relations}. This requires computing $\nulabel_{c^{-1}}$ for $c^{-1} = [1432657]^{-1} = [7562341]$, where the upper-barred and lower-barred numbers have swapped. We include these alongside the values for $\nulabel_c$ for comparison.

\begin{center}
\begin{tabular}{*{9}{|c}|}
\hline
$i$&1&2&3&4&5&6&7&8\\\hline
$\nulabel_c(i)$&0&1&-1&-2&2&5&3&4\\\hline 
$\nulabel_{c^{-1}}(i)$&0&-1&1&2&6&3&5&4\\
\hline
\end{tabular}
\end{center}

We are now ready to compute the relations from  Theorem \ref{thm:second_half_relations}. When $y = 4$, these are already implied by the relations from Theorem \ref{thm:first_half_relations} for $y =4$ plus the fact that all columns will sum to 1, so we omit these. 

\[
\sum_{j \in \{1,5,6\}} \sum_{i = 6}^8 X(i,j) = 1 \qquad \sum_{j \in \{2,4,7\}} \sum_{i=6}^8 X(i,j) = 1 \qquad \sum_{j \in \{3,8\}} \sum_{i=6}^8 X(i,j) = 1 
\]

\[
\sum_{j \in \{1,4,5,8\}} \sum_{i=7}^8 X(i,j) = 1 \qquad \sum_{j \in \{2,3,6,7\}} \sum_{i=7}^8 X(i,j) = 1
\]

\end{example}

We now discuss the independence of the relations exhibited in this section.

\begin{proposition}\label{prop:IndependenceOfRelations}
The following set of relations on all points $X \in \aff(c)$ are linearly independent.
\begin{itemize}
    \item The relations from Lemma~\ref{lem:RowAndCol} for $i \in [n+1]$ and $j \in [2,n+1]$.
    \item All relations from Proposition~\ref{prop:ZeroRelationsOnMatrix}.
    \item The relations from Theorem~\ref{thm:first_half_relations} when $2 \leq y \leq \frac{n+2}{2}$ and $z=\nulabel_c(x)$ for $x\in[2,y]$.
    \item The relations from Theorem~\ref{thm:second_half_relations} when $2 \leq y \leq \frac{n-1}{2}$ and $z=\nulabel_{c^{-1}}(x)$ for $x\in[2,y]$.
\end{itemize}
\end{proposition}

\begin{proof}
We first suppose $n+1$ is even. 

We will regard each $X \in \aff(c)$ as a vector in $\mathbb{R}^{(n+1)^2}$ by reading the entries in the following order.
\begin{itemize}
    \item Label the entries in the first column, from row 1 to row $n+1$.
    \item Label the entries in row $\frac{n+3}{2}$, from column 2 to column $n+1$.
    \item Beginning with $y = \frac{n+1}{2}$ and working backwards to $y = 2$, label the entries $(y,2),(y,3),\ldots,(y,y)$.
    \item Beginning with $y = \frac{n-1}{2}$ and working backwards to $y = 2$, label the entries $(n+2-y,2),(n+2-y,3),\ldots,(n+2-y,y)$.
    \item Label the entries guaranteed to be 0 from Proposition~\ref{prop:ZeroRelationsOnMatrix} in lexicographic order with respect to their index $(i,j)$.
    \item Label all remaining entries in lexicographic order with respect to their index $(i,j)$.
\end{itemize}

We can regard each relation as an equation $\mathbf{v} \cdot X = 0$ or $\mathbf{v} \cdot X = 1$ for some $\mathbf{v} \in \mathbb{R}^{(n+1)^2}$. We will show these relations are linearly independent by showing the associated vectors are linearly independent. To see this, we create a matrix with these relation vectors as rows in the following order:
\begin{itemize}
\item Write the row relation vectors, in order from row 1 to row $n+1$.
\item Write the column relation vectors, in order from column 2 to column $n+1$.
\item Write the vectors for the relations from Theorem~\ref{thm:first_half_relations}, working backwards from $y = \frac{n+1}{2}$ to $y = 2$, and for a fixed $y$, working from $z = \nulabel_c(2)$ to $ z = \nulabel_c(y)$.
\item Write the vectors for the relations from Theorem~\ref{thm:second_half_relations}, working backwards from $y = \frac{n-1}{2}$ to $y = 2$, and for a fixed $y$, working from $z = \nulabel_{c^{-1}}(2)$ to $ z = \nulabel_{c^{-1}}(y)$.
\item Write the vectors for the zero relations from Proposition~\ref{prop:ZeroRelationsOnMatrix} in lexicographic order with respect to their index $(i,j)$.
\end{itemize}

By construction, this matrix is in row echelon form, so it is full rank. Thus, this set of relations is independent.

If $n+1$ is odd, then we can use a similar proof but we instead consider $2 \leq y \leq \frac{n}{2}$ for the top and bottom sum relations.
\end{proof}

\begin{example}
We give an example of the ordering on the entries of $X \in \aff(c)$ from the previous proof for $c=\rw{1432657}$.

\begin{center}
\begin{tabular}{|c|c|c|c|c|c|c|c|}
\hline
1 & 37 & 25 & 26 & 38 & 27 & 39 & 40\\ \hline
 2 & 21 & 28  & 29 & 41 & 30 & 42 & 43 \\ \hline
 3 &  19  &  20  & 31 &44  & 45 & 46 & 47 \\ \hline
 4 &  16  &  17  &  18 & 48 & 49 & 50 &  51\\ \hline
 5 &  9  & 10   & 11  &  12  & 13 & 14 & 15 \\ \hline
 6 &  22 &  23  & 52  & 32  &  53 & 54 & 55 \\ \hline
7 & 24 & 56  & 57 &  33  &  58 &  59 & 60\\ \hline
8 & 34  & 61 & 62  &  35  &  63  & 36 &  64  \\ \hline
\end{tabular}
\end{center}

\end{example}

Using Proposition~\ref{prop:IndependenceOfRelations}, we can give an upper bound on the dimension of $\aff(c)$. 

\begin{corollary}\label{cor:Affc has dimension at most n+1 choose 2}
$\aff(c)$ has dimension at most $\binom{n+1}{2}$.    
\end{corollary}
\begin{proof}
To prove this statement, we compute the number of relations listed in Proposition~\ref{prop:IndependenceOfRelations}. The number of zero relations is $(n-1) + (n-3) + \cdots$. It is easiest to enumerate the top and bottom sum relations by counting how many hold for a fixed $z$. For instance, when $z = 1$, we get $n-2$ relations and when $z = 2$, we get $n-4$ relations. That is, the total number of top and bottom sum relations is $(n-2) + (n-4) + \cdots$. 
In particular, adding this to our count of the zero relations yields ${n\choose 2}$ relations. 
Then the claim holds since $(n+1)^2 - (2n+1) - {n\choose 2} = {n+1 \choose 2}$.
\end{proof}

\section{A lattice-preserving projection}\label{sec.lattice}

In the previous section, we developed an understanding of the subspace of $\mathbb{R}^{(n+1)^2}$ containing $\aff(c)$.
We now use this description to provide a projection from $\mathbb{R}^{(n+1)^2}$ to $\mathbb{R}^{{n+1 \choose 2}}$ whose restriction to $\aff(c)$ is injective and preserves the integer lattice.

Consider a Coxeter element $c=(\upperbarletter_s ~ ... ~ \upperbarletter_1 ~  1 ~ \lowerbarletter_1 ~ ...~ \lowerbarletter_r ~ (n+1))=
(1~  \lowerbarletter_1 ~ ...~  \lowerbarletter_r ~ (n+1) ~ \upperbarletter_s ~ ... ~ \upperbarletter_1)$.
Let $\sigma_c$ be the permutation 
$(n+1) \, {d_r} \, {d_{r-1}} \dots \, {d_1} \, 1 \, {u_1} {u_2} \,\dots \, {u_s}$ written in one-line notation.  When $c$ is understood, we will just write $\sigma$. Note that $\sigma_c$ is simply the result of dropping the parentheses in the cycle-notation for $c^{-1}= ((n+1) ~ \lowerbarletter_r ~ ... ~\lowerbarletter_1 ~ 1 ~ \upperbarletter_1 ~ ... ~ \upperbarletter_s)$ and regarding this as a permutation in one-line notation. We also let $\underline{\upsilon}_c = \vert \underline{[2,\frac{n+1}{2} ]} \vert$, analogous to the definition of $\overline{\upsilon}_c$ in the previous section.

\begin{lemma}\label{lem:LabelSigma}
For a Coxeter element $c$, 
\[
\nulabel_{c^{-1}}(\sigma_c(i)) = \begin{cases} s+i & 1\leq i \leq r-\underline{\upsilon}_c+1\\ i - (r+2) & r-\underline{\upsilon}_c+1<i\leq n+1
\end{cases}
\]
\end{lemma} 

\begin{proof}
This follows immediately from the definition of $\nulabel_c$. Note that $\sigma_c(2),\ldots,\sigma_c(r-\underline{\upsilon}_c+1)$ correspond exactly to the upper-barred numbers with respect to $c^{-1}$ which are in $[ \frac{n+2}{2}, n+1]$.
\end{proof}

\begin{example} Recall we computed $\nu_{c^{-1}}$ for $c = [1432657]$ in Example \ref{ex:SummingRelations}. We reorder the entries according to $\sigma$ to see an instance of the above result. Notice that here $r = s = 3$ and there is one lower-barred number less than $\frac{n+1}{2} = 4$, so $\underline{\upsilon}_c = 1$. 

\begin{center}
\begin{tabular}{*{9}{|c}|}
\hline
$i$&1&2&3&4&5&6&7&8\\\hline
$\sigma(i)$&8&7&5&2&1&3&4&6\\\hline 
$\nulabel_{c^{-1}}(\sigma(i))$&4&5&6&-1&0&1&2&3\\
\hline
\end{tabular}
\end{center}
\end{example}

\begin{definition}\label{defn.projection}
We define a projection $\Pi_c$ from the space of $(n+1) \times (n+1)$ matrices with entries in $\mathbb{R}$ to $\mathbb{R}^{n+1 \choose 2}$ which will be used to compare the $c$-Birkhoff polytope and the order polytope of the corresponding heap. 
To define $\Pi_c$, the first entries  we will read are from columns $\lowerbarletter_1,\dots, \lowerbarletter_r$: 

\begin{align*}
&({\lowerbarletter_1}-1 , {\lowerbarletter_1}), ({\lowerbarletter_1} - 2 , {\lowerbarletter_1}), \dots, (1, {\lowerbarletter_1}),\\
& \dots \\
    &({\lowerbarletter_r}-1,{\lowerbarletter_r}), ({\lowerbarletter_r} - 2 , {d_r}), \dots, (1, {\lowerbarletter_r}), \\
    & (n,n+1), (n-1,n+1), \dots, (1,n+1).
\end{align*}
The remaining entries are determined using the upper-barred integers ${\upperbarletter_s}, \dots, {\upperbarletter_1}$, in decreasing order. 
For each upper-barred integer ${\upperbarletter}$,
 we take ${\upperbarletter}-1$ entries from our matrices as follows: 
\begin{enumerate}
\item Let $m= \mu(u):= \min({\upperbarletter} - 1, n+1- {\upperbarletter})$.   
\item First take the $m$ entries $(n+1,c^1(\upperbarletter)), (n,c^2(\upperbarletter)), \dots, (n+2-m, c^m(\upperbarletter))$. Note that we are thinking of $c$ as a permutation here so $c^k$ is the permutation applied $k$ times.
\item Then, if $\upperbarletter > \frac{n+2}{2}$, take the
additional ${\upperbarletter} - 1 - m$ entries  
 $({\upperbarletter} -1 , {\upperbarletter}), ({\upperbarletter} - 2, {\upperbarletter}), \dots, (m+1, {\upperbarletter})$. 

\end{enumerate}
After we are done reading the entries, we reverse them.
\end{definition}

In the rest of the paper, we will call the entries of $\Pi_c$ that we obtained from the enumerated list above with $u=u_t$ entries \emph{associated to} $u_t$.

\begin{example}\label{eg.proj} Continuing with Example \ref{ex:SummingRelations} for $c = [1432657] = (1{\color{blue}257}8{\color{red}643})$, we compute the projection $\Pi_c$ in Figure~\ref{fig.proj} (left).

\begin{figure}[htbp]
\begin{center}
\begin{minipage}{0.45\textwidth}
\begin{tabular}{|c|c|c|c|c|c|c|c|}
\hline
  & $28$ & $\X$  & $\X$ & $24$ & $\X$ & $18$ & $11$ \\ \hline
  &    & $\X$  & $\X$ & $25$ & $\X$ & $19$ & $12$ \\ \hline
  &    &    & $\X$ & $26$ & $6$ & $20$ & $13$ \\ \hline
  &    &    &   & $27$ & $7$ & $21$ & $14$ \\ \hline
  &    &    &   &    & $8$ & $22$ & $15$ \\ \hline
  & $3$  &    &   &  $\X$  &   & $23$ & $16$ \\ \hline
$4$ & $1$  & $9$ &   &  $\X$  &   &    & $17$ \\ \hline
$2$ & $\X$  & $5$  & $10$ &  $\X$  &   & $\X$   &    \\ \hline
\end{tabular}
\end{minipage}\hfill
\begin{minipage}{0.45\textwidth}
\begin{tabular}{|c|c|c|c|c|c|c|c|}
\hline
$0$ & $\Circled{1}$ & $0$ & $0$ & $\Circled{0}$ & $0$ & $\Circled{0}$ & $\Circled{0}$\\ \hline
$0$ & $0$ & $0$ & $0$ & $\Circled{1}$ & $0$ & $\Circled{0}$ & $\Circled{0}$ \\ \hline
$1$ & $0$ & $0$ & $0$ & $\Circled{0}$ & $\Circled{0}$ & $\Circled{0}$ & $\Circled{0}$ \\ \hline
$0$ & $0$ & $1$ & $0$ & $\Circled{0}$ & $\Circled{0}$& $\Circled{0}$ & $\Circled{0}$ \\ \hline
$0$ & $0$ & $0$ & $1$ & $0$ & $\Circled{0}$ & $\Circled{0}$ & $\Circled{0}$ \\ \hline
$0$ & $\Circled{0}$ & $0$ & $0$ & $0$ & $1$ & $\Circled{0}$ & $\Circled{0}$\\ \hline
$\Circled{0}$ & $\Circled{0}$ & $\Circled{0}$& $0$ & $0$ & $0$ & $1$ & $\Circled{0}$\\ \hline
$\Circled{0}$ & $0$ & $\Circled{0}$ & $\Circled{0}$ & $0$ & $0$ & $0$ & $1$ \\ \hline
\end{tabular}
\end{minipage}
\caption{Left: We depict the projection $\Pi_c$ for $c = [1432657]$. We also place red \textcolor{red}{X}'s in the entries which are guaranteed to be zero by  Proposition~\ref{prop:ZeroRelationsOnMatrix}. Right: We draw the permutation matrix for $s_1s_4s_3s_2$ and circle the entries which are recorded by $\Pi_c$.}\label{fig.proj}
\end{center}
\end{figure}
\end{example}

\begin{prop}\label{prop.sw}
For each upper-barred integer $u$ and each $k \leq \mu(u)$, the entry  $(n+2-k,c^k(u))$ is strictly below the main diagonal.
\end{prop}

\begin{proof}
Suppose for sake of contradiction that $c^k(u) \geq n+2-k$ for some $k \leq \mu(u)$. Since $k \leq n+1-u$, this implies $c^k(u) \geq u+1$. An inspection of the cycle notation of $c = (\upperbarletter_s ~ ... ~ \upperbarletter_1 ~  1 ~ \lowerbarletter_1 ~ ...~ \lowerbarletter_r ~ (n+1))$ informs us that the set $\{u,c^1(u),c^2(u),\ldots,c^k(u)\}$ must then contain all numbers in $[u]$ as well as $c^k(u)$. In particular, the set$\{u,c^1(u),c^2(u),\ldots,c^k(u)\}$ has size at least $u+1$. It also has size $k+1$, implying  $k \geq u$.  However, we know that $k\leq\mu(u)<u$, so this is a contradiction.

\end{proof}

\begin{remark}\label{rem.ne}
All entries of $\Pi_c$ other than those listed in the above proposition are strictly above the main diagonal.
\end{remark}

We next show that none of the entries recorded by $\Pi_c$ are entries guaranteed to be zero by  Proposition~\ref{prop:ZeroRelationsOnMatrix}. One can notice this is true on the left-hand side of Figure \ref{fig.proj}. 

\begin{proposition}\label{prop:DontRecordZeroes}
There is no intersection between the entries recorded by $\Pi_c$ and those guaranteed to be zero in  Proposition~\ref{prop:ZeroRelationsOnMatrix}. 
\end{proposition}

\begin{proof}
We will prove this by showing that none of the entries described in  Proposition~\ref{prop:ZeroRelationsOnMatrix} are listed in the definition of $\Pi_c$.

First, for each upper-barred number $\upperbarletter$, the entries guaranteed to be zero are $X(i,u)$ for $1 \leq i \leq \mu(u)$. 
From Proposition \ref{prop.sw}, we know the entries $(n+2-k,c^k(u))$ for $k\leq\mu(u)$ lie below the main diagonal so they will not coincide with the entries guaranteed to be zero from upper-barred numbers. All of the other entries listed in the definition of $\Pi_c$ are in a column indexed by a lower-barred number, column $n+1$, or a row greater than $\mu(u)$. Therefore the entries $X(i,u)$ for $1 \leq i \leq \mu(u)$ are not recorded under $\Pi_c$.

The other set of entries that are always zero from  Proposition~\ref{prop:ZeroRelationsOnMatrix} are $X(i,\lowerbarletter)$ for each lower-barred number $\lowerbarletter$ and each $i$ satisfying $\max(\lowerbarletter+1,n+3-\lowerbarletter) \leq i \leq n+1$. Since these entries are all below the main diagonal, we just need to show that $(i,\lowerbarletter)$ will not be equal to $(n+2-k, c^k(u))$ for some upper-barred number $u$ and $k \leq \mu(u)$.

Consider the numbers that appear  between $d$ and $u$, inclusive, in the one-line notation of $\sigma$.  That is, the numbers in bold below:
\[
\sigma = (n+1)\ d_r \dots \mathbf{d \dots \ d_1 \ 1\  u_1 \dots u} \dots u_s\ .
\]
These values include all numbers in the interval $[1,\min(d,u)]$. In particular, if $u<d$, there are at least $u+1>u-1\geq\mu(u)$ numbers weakly between $u$ and $d$ in $\sigma$. From the definition of $\sigma$ and from writing $c$ in cycle notation we can see that $c^k(u)$ is the value $k$ places before $u$ in $\sigma$.  Therefore we must have $u > d$ in order for $c^k(u) = d$ for some $k \leq \mu(u)$.

Since $u > d$, we now know that all numbers in $[d]$ appear weakly between $u$ and $d$ in $\sigma$.  This means that, in order for $c^k(u) = d$, we need $k \geq d$. Therefore, if the pair $(n+2-k, c^k(u)) = (n+2-k,d)$ is an entry recorded by $\Pi_c$, we have $n+2-k \leq n+2-d < n+3-d$. Therefore, the entry $X(n+2-k,d)$ is not one guaranteed to be 0 for all $X \in \aff(c)$ by  Proposition~\ref{prop:ZeroRelationsOnMatrix}.  
\end{proof}

For $X \in \aff(c)$, we will refer to values $X(i,j)$ recorded by $\Pi_c$ as well as those guaranteed to be 0 by Proposition~\ref{prop:ZeroRelationsOnMatrix} as ``determined values.'' We will refer to all other values of $X$ as ``undetermined values.'' 

Our main result of this section concerns what happens when we  apply $\Pi_c$ to $\aff(c)$. We prepare for the proof of this with a technical lemma.

\begin{lemma}\label{lem:CanPartitionUnrecordedLowerRow}
If $\frac{n+2}{2}$ is upper-barred, fix $\frac{n+2}{2} \leq k < n+1$, and otherwise fix $\frac{n+2}{2} < k \leq n+1$. Let $\{a_1,\ldots,a_t\}\subseteq [n+1]$ be the set such that the entries $X(k,a_i)$ are the undetermined values of $X$.  Then, $t = n+3-k$ and $\{a_1,\ldots,a_t\} = \{k\} \cup \{\sigma(\alpha),\sigma(\alpha+1),\ldots,\sigma(\alpha + t - 2)\}$ for some positive integer $\alpha > r - \underline{\upsilon}_c + 1$.
\end{lemma}

\begin{proof}
Our goal is to understand how the values $\{a_1,\ldots,a_t\}$ appear in the cycle notation of $c$. It will be easier for us to first study the complement, $[n+1] \backslash \{a_1,\ldots,a_t\}$. To this end, we introduce three sets and show that they partition $[n+1] \backslash \{a_1,\ldots,a_t\}$.

\begin{enumerate}[(1)]
\item \label{lem:CanPartitionUnrecordedLowerRow:itm1} Let $I^k_1=(k,n+1]$.
\item \label{lem:CanPartitionUnrecordedLowerRow:itm2} Let $I^k_2=\underline{[n+3-k,k-1]}$. 
\item \label{lem:CanPartitionUnrecordedLowerRow:itm3} Let $I^k_3=\{c^{n+2-k}(u): u \in \overline{[n+3-k,k-1]}\}$ if  $k \geq \frac{n+4}{2}$.
\end{enumerate}

The numbers in $I_1^k$ index the columns $j$ of all entries $(k,j)$ recorded by $\Pi_c$ such that either $j$ is lower-barred, $j$ is upper-barred and $(k,j)$ is covered by the third case of Definition~\ref{defn.projection} of $\Pi_c$, or $j = n+1$. The numbers $I_2^k$ index the columns $j$ such that $X(k,j) = 0$ from (Proposition~\ref{prop:ZeroRelationsOnMatrix}). The numbers in $I_3^k$ index the columns $j$ of all entries $(k,j)$ recorded by $\Pi_c$ such that $j$ is upper-barred and $(k,j)$ is covered by the second case of Definition~\ref{defn.projection} of $\Pi_c$.  Thus, together $I_1^k$ and $I_3^k$ exactly index the columns of all entries $(k,j)$ which are recorded by $\Pi_c$.  Note that if $k < \frac{n+4}{2}$, then $I_2^k$ and $I_3^k$ will both be empty.

Recall that $c=(\upperbarletter_s ~ ... ~ \upperbarletter_1 ~  1 ~ \lowerbarletter_1 ~ ...~ \lowerbarletter_r ~ (n+1))=
(1~  \lowerbarletter_1 ~ ...~  \lowerbarletter_r ~ (n+1) ~ \upperbarletter_s ~ ... ~ \upperbarletter_1)$ in cycle notation.  It is clear that $I_1^k \cup \{k\}, I_2^k$, and $I_3^k$ each individually form a connected cyclic interval of $c$.
Proposition~\ref{prop.sw} shows that $I_1^k \cup \{k\}$ and $I_3^k$ are disjoint while Proposition \ref{prop:DontRecordZeroes} shows $I_2^k$ is disjoint from $I_1^k \cup I_3^k$. By definition,  $k \notin I_2^k$. Thus, $I_1^k\cup\{k\}, I_2^k$, and  $I_3^k$  are pairwise disjoint.

Now, it is clear that $I_1^k \cup \{k\}$ and $I_2^k$ are cyclically consecutive in $c$. We have $I_2^k = \underline{[n+3-k,k-1]}$ and $I_3^k$ is the result of shifting all numbers in $\overline{[n+3-k,k-1]}$ by $n+2-k$ in the ordering imposed by $c$, i.e. $I_3^k = c^{n+2-k}(\overline{[n+3-k,k-1]})$. There are exactly $n+2-k$ numbers cyclically after $\overline{[n+3-k,k-1]}$ and before $\underline{[n+3-k,k-1]}$. It is then clear that $I_2^k$ and $I_3^k$ are consecutive. Thus, the complement of $I_1^k \cup I_2^k \cup I_3^k \cup \{k\}$ must be consecutive in $c$. Since $n+1$ is not in the complement, it is easy to see that it is also consecutive in $\sigma$. 

We will now show $t=n+3-k$. We know $I_1^k, I_2^k$, and $I_3^k$ are disjoint.  Also, by their description, all determined values $X(k,j)$ satisfy $j \in I_1^k \sqcup I_2^k \sqcup I_3^k$. There are $(n-k+1) + (2k-n-3)= k - 2$ values in $I_1^k \sqcup I_2^k \sqcup I_3^k$, implying the complement in $[n+1]$ is size $n+3-k$. 

Finally, we show that none of the elements in the complement of $I_1^k \cup I_2^k \cup I_3^k \cup \{k\}$ in $[n+1]$ are lowered-barred numbers larger than $\frac{n+1}{2}$, which will show the desired bound on $\alpha$. In order for a lower-barred number larger than $\frac{n+1}{2}$ to be in the complement of $I_2^k$, we require $\frac{n+1}{2} < n+3-k$, which implies $k < \frac{n+5}{2}$.  Thus we only need to check the cases where $k=\frac{n+2}{2}$, where $k=\frac{n+3}{2}$, and where $k=\frac{n+4}{2}$.  When $k=\frac{n+2}{2}$ or $\frac{n+3}{2}$, there are no integers less than $k$ and greater than $\frac{n+1}{2}$ to check. When $k=\frac{n+4}{2}$, we have $\frac{n+2}{2}\in I_2^k$. Thus, in all cases, the bound on $\alpha$ holds.

\end{proof}

\begin{theorem}\label{thm.lattice.preserving.projection}
The restriction of the projection $\Pi_c: \mathbb{R}^{(n+1)^2} \to \mathbb{R}^{{n+1 \choose 2}}$ to 
$\aff(c)$
is an injective linear transformation which sends integral points to integral points.
\end{theorem}

\begin{proof}
Let $X\in\aff(c)$.
Our objective is to demonstrate that each undetermined value can be deduced from the determined values. 
This will imply that $X$ is uniquely determined by $\Pi_c(X)$, and further, if $\Pi_c(X)$ is integral, $X$ must have been integral. 

We first prove by induction on $k$ that for any $1 \leq k\leq \frac{n+1}{2}$, we can deduce the entire $k$\textsuperscript{th} row of~$X$. First we check this for $k=1$. The only entry that is not yet determined in this case is $X(1,1)$. By Lemma~\ref{lem:RowAndCol}, this value can be computed. Next we assume all entries in rows strictly above row $k$ are known and we will show that the entries in row $k$ can be computed. The undetermined entries in row $k$ are $X(k,j)$ for $j\in [k]$. Theorem~\ref{thm:first_half_relations} gives a linear equation
\[
\sum_{\{j:\nulabel_c(j) \equiv z\}} \ X(k,j) = 1 - \sum_{\{j:\nulabel_c(j) \equiv z\}} \sum_{i=1}^{k-1} X(i,j)
\]
for each $0\leq z\leq k-1$.
By the inductive hypothesis, we can compute the right-hand side for each of these equations.
Since $\{j:\nulabel_c(j) \equiv z\}$ are disjoint for different $z$ values, these $k$ linear equations are linearly independent and thus give unique solution to the unknowns $X(k,j)$ for $j\in [k]$. If $\frac{n+2}{2}$ is an integer and is lower-barred, then we determine entries in this row in the same way; otherwise, we stop at the largest integer less than $\frac{n+2}{2}$.

We next prove by (backwards) induction on $k$ that for any $\frac{n+2}{2} < k \leq n+1$, we can deduce the entire $k$\textsuperscript{th} row. First we check this for row $k=n+1$. The only two entries that are not determined are $X(n+1,n+1)$ and $X(n+1,j)$ where $j = u_s$ if $s\geq 1$ and $j=1$ otherwise. Note, by Lemma~\ref{lem:RowAndCol}, $X(n+1,n+1)$ can be computed as $X(1,n+1),\dots, X(n,n+1)$ are known. Then, by Lemma~\ref{lem:RowAndCol} again, $X(n+1,j)$ can be computed too. 

Next we assume all entries in rows strictly below row $k$ are known and we will show that the entries in row $k$ can be computed for some $\frac{n+2}{2} < k < n+1$.  By Lemma~\ref{lem:CanPartitionUnrecordedLowerRow}, the undetermined entries in row $k$ are $X(k,a_1),\ldots,X(k,a_t)$ where $\{a_1,\ldots,a_t\}=\{k,\sigma(\alpha),\dots,\sigma(\alpha+t-2)\}$ with $\alpha > r - \vert \underline{[2,\frac{n+1}{2}]} \vert + 1$ and $t = n+3-k$. Note that entries $X(i,k)$ are known for all $i \neq k$ by the inductive hypothesis, the definition of $\Pi_c$, and Proposition~\ref{prop:ZeroRelationsOnMatrix}. Therefore, by Lemma~\ref{lem:RowAndCol} the entry $X(k,k)$ can be computed. 

We know $k\in\{a_1,\dots,a_t\}$ by Lemma~\ref{lem:CanPartitionUnrecordedLowerRow}.  Without loss of generality, let $a_1=k$. From Lemma~\ref{lem:LabelSigma} we then have that $\{\nulabel_{c^{-1}}(a_2),\ldots,\nulabel_{c^{-1}}(a_t)\}=[\alpha-(r+2),\alpha+t-2-(r-2)]$.  Since $t = n+3-k$ and $\frac{n+1}{2}  \leq k < n+1$, we know $t\leq k$.  Thus, the values $\nulabel_{c^{-1}}(a_2),\ldots,\nulabel_{c^{-1}}(a_{t})$ are pairwise distinct values modulo $k$. Theorem~\ref{thm:second_half_relations} gives a linear equation
\[
\sum_{\{j:\nulabel_{c^{-1}}(j) \equiv z\}}\  X(k,j) = 1 - \sum_{\{j:\nulabel_{c^{-1}}(j) \equiv z\}}\  \sum_{i=k+1}^{n+1} X(i,j)
\]
for each $0\leq z\leq n+1-k$.  By the inductive hypothesis, we can compute the right-hand side for each of these equations. Since $\{j:\nulabel_{c^{-1}}(j) \equiv z\}$ are disjoint for different $z$ values, these $n+2-k$ linear equations are linearly independent and thus give a unique solution to the remaining $t = n+2-k$ unknowns. When $\frac{n+2}{2}$ is an upper-barred integer, this induction works for $k = \frac{n+2}{2}$ as well. 
\end{proof}

\section{Unimodular equivalence}\label{sec:UnimodularEquivalence}

This section will prove our main result (Theorem \ref{thm:main}) which states the two polytopes $\bir(c)$ and $\ord(\heap(\sort(w_0)))$ are unimodularly equivalent. 
Since we showed in Theorem~\ref{thm.lattice.preserving.projection} that $\Pi_c$ is injective and lattice-preserving on $\aff{(c)}$, it remains to prove there is a unimodular transformation~$\mathcal{U}_c$ from $\Pi_c(\bir(c))$ to $\ord(\heap(\sort(w_0)))$.
 
We begin in Section \ref{sec:heap} with some results concerning the structure of $\heap(\sort(w_0))$. In Section~\ref{subsec:Square}, we define a subset of $\binom{n+1}{2}$  $c$-singletons $\{b_i\mid 1\leqslant i \leqslant \ell(w_0)\}$ and show that the matrix whose columns are the projections of those elements under $\Pi_c$ is an antidiagonal lower unitriangular matrix (Lemma~\ref{lem.lower} and Lemma~\ref{lem.upper}). 
As described in Proposition \ref{prop:vertex-bijection}, each $c$-singleton is associated with an order ideal of $\heap(\sort(w_0))$ and therefore with a vertex of $\ord(\heap(\sort(w_0)))$. The vertices of $\ord(\heap(\sort(w_0)))$ corresponding to the elements of $\{b_i \}$ also form an antidiagonal lower unitriangular matrix. We then define $\mathcal{U}_c$ to be the unique linear transformation which sends the projections of $\{b_i \}$ under $\Pi_c$ to the associated subset of vertices of $\ord(\heap(\sort(w_0)))$. In Section \ref{subsec:SquareToRectangle}, we show that for \emph{every} $c$-singleton $w$,  the vector $\mathcal{U}_c \circ \Pi_c(X(w))$ is the vertex of $\ord(\heap(\sort(w_0)))$ associated with $w$.

\subsection{Construction and properties of the heap of the $c$-sorting word of $w_0$}\label{sec:heap}
This section describes the ``shape" of $\Heapwnot$. 
Let $c$ be a Coxeter element in $A_n$ with 
$r$ lower-barred numbers $1<\underline{\lowerbarletter_1}< \dots < \underline{\lowerbarletter_r} < n+1$ and 
$s$ upper-barred numbers $1< \overline{\upperbarletter_1} < \dots < \overline{\upperbarletter_s} < n+1$. 
We call a decreasing sequence of consecutive integers a \emph{decreasing run}; for example, $5432$ is a decreasing run, but $6532$ is not.
Denote by $\DiagonalReadingWord$ a word which 
is the concatenation of $n$ factors
\begin{equation}\label{eq:Def:Rc}
\DiagonalReadingWord = \rw{(\underline{\lowerbarletter_1} -1) ... 1} \dots  
\rw{(\underline{\lowerbarletter_r} -1) ... 1} 
\rw{n ... 1} 
\rw{n ... (n-\overline{\upperbarletter_s} +2)} \dots  
\rw{n ... (n- \overline{\upperbarletter_1} +2)}
\end{equation}
where each factor is a decreasing run.

Following Definition~\ref{defn:heap of a reduced word}, we can construct the heap $\HeapDiagonalReading \coloneqq \heap(\DiagonalReadingWord)$ of $\DiagonalReadingWord$ 
and its heap diagram. 
Below, we describe an equivalent algorithmic construction of the heap diagram for $\HeapDiagonalReading$. 
We will refer to a subset of a heap where all the elements form a line with slope $-1$ as a \emph{diagonal}. Similarly, we will refer to a subset of a heap where all the elements form a line with slope $1$ as an \emph{antidiagonal}.

\def\LongDiagonal{D_{\text{long}}}
\begin{algorithm}
[Algorithm for constructing the heap diagram for $\HeapDiagonalReading$]

\label{alg:algorithm for heap using diagonal reading word}
Consider the lattice $L\subseteq \mathbb{R}^2$ generated by $\epsilon_1=(1/2,1/2)$ and $\epsilon_2=(-1/2,1/2)$. Let $\Gamma$ be the grid of all points in $L$ which are of the form $x_1 \epsilon_1 + x_2 \epsilon_2$, where $x_1$ and $x_2$  are nonnegative integers satisfying $x_1+x_2 \leq 2n$ and $0 \leq x_1 - x_2 \leq n-1$. We refer to the lattice point $x_1 \epsilon_1 + y_2 \epsilon_2$ as vertex $(a,b)=(x_1-x_2+1, x_2+1)$. The grid $\Gamma$ is then a set of vertices $\{ (a,b) \mid 1 \leq a,b \leq n \}$ as in Figure \ref{fig:longest element heap A7 4321657 alternative}. 
Note that our indexing convention is such that $(a,b)$ and $(a',b')$ are in the same diagonal if and only if $a+b = a'+b'$ and they are in the same antidiagonal if and only if $b= b'$.
\begin{enumerate}[(Step 1)]
\item \label{alg:algorithm for heap using diagonal reading word:itm1}
Take the vertices along the diagonal from 
vertex $(n,1)$ to vertex $(1,n)$. Denote this as $\LongDiagonal$. 
\item \label{alg:algorithm for heap using diagonal reading word:itm2}
For each lower-barred number $\lowerbarletter_i$, we add vertices $(1, n-r+i-1),(2, n-r+i-2),\dots,(\lowerbarletter_i-1, n-r+i-\lowerbarletter_i+1)$.
That is, we add a ``flushed-left" diagonal with $\lowerbarletter_i - 1$ vertices below~$\LongDiagonal$.  
Note that immediately below $\LongDiagonal$ we have a shorter diagonal corresponding to $\lowerbarletter_r$, then an even shorter diagonal corresponding to $\lowerbarletter_{r-1}$, and so on. 

\item \label{alg:algorithm for heap using diagonal reading word:itm3}
For each upper-barred number $\upperbarletter_i$, we add vertices $(n, 2+s-i),(n-1, 2+s-i+1),\dots,(n-\upperbarletter_i+2, s-i+\upperbarletter_i)$.
That is, we add a ``flushed-right" diagonal with $\upperbarletter_i - 1$ vertices above $\LongDiagonal$.  
Similarly to the previous step, immediately above $\LongDiagonal$ we now have a shorter diagonal corresponding to $\upperbarletter_s$, then an even shorter diagonal corresponding to $\upperbarletter_{s-1}$, and so on. 
\end{enumerate} 
For each vertex $(a,b)$ we have collected, label each vertex by $a$. The resulting diagram is the heap diagram of $\HeapDiagonalReading$. 
\end{algorithm}

\begin{example}
Let $c=\rw{4321 65 7}$ with lower-barred numbers $\underline{5}, \underline{7}$ and upper-barred numbers $\overline{2}, \overline{3}, \overline{4}, \overline{6}$. Then the word $\DiagonalReadingWord$ in this case is 
\[\mathcal{R}_{\rw{4321 65 7}}=\rw{4321}\rw{654321}\rw{7654321}\rw{76543}\rw{765}\rw{76}\rw{7}.\]

The heap diagram of $\HeapDiagonalReading$ is shown in Figure \ref{fig:longest element heap A7 4321657 alternative}.  We construct this diagram using Algorithm~\ref{alg:algorithm for heap using diagonal reading word}.  
In \ref{alg:algorithm for heap using diagonal reading word:itm1} we add the diagonal from position $(7,1)$ to $(1,7)$.  We put diamonds around the positions added in this step.  In \ref{alg:algorithm for heap using diagonal reading word:itm2} we add two diagonals of sizes 4 and 6 respectively, corresponding to the lower-barred numbers $\lowerbarletter_1=\underline{5}, \lowerbarletter_2=\underline{7}$. We put circles around the positions added in this step.  Lastly, for \ref{alg:algorithm for heap using diagonal reading word:itm3} we add four diagonals of sizes $5, 3, 2, 1$, respectively, corresponding to the upper-barred integers ${\upperbarletter_4}=\overline{6}, {\upperbarletter_3}=\overline{4}, {\upperbarletter_2}=\overline{3}, {\upperbarletter_1}=\overline{2}$. We put rectangles around the positions added in this final step. 
\end{example}

\def\myscale{.64}
\begin{figure}[htb!]
\centering
\begin{tikzpicture}[scale=\myscale]
\filldraw [yellow!20]
(6+.4,14+.4) -- (3,11+.4) -- (2,12+.4) -- (1,11+.4) --
(0-.4,12+.4) --(0-.4,8-.4) -- 
(3,5-.4) -- (4,6-.4) -- 
(5,5-.4) -- (6+.2,6-.4) -- 
(6+.4,14+.4) ; 

\draw [black!40, 
thin, 
] (0,0) -- (6,6)
(0,2) -- (6, 6+2)
(0,4) -- (6-2, 6+4-2)
(0,6) -- (6-2, 6+6-2)
(0,8) -- (6-2, 8+6-2)
;

\draw [black!40, 
%very 
thin, 
] (0,2) -- (1,1)
(0,2+2) -- (2,2)
(0,2+4) -- (3,3)
(0,2+6) -- (4,4)
(0,4) -- (6, 6+4)
(0,6) -- (6, 6+6)
(0,8) -- (6, 8+6)
(0,10) -- (4, 8+6)
(0,12) -- (2, 14)
;

\draw [black!40, 
thin, 
] 
(0,10) -- (5,5)
(0,12) -- (6,6)
(0,14) -- (6,8)
(2,14) -- (6,10)
(4,14) -- (6,12)
;

% First, locate each of the nodes and name them
\node[gray] (s1r1) at (0,0) {11}; 
\node[gray] (s2r2) at (1,1) {21}; 

\node[gray] (s1r3) at (0,2) {12}; 
\node[gray] (s3r3) at (2,2) {31}; 
\node[gray] (22) at (1,3) {22}; 
\node[gray] (41) at (3,3) {41}; 

\node[gray] (13) at (0,2+2) {13}; 
\node[gray] (32) at (2,2+2) {32}; 
\node[gray] (51) at (4,2+2) {51}; 
\node[gray] (23) at (1,3+2) {23}; 
\node (42) at (3,3+2) {\circled{\textbf{42}}}; 
\node (61) at (5,3+2) {\circled{\textbf{61}}}; 

\node[gray] (14) at (0,2+4) {14}; 
\node (33) at (2,2+4) {\circled{\textbf{33}}}; 
\node (52) at (4,2+4) {\circled{\textbf{52}}}; 
\node (71) at (6,2+4) {\starred{\textbf{71}}}; 
\node (24) at (1,3+2+2) {\circled{\textbf{24}}}; 
\node (43) at (3,3+2+2) {\circled{43}}; 
\node (62) at (5,3+2+2) {\starred{62}}; 

\node (15) at (0,2+4+2) {\circled{\textbf{15}}}; 
\node (34) at (2,2+4+2) {\circled{34}}; 
\node (53) at (4,2+4+2) {\starred{53}}; 
\node (72) at (6,2+4+2) {\rectangled{72}}; 
\node (25) at (1,3+2+2+2) {\circled{25}}; 
\node (44) at (3,3+2+2+2) {\starred{44}}; 
\node (63) at (5,3+2+2+2) {\rectangled{63}}; 

\node (16) at (0,2+4+2+2) {\circled{16}}; 
\node (35) at (2,2+4+2+2) {\starred{35}}; 
\node (54) at (4,2+4+2+2) {\rectangled{54}}; 
\node (73) at (6,2+4+2+2) {\rectangled{73}}; 
\node (26) at (1,3+2+2+2+2) {\starred{26}}; 
\node (45) at (3,3+2+2+2+2) {\rectangled{45}}; 
\node (s6r12) at (5,3+2+2+2+2) {\rectangled{64}};

\node (17) at (0,2+4+2+2+2) {\starred{17}}; 
\node (36) at (2,2+4+2+2+2) {\rectangled{36}}; 
\node (55) at (4,2+4+2+2+2) {\rectangled{55}}; 
\node (74) at (6,2+4+2+2+2) {\rectangled{74}}; 
\node[gray] (27) at (1,3+2+2+2+2+2) {27}; 
\node[gray] (46) at (3,3+2+2+2+2+2) {46}; 
\node (65) at (5,3+2+2+2+2+2) {\rectangled{65}}; 

\node[gray] (18) at (0,2+4+2+2+2+2) {18}; 
\node[gray] (37) at (2,2+4+2+2+2+2) {37}; 
\node[gray] (56) at (4,2+4+2+2+2+2) {56}; 
\node (75) at (6,2+4+2+2+2+2) {\rectangled{75}}; 
\end{tikzpicture}
\begin{tikzpicture}[scale=\myscale]
% First, locate each of the nodes and name them
 \node[white] (s1r1) at (0,0) {11}; 

\node (42) at (3,3+2) {$\mathbf{\posetLABELwithS{4}}$}; 
\node (61) at (5,3+2) {$\mathbf{\posetLABELwithS{6}}$};

\node (33) at (2,2+4) {$\mathbf{\posetLABELwithS{3}}$}; 
\node (52) at (4,2+4) {$\mathbf{\posetLABELwithS{5}}$}; 
\node (71) at (6,2+4) {$\mathbf{\posetLABELwithS{7}}$}; 
\node (24) at (1,3+2+2) {$\mathbf{\posetLABELwithS{2}}$}; 
\node (43) at (3,3+2+2) {$\posetLABELwithS{4}$}; 
\node (62) at (5,3+2+2) {$\posetLABELwithS{6}$}; 

\node (15) at (0,2+4+2) {$\mathbf{\posetLABELwithS{1}}$}; 
\node (34) at (2,2+4+2) {$\posetLABELwithS{3}$}; 
\node (53) at (4,2+4+2) {$\posetLABELwithS{5}$}; 
\node (72) at (6,2+4+2) {$\posetLABELwithS{7}$}; 
\node (25) at (1,3+2+2+2) {$\posetLABELwithS{2}$}; 
\node (44) at (3,3+2+2+2) {$\posetLABELwithS{4}$}; 
\node (63) at (5,3+2+2+2) {$\posetLABELwithS{6}$}; 

\node (16) at (0,2+4+2+2) {$\posetLABELwithS{1}$}; 
\node (35) at (2,2+4+2+2) {$\posetLABELwithS{3}$}; 
\node (54) at (4,2+4+2+2) {$\posetLABELwithS{5}$}; 
\node (73) at (6,2+4+2+2) {$\posetLABELwithS{7}$}; 
\node (26) at (1,3+2+2+2+2) {$\posetLABELwithS{2}$}; 
\node (45) at (3,3+2+2+2+2) {$\posetLABELwithS{4}$}; 
\node (64) at (5,3+2+2+2+2) {$\posetLABELwithS{6}$};

\node (17) at (0,2+4+2+2+2) {$\posetLABELwithS{1}$}; 
\node (36) at (2,2+4+2+2+2) {$\posetLABELwithS{3}$}; 
\node (55) at (4,2+4+2+2+2) {$\posetLABELwithS{5}$}; 
\node (74) at (6,2+4+2+2+2) {$\posetLABELwithS{7}$}; 
\node (65) at (5,3+2+2+2+2+2) {$\posetLABELwithS{6}$}; 

\node (75) at (6,2+4+2+2+2+2) {$\posetLABELwithS{7}$};

% Now draw the lines:
\draw [black, thick, 
shorten <=-2pt, shorten >=-2pt] (42) -- (33); 
\draw [black, thick, 
shorten <=-2pt, shorten >=-2pt] (33) -- (24); 
\draw [black, thick, 
shorten <=-2pt, shorten >=-2pt] (24) -- (15); 
\draw [black, thick, 
shorten <=-2pt, shorten >=-2pt] (61) -- (52); 
\draw [black, thick, 
shorten <=-2pt, shorten >=-2pt] (71) -- (61); 

\draw [black, thick, 
shorten <=-2pt, shorten >=-2pt] (52) -- (43); 
\draw [black, thick, 
shorten <=-2pt, shorten >=-2pt] (43)-- (34);  
\draw [black, thick, 
shorten <=-2pt, shorten >=-2pt] (34) -- (25);  
\draw [black, thick, 
shorten <=-2pt, shorten >=-2pt] (25) -- (16); 

\draw [black, thick, 
shorten <=-2pt, shorten >=-2pt] (71) -- (62); 
\draw [black, thick, 
shorten <=-2pt, shorten >=-2pt] (62) -- (53); 
\draw [black, thick, 
shorten <=-2pt, shorten >=-2pt] (53) -- (44); 
\draw [black, thick, 
shorten <=-2pt, shorten >=-2pt] (44) -- (35); 

\draw [black, thick, 
shorten <=-2pt, shorten >=-2pt] (26) -- (17); 
\draw [black, thick, 
shorten <=-2pt, shorten >=-2pt] (35) -- (26); 
\draw [black, thick, 
shorten <=-2pt, shorten >=-2pt] (26) -- (17); 

\draw [black, thick, 
shorten <=-2pt, shorten >=-2pt] (72) -- (63);
\draw [black, thick, 
shorten <=-2pt, shorten >=-2pt] (63) -- (54);
\draw [black, thick, 
shorten <=-2pt, shorten >=-2pt] (54) -- (45);
\draw [black, thick, 
shorten <=-2pt, shorten >=-2pt] (45) -- (36); 

\draw [black, thick, 
shorten <=-2pt, shorten >=-2pt] (73) -- (64); 
\draw [black, thick, 
shorten <=-2pt, shorten >=-2pt] (64) -- (55); 

\draw [black, thick, 
shorten <=-2pt, shorten >=-2pt] (74) -- (65);

\draw [black, thick, 
shorten <=-2pt, shorten >=-2pt] (16) -- (26); 
\draw [black, thick, 
shorten <=-2pt, shorten >=-2pt] (26) -- (36);

\draw [black, thick, 
shorten <=-2pt, shorten >=-2pt] (15) -- (25); 
\draw [black, thick, 
shorten <=-2pt, shorten >=-2pt] (25) -- (35); 
\draw [black, thick, 
shorten <=-2pt, shorten >=-2pt] (35) -- (45); 
\draw [black, thick, 
shorten <=-2pt, shorten >=-2pt] (45) -- (55); 
\draw [black, thick, 
shorten <=-2pt, shorten >=-2pt] (55) -- (65);
\draw [black, thick, 
shorten <=-2pt, shorten >=-2pt] (65) -- (75); 

\draw [black, thick, 
shorten <=-2pt, shorten >=-2pt] (24) -- (34); 
\draw [black, thick, 
shorten <=-2pt, shorten >=-2pt] (34) -- (44); 
\draw [black, thick, 
shorten <=-2pt, shorten >=-2pt] (44) -- (54); 
\draw [black, thick, 
shorten <=-2pt, shorten >=-2pt] (54) -- (64); 
\draw [black, thick, 
shorten <=-2pt, shorten >=-2pt] (64) -- (74); 

\draw [black, thick, 
shorten <=-2pt, shorten >=-2pt] (33) -- (43); 
\draw [black, thick, 
shorten <=-2pt, shorten >=-2pt] (43) -- (53); 
\draw [black, thick, 
shorten <=-2pt, shorten >=-2pt] (53) -- (63); 
\draw [black, thick, 
shorten <=-2pt, shorten >=-2pt] (63) -- (73);

\draw [black, thick, 
shorten <=-2pt, shorten >=-2pt] (42) -- (52); 
\draw [black, thick, 
shorten <=-2pt, shorten >=-2pt] (52) -- (62); 
\draw [black, thick, 
shorten <=-2pt, shorten >=-2pt] (62) -- (72); 
\end{tikzpicture}
\caption{Left: Algorithm for constructing the heap diagram for $\HeapDiagonalReading=\heap(\DiagonalReadingWord)$ for $c=\rw{4321657}=(1~ \underline{5}~\underline{7} ~ 8 ~ \overline{6}~\overline{4}~\overline{3}~ \overline{2})$ (equivalently, the lower-barred letters are $\underline{5}$, $\underline{7}$, and the upper-barred letters are  
$\overline{6}, \overline{4}, \overline{3}, \overline{2}$) in $A_7$. 
Right: The heap diagram for $\HeapDiagonalReading$.}
\label{fig:longest element heap A7 4321657 alternative}
\end{figure}

Notice that by construction $\DiagonalReadingWord$ is the ``diagonal reading word" of $\HeapDiagonalReading$, that is, the linear extension of $\HeapDiagonalReading$ formed by concatenating the diagonals of the diagram from left to right and within each diagonal reading from southeast to northwest. In addition, note that the bottom layer of $\HeapDiagonalReading$ is $\heap(c)$ and 
the rightmost vertex in this bottom layer is labeled $n$ and is at position $(n,1)$. In Figure \ref{fig:longest element heap A7 4321657 alternative}, this bottom layer of shape $\heap(c)$ is in bold.

It is well-known from the literature that the heap of $\sort(w_0)$ matches the description of Algorithm~\ref{alg:algorithm for heap using diagonal reading word}; for example, see 
\cite[Section 2.4]{CharmedRoots} and 
\cite[Section 6.2]{DefantLi23}. 

\begin{proposition}\label{prop:diagonal reading word}
The heap diagrams of $\HeapDiagonalReading$ 
and 
$\heap(\sort(w_0))$ are the same. 
\end{proposition}

Note that the above proposition is equivalent to the statement: 
the word $\DiagonalReadingWord$
is a reduced word in the commutation class of $\sort(w_0)$.

\begin{example}
For $\cTamari=\rw{12 ... n}$, all of $2,\dots,n$ are lower-barred, and thus 
\[\mathcal{R}_{\cTamari}= \rw{1} \rw{21} \rw{321}  ...  \rw{n(n-1) ... 21 }. \] 
Following Algorithm~\ref{alg:algorithm for heap using diagonal reading word}, $H_{\cTamari}$ is given in Figure \ref{fig:heapTamari}.  Recall from Example~\ref{eg.tamari.sort} that the $\cTamari$-sorting word for $w_0$ is 
\[\SpecialSort{\cTamari}(w_0)=\rw{1 ... (n-1)n \mid 1 ... (n-1) \mid ... \mid 12 \mid 1 }.
\]
As guaranteed by Proposition~\ref{prop:diagonal reading word}, $\heap(\SpecialSort{\cTamari}(w_0))$ is the same heap shown in Figure~\ref{fig:heapTamari}.
\end{example}

\begin{figure}[htb!]
\begin{tikzpicture}[scale=0.9]
\node (s1a) at (0,0) {$s_1$};
\node (s2a) at (1,1) {$s_2$};
\node (s3a) at (2,2) {$s_3$};
\node (s4a) at (3,3) {$s_4$};
\node (s1b) at (0,2) {$s_1$};
\node (s2b) at (1,3) {$s_2$};
\node (s3b) at (2,4) {$s_3$};
\node (s1c) at (0,4) {$s_1$};
\node (s2c) at (1,5) {$s_2$};
\node (s1d) at (0,6) {$s_1$};

\draw [black, thick, shorten <=-2pt, shorten >=-2pt] (s1a) -- (s2a) node[pos=0.5, left=5mm]{};
\draw [black, thick, shorten <=-2pt, shorten >=-2pt] (s2a) -- (s3a) node[pos=0.5, left=5mm]{};
\draw [black, thick, shorten <=-2pt, shorten >=-2pt] (s3a) -- (s4a) node[pos=0.5, left=5mm]{};

\draw [black, thick, shorten <=-2pt, shorten >=-2pt] (s1b) -- (s2b) node[pos=0.5, left=5mm]{};
\draw [black, thick, shorten <=-2pt, shorten >=-2pt] (s2b) -- (s3b) node[pos=0.5, left=5mm]{};

\draw [black, thick, shorten <=-2pt, shorten >=-2pt] (s1c) -- (s2c) node[pos=0.5, left=5mm]{};

\draw [thick, shorten <=-2pt, shorten >=-2pt] (s2a) -- (s1b); 
\draw [thick, shorten <=-2pt, shorten >=-2pt] (s3a) -- (s2b) -- (s1c);
\draw [thick, shorten <=-2pt, shorten >=-2pt] (s4a) -- (s3b) -- (s2c) -- (s1d);
\end{tikzpicture}
\hspace{0.5in}
\begin{tikzpicture}
\node (s1a) at (0,0) {$s_1$};
\node (s2a) at (1,1) {$s_2$};
\node (s3a) at (2,2) {$s_3$};
\node (s4a) at (3,3) {$s_4$};
\node (s1b) at (0,2) {$s_1$};
\node (s2b) at (1,3) {$s_2$};
\node (s3b) at (2,4) {$s_3$};
\node (s1c) at (0,4) {$s_1$};
\node (s2c) at (1,5) {$s_2$};
\node (s1d) at (0,6) {$s_1$};

\node at (0.5,6.5) {$\vdots$};
\node at (3,4) {$\vdots$};
\node (sn1) at (3,5) {$s_{n-1}$};
\node (sn) at (4,4) {$s_n$};
\node at (2,6) {$\ddots$};
\node (s1) at (0,8) {$s_1$};
\node (s2) at (1,7) {$s_2$};
\draw [black, thick, shorten <=-2pt, shorten >=-2pt] (sn1) -- (sn); 
\draw [black, thick, shorten <=-2pt, shorten >=-2pt] (s1) -- (s2);

\draw [black, thick, shorten <=-2pt, shorten >=-2pt] (s1a) -- (s2a) node[pos=0.5, left=5mm]{};
\draw [black, thick, shorten <=-2pt, shorten >=-2pt] (s2a) -- (s3a) node[pos=0.5, left=5mm]{};
\draw [black, thick, shorten <=-2pt, shorten >=-2pt] (s3a) -- (s4a) node[pos=0.5, left=5mm]{};

\draw [black, thick, shorten <=-2pt, shorten >=-2pt] (s1b) -- (s2b) node[pos=0.5, left=5mm]{};
\draw [black, thick, shorten <=-2pt, shorten >=-2pt] (s2b) -- (s3b) node[pos=0.5, left=5mm]{};

\draw [black, thick, shorten <=-2pt, shorten >=-2pt] (s1c) -- (s2c) node[pos=0.5, left=5mm]{};

\draw [thick, shorten <=-2pt, shorten >=-2pt] (s2a) -- (s1b); 
\draw [thick, shorten <=-2pt, shorten >=-2pt] (s3a) -- (s2b) -- (s1c);
\draw [thick, shorten <=-2pt, shorten >=-2pt] (s4a) -- (s3b) -- (s2c) -- (s1d);
\end{tikzpicture}
\caption{Left: Heap of $s_1s_2s_3s_4$-sorting word of the longest element $w_0$ in $A_4$.  Right: Heap of $s_1s_2s_3s_4 \dots s_n$-sorting word of the longest element $w_0$ in $A_n$.}
\label{fig:heapTamari}
\end{figure}

\begin{remark}\label{rem:prop:vertex-bijection}
In light of Proposition~\ref{prop:diagonal reading word}, we can consider the poset isomorphism in
Proposition~\ref{prop:vertex-bijection} to be a poset isomorphism from the lattice of $c$-singletons onto $J(\HeapDiagonalReading)$. 
By abuse of notation, from now on 
we write $f(w)$ to mean an order ideal of $\HeapDiagonalReading$. 
\end{remark}

\begin{remark}
We summarize some basic properties of $\HeapDiagonalReading$ below.
\label{rem:diagonal above main diagonal in Heap has an element labeled n}
\label{rem:shape of H}
\begin{enumerate} [(a)]
\item \label{rem:diagonal above main diagonal in Heap has an element labeled n:itm}
A diagonal in $\HeapDiagonalReading$ has an element labeled $1$ if and only if it is a diagonal created in \ref{alg:algorithm for heap using diagonal reading word:itm1} or \ref{alg:algorithm for heap using diagonal reading word:itm2}; this element labeled 1 is at the northwest-most vertex of the diagonal.
Similarly, 
a diagonal has an element labeled $n$ if and only if it is a diagonal created in \ref{alg:algorithm for heap using diagonal reading word:itm1} or \ref{alg:algorithm for heap using diagonal reading word:itm3}; this element labeled $n$ is at the southeast-most vertex of the diagonal.

\item \label{rem:shape of H:item:when not smallest elt with its label}
Let $t$ be an element of $\HeapDiagonalReading$ with label $p$. Suppose $t$ is not the smallest element in $\HeapDiagonalReading$ with label $p$.
If $1<p<n$, then $H_c$ has 
an interval whose heap diagram is of the form 
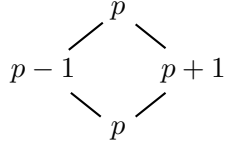
\begin{center}
\begin{tikzpicture}[xscale=1, yscale=.8]
\node (plower) at (0,0) {$p$};
\node (pupper) at (0,2) {$p$};
\node (pminus) at (-1,1) {$p-1$};
\node (pplus) at (1,1) {$p+1$}; \draw [black, thick, shorten <=-1pt,shorten >=-1pt] 
(plower) -- (pplus) 
(plower) -- (pminus) 
(pupper) -- (pplus) 
(pupper) -- (pminus) ;     
\end{tikzpicture}
\end{center}

If $p=1$, then $H_c$ has an interval whose heap diagram is of the form 
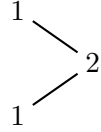
\begin{center}
\begin{tikzpicture}[xscale=1, yscale=.7]
\node (plower) at (0,0) {$1$};
\node (pupper) at (0,2) {$1$};
\node (pplus) at (1,1) {$2$}; \draw [black, thick, shorten <=-1pt,shorten >=-1pt] (plower) -- (pplus);
\draw [black, thick, shorten <=-1pt,shorten >=-1pt] (pupper) -- (pplus);     
\end{tikzpicture}
\end{center}

If $p=n$, then $H_c$ has an interval 
whose heap diagram is of the form 
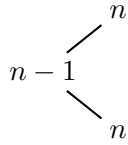
\begin{center}
\begin{tikzpicture}[xscale=1, yscale=.8]
\node (plower) at (0,0) {$n$};
\node (pupper) at (0,2) {$n$};
\node (pminus) at (-1,1) {$n-1$}; \draw [black, thick, shorten <=-1pt,shorten >=-1pt] (plower) -- (pminus);
\draw [black, thick, shorten <=-1pt,shorten >=-1pt] (pupper) -- (pminus);     
\end{tikzpicture}
\end{center}

In each of these three cases, the maximal vertex corresponds to the poset element $t$.
\end{enumerate}
\end{remark}

Recall from Definition~\ref{defn:heap of a reduced word} that the elements in $\HeapDiagonalReading$ are $\{1,\dots, \Coxeterlength(w_0)\}$ where the poset element $i$ corresponds to the $i$\textsuperscript{th} letter of $\DiagonalReadingWord$. For $x$ in $\HeapDiagonalReading$, let $D_x$ denote the diagonal of $\HeapDiagonalReading$ containing $x$, and let  $A_x$ denote the antidiagonal of $\HeapDiagonalReading$ containing $x$. 
The following states that the poset element $j$ is comparable to every poset element $i\in \{1,\dots,j\}$ labeled $n$. 
\begin{lemma}\label{lem:j is comparable to smaller integers i labeled n}
Let $j \in \{ 1, \dots, \Coxeterlength(w_0)\}$ be a poset element of $\HeapDiagonalReading$. 
If $i \leq j$ and  
the label of the poset element $i$ is $n$, 
then 
$i \HeapLEQ j$.  
\end{lemma}
\begin{proof}
Remark \ref{rem:diagonal above main diagonal in Heap has an element labeled n}\ref{rem:diagonal above main diagonal in Heap has an element labeled n:itm} tells us that all poset elements labeled $n$ are in the long diagonal $\LongDiagonal$ or the diagonals above $\LongDiagonal$. So either $D_i=\LongDiagonal$ or $D_i$ is above $\LongDiagonal$.

Since $j\geq i$ as integers, either $D_j=D_i$ or $D_j$ is above $D_i$.
So, either $D_j=\LongDiagonal$ or $D_j$ is above $\LongDiagonal$. Therefore $D_j$ contains an element $j'$ labeled $n$ such that $j' \HeapLEQ j$, by Remark \ref{rem:diagonal above main diagonal in Heap has an element labeled n}\ref{rem:diagonal above main diagonal in Heap has an element labeled n:itm}.

Remark \ref{rem:poset elements with the same label are comparable} tells us that two elements with the same label are comparable, so
$j'$ and $i$ are comparable. In particular, since $D_{j'}=D_j$ is either equal to $D_i$ or above $D_i$, we have $i \HeapLEQ j'$. 
Thus, $i \HeapLEQ j$. 
\end{proof}

\begin{lemma}
\label{lem:if z is in a diagonal to the right and in an antidiagonal to the right}
Let $y$ be a poset element of $\HeapDiagonalReading$ with label $p$. 
If $z$ is in a diagonal (strictly) above
$D_y$ and in an antidiagonal (strictly) below $A_y$, then the label of $z$ 
is in 
$\{p+2,...,n\}$.
\end{lemma}
\begin{proof}
Let $(a,b)$ and $(a',b')$ denote the positive integer tuples which index the positions of $y$ and $z$ in the heap diagram of $\HeapDiagonalReading$ (see  Algorithm~\ref{alg:algorithm for heap using diagonal reading word}).

Since $z$ is in  an antidiagonal below $A_y$, we have $b'<b$, which implies that $b-b'$ is a positive integer. 
Since $z$ is in a diagonal above $D_y$, we have $a+b < a' + b'$, and thus
\[ a+ (b-b') < a' .\]
This implies $a+1<a'$, and so $a' \in \{ a+2, ..., n\}$. 
By the last step in the algorithm,  
the labels of 
$y$ and $z$ are $a$ and $a'$, and the conclusion follows.
\end{proof}

\begin{lemma}\label{lem:southeast vertex of a diagonal whose label is p not n is the minimum element of Heap with label p}
Let $i$ be the poset element at the southeast vertex of a diagonal $D_i$ such that its label is 
$p \in \{2,...,n-1 \}$.
Then we have the following. 
\begin{enumerate}[(a)]
\item \label{lem:southeast vertex of a diagonal whose label is p not n is the minimum element of Heap with label p:itm1}
$i$ is the minimum element of $\HeapDiagonalReading$ with label $p$. 

\item \label{lem:southeast vertex of a diagonal whose label is p not n is the minimum element of Heap with label p:itm2}
If a poset element $j$ is in a diagonal below $D_i$, then the label of $j$ is in $\{1,\dots, p-1\}$.
\end{enumerate}
\end{lemma}
\begin{proof}
\begin{enumerate}[(a)] 
\item 
Suppose for the sake of contradiction that $H$ contains a smaller element also with label $p$. 
Then by Remark \ref{rem:shape of H}\ref{rem:shape of H:item:when not smallest elt with its label} we have an interval whose heap subdiagram is of the form 
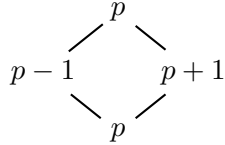
\begin{center}
\begin{tikzpicture}[xscale=1, yscale=.8]
\node (plower) at (0,0) {$p$};
\node (pupper) at (0,2) {$p$};
\node (pminus) at (-1,1) {$p-1$};
\node (pplus) at (1,1) {$p+1$}; \draw [black, thick, shorten <=-1pt,shorten >=-1pt] 
(plower) -- (pplus) 
(plower) -- (pminus) 
(pupper) -- (pplus) 
(pupper) -- (pminus) ;     
\end{tikzpicture}
\end{center}
where $i$ corresponds to the top row. This contradicts the fact that $i$ is the southeast vertex of $D_i$. Therefore, $i$ is indeed the minimum element of $H$ with label $p$.

\item 
Since $D_i$ is below the long diagonal, \ref{alg:algorithm for heap using diagonal reading word:itm2} of the algorithm tells us that any diagonal~$D_j$ below $D_i$ is shorter than $D_i$. Both $D_i$ and $D_j$ have northwest-most vertex labeled $1$. Since the vertices of $D_i$ have labels $1,\dots,p$, it must be that $D_j$ have labels $1,\dots,q$ for some $q \leq p-1$. 
\end{enumerate}
\end{proof}

The following remark explains the connection between $\Heapwnot$ and an Auslander--Reiten quiver.

\begin{remark}
\label{rem:combinatorial AR quiver}
Consider the following minor modifications to the heap diagram of $\Heapwnot$.
\begin{enumerate}
\item 
Replace each edge $x$ --- $y$ (which corresponds to cover relation $x \HeapLessThan y$)  with an arrow $x \to y$.
\item 
Replace each label $j$ by $s_j$ (note that we have already been doing this for clarity in all of our figures)
\item 
Rotate the diagram by $90$ degree clockwise.
\end{enumerate}
The resulting labeled directed graph is known as the
``combinatorial $AR$ quiver" for the $c$-sorting word of $w_0$, defined in \cite[Section 9.2]{Cataland}; see also \cite[Section 2.4]{CharmedRoots}. Note that all arrows in this  directed graph are pointing either diagonally or antidiagonally from left to right.

This terminology comes from the fact that this directed graph is isomorphic to the Auslander--Reiten (AR) quiver of finite-dimensional representations of a 
quiver $Q(c)$ corresponding to $c$. The quiver $Q(c)$ is an orientation of the type $A_n$ Dynkin diagram with vertices $1,2,\dots,n$ and with arrows defined as follows:
 if $i$ appears to the right of $i-1$ in a reduced word of $c$, 
then $Q(c)$ has an arrow $i  \to i-1$; otherwise $Q(c)$ has an arrow $i-1 \to i$. For more information about representations of quivers, see for example the textbooks~\cite{ASS06,SchifflerBook}.

There is a well-known procedure for constructing the AR quiver for $Q(c)$ known as the \emph{knitting algorithm}; it recursively computes the indecomposable representations of $Q(c)$. For a reference on knitting, see for example \cite[Chapter 3]{SchifflerBook}. 
One can view a linear extension as a prescription of the order in which to perform the knitting algorithm. In particular, the number of linear extensions gives the number of ways to compute the AR quiver of a given quiver via knitting together with a choice for adding indecomposable projective representations. 
Later in this paper 
we will show that the normalized volume of our $c$-Birkhoff polytope $\bir(c)$ counts the linear extensions of $\Heapwnot$, and thus this volume also counts the ways to perform the knitting algorithm (see Corollary~\ref{cor:volume}).

The AR quiver has also been used in \cite{Bed99} to compute the number of reduced words in a commutation class. 
\end{remark}

\subsection{Defining a unitriangular matrix $\mathcal{U}_c$}\label{subsec:Square}

The main goal of this subsection is to prove Proposition~\ref{prop.square}. We first define the notation used in this result. 

Note that the length of $\mathcal{R}_c$ is $\binom{n+1}{2}=\ell(w_0)$, and write $\mathcal{R}_c = \rw{\DiagWordSubseq{1} \dots \DiagWordSubseq{\ell(w_0)}}$. 
For each $1\leqslant i \leqslant \ell(w_0)$,  
define $\prefixR{i}$ to be the permutation with reduced word given by the length $i$ prefix of  $\mathcal{R}_c$,  i.e. $\prefixR{i}$ has reduced word $\rw{\DiagWordSubseq{1} \dots \DiagWordSubseq{i}}$. 
Since $\mathcal{R}_c$ is a labeled linear extension of $\heap(\sort(w_0))$, 
Proposition \ref{prop:set of labeled linear extensions is commutativity class Stembridge} tells us that
$\mathcal{R}_c$ is in the commutation class of $\sort(w_0)$.  Thus, Theorem~\ref{thm:c-sing-prefix} tells us that each $\prefixR{i}$ is a $c$-singleton.

Given a $c$-singleton $w$, let $f(w)$ be the corresponding order ideal of $\HeapDiagonalReading$ given in  Proposition~\ref{prop:vertex-bijection} (see Remark~\ref{rem:prop:vertex-bijection}).
Consider the vector in $\mathbb{R}^{\ell(w_0)}$ defined by the indicator function of $f(w)$, following the linear extension given by $\mathcal{R}_c$. 
Let $o(w)$ denote this vector in reverse order.

\begin{example}\label{eg.basis}
Let $c$ be as in Example~\ref{eg.proj}
with $\underline{\lowerbarletter_1}, \underline{\lowerbarletter_2}, \underline{\lowerbarletter_3} = \underline{2}, \underline{5}, \underline{7}$ and $\overline{\upperbarletter_1} , \overline{\upperbarletter_2}, \overline{\upperbarletter_3} = \overline{3}, \overline{4}, \overline{6}$.
We have $\mathcal{R}_c = \rw{1\,\, 4321\,\, 654321\,\, 7654321\,\, 76543\,\, 765\,\, 76}$. Therefore $\prefixR{4} = s_1s_4s_3s_2=25134678$.
The heap 
$\HeapDiagonalReading$ 
is in Figure \ref{fig:eg.basis}, with the highlighted elements corresponding to $f(b_4)$. 
Since the highlighted elements are the first four elements in $\mathcal{R}_c$, we have
\[
o(\prefixR{4}) = (0,0,0,0,0, 0,0,0,0,0, 0,0,0,0,0, 0,0,0,0,0, 0,0,0,0, 1,1,1,1)^T\,.
\]
\end{example}

\def\myscale{.8}
\def\mycirclecolor{violet!10}
\def\mycirclesize{2ex}
\begin{figure}[htb!]
\begin{tikzpicture}[scale=\myscale]
\draw[\mycirclecolor,fill=\mycirclecolor] (3,0) circle (\mycirclesize);
\node (2) at (3,0) {$2$};
\node (6) at (5,0) {$6$};

\draw[\mycirclecolor,fill=\mycirclecolor] (0,1) circle (\mycirclesize);
\node (1) at (0,1) {$1$};
\draw[\mycirclecolor,fill=\mycirclecolor] (2,1) circle (\mycirclesize);
\node (3) at (2,1) {$3$};
\node (7) at (4,1) {$7$};
\node (12) at (6,1) {$12$};

\draw[\mycirclecolor,fill=\mycirclecolor] (1,2) circle (\mycirclesize);
\node (4) at (1,2) {$4$};
\node (8) at (3,2) {$8$};
\node (13) at (5,2) {$13$};

\node (5) at (0,3) {$5$};
\node (9) at (2,3) {$9$};
\node (14) at (4,3) {$14$};
\node (19) at (6,3) {$19$};

\node (10) at (1,4) {$10$};
\node (15) at (3,4) {$15$};
\node (20) at (5,4) {$20$};

\node (11) at (0,5) {$11$};
\node (16) at (2,5) {$16$};
\node (21) at (4,5) {$21$};
\node (24) at (6,5) {$24$};

\node (17) at (1,6) {$17$};
\node (22) at (3,6) {$22$};
\node (25) at (5,6) {$25$};

\node (18) at (0,7) {$18$};
\node (23) at (2,7) {$23$};
\node (26) at (4,7) {$26$};
\node (27) at (6,7) {$27$};

\node (28) at (5,8) {$28$};

\begin{scope}[thick]
\draw (2) -- (3) -- (4) -- (5);
\draw (6) -- (7) -- (8) -- (9) -- (10) -- (11);
\draw (12) -- (13) -- (14) -- (15) -- (16) -- (17) -- (18);
\draw (19) -- (20) -- (21) -- (22) -- (23);
\draw (24) -- (25) -- (26);
\draw (27) -- (28);

\draw (6) -- (12);
\draw (2) -- (7) -- (13) -- (19);
\draw (3) -- (8) -- (14) -- (20) -- (24);
\draw (1) -- (4) -- (9) -- (15) -- (21) -- (25) -- (27);
\draw (5) -- (10) -- (16) -- (22) -- (26) -- (28);
\draw (11) -- (17) -- (23);
\end{scope}
\end{tikzpicture}
\qquad
\begin{tikzpicture}[scale=\myscale]

\draw[\mycirclecolor,fill=\mycirclecolor] (3,0) circle (\mycirclesize);
\node (2) at (3,0) {$s_4$};
\node (6) at (5,0) {$s_6$};

\draw[\mycirclecolor,fill=\mycirclecolor] (0,1) circle (\mycirclesize);
\node (1) at (0,1) {$s_1$};
\draw[\mycirclecolor,fill=\mycirclecolor] (2,1) circle (\mycirclesize);
\node (3) at (2,1) {$s_3$};
\node (7) at (4,1) {$s_5$};
\node (12) at (6,1) {$s_7$};

\draw[\mycirclecolor,fill=\mycirclecolor] (1,2) circle (\mycirclesize);
\node (4) at (1,2) {$s_2$};
\node (8) at (3,2) {$s_4$};
\node (13) at (5,2) {$s_6$};

\node (5) at (0,3) {$s_1$};
\node (9) at (2,3) {$s_3$};
\node (14) at (4,3) {$s_5$};
\node (19) at (6,3) {$s_7$};

\node (10) at (1,4) {$s_2$};
\node (15) at (3,4) {$s_4$};
\node (20) at (5,4) {$s_6$};

\node (11) at (0,5) {$s_1$};
\node (16) at (2,5) {$s_3$};
\node (21) at (4,5) {$s_5$};
\node (24) at (6,5) {$s_7$};

\node (17) at (1,6) {$s_2$};
\node (22) at (3,6) {$s_4$};
\node (25) at (5,6) {$s_6$};

\node (18) at (0,7) {$s_1$};
\node (23) at (2,7) {$s_3$};
\node (26) at (4,7) {$s_5$};
\node (27) at (6,7) {$s_7$};

\node (28) at (5,8) {$s_6$};

\begin{scope}[thick]
\draw (2) -- (3) -- (4) -- (5);
\draw (6) -- (7) -- (8) -- (9) -- (10) -- (11);
\draw (12) -- (13) -- (14) -- (15) -- (16) -- (17) -- (18);
\draw (19) -- (20) -- (21) -- (22) -- (23);
\draw (24) -- (25) -- (26);
\draw (27) -- (28);

\draw (6) -- (12);
\draw (2) -- (7) -- (13) -- (19);
\draw (3) -- (8) -- (14) -- (20) -- (24);
\draw (1) -- (4) -- (9) -- (15) -- (21) -- (25) -- (27);
\draw (5) -- (10) -- (16) -- (22) -- (26) -- (28);
\draw (11) -- (17) -- (23);
\end{scope}
\end{tikzpicture}
\caption{Accompanying figure for Example~\ref{eg.basis}. Left: Hasse diagram of the underlying poset of $\HeapDiagonalReading$ with the elements in $f(\prefixR{4})$ highlighted. Right: the heap diagram of $\HeapDiagonalReading$ with the elements in $f(\prefixR{4})$ highlighted.}
\label{fig:eg.basis} 
\end{figure}

\begin{remark}\label{rmk:o(bi)}
In general,  $o(\prefixR{i})$ is the vector with 1's in the last $i$ entries and 0's everywhere else.
\end{remark}

\begin{prop}\label{prop.square}
Let $c$ be a Coxeter element of $A_n$. There exists a unique $\binom{n+1}{2} \times \binom{n+1}{2}$ lower unitriangular matrix $\mathcal{U}_c$ such that
\begin{equation}
\label{eq:U Pi M bi equals o of bi}
\mathcal{U}_c\circ\Pi_c(\PermMatrix{\prefixR{i}}) = o(\prefixR{i}) \text{ for all }  1\leqslant i \leqslant \binom{n+1}{2}.
\end{equation}
\end{prop}

Before we prove this proposition, we first introduce some notation building on Definition~\ref{defn.projection}. The projection $\Pi_c$ picks $\binom{n+1}{2}$ elements from an ${(n+1)} \times {(n+1)}$ matrix. We denote the entries that come from strictly above the main diagonal as $\nepi$, and those that come from strictly below the main diagonal as $\swpi$. By abuse of notation, we will also use $\nepi$ (resp. $\swpi$) to denote entries in $X(w)$ that are recorded by $\Pi_c$ and are above the main diagonal (resp. below the main diagonal).

\begin{proposition}\label{prop.swpi}
Let $c\in A_n$ be a Coxeter element and suppose $(p,q)\in \swpi$. There exists $1\leq t \leq s$ such that:
\begin{enumerate}[(1)]
\item \label{it.d} if $q = d_a$, then $p = n+2-t-a$;
\item \label{it.u} if $q = u_a$, then $p = n+2-t+a$ and $a<t$. 
\end{enumerate}
\end{proposition}
\begin{proof}
By Proposition~\ref{prop.sw} and Remark~\ref{rem.ne}, if the entry $(p,q)\in \swpi$, then $p = n+2-k$ and $q = c^k(u_t)$ for some  $1\leq t \leq s$ and $1\leq k \leq m = \mu(u_t)$. The statement then follows from the observation in Remark~\ref{rem:Coxeter element in cycle notation} that
\[
c^k(u_t) = \begin{cases}
    u_{t-k}, & \text{if } t>k\,,\\
    d_{k-t}, & \text{if } t\leq k\,.
 \end{cases}
\]
\end{proof}

\begin{proposition}
Let $c\in A_n$ be a Coxeter element. Then $(p,q)\in \nepi$ if and only if one of the following holds:
\begin{enumerate}
\item $q = d_a$ and $p\in [1,d_a-1]$.
\item $q = u_a$ and $p\in [m+1,u_a-1]$ where $m = \mu(u_a)$.
\end{enumerate}
\end{proposition}
\begin{proof}
This follows directly from Definition~\ref{defn.projection}, Proposition~\ref{prop.sw}, and Remark~\ref{rem.ne}.
\end{proof}

Recall from \eqref{eq:Def:Rc} that 
\[
\mathcal{R}_c = \rw{(\underline{\lowerbarletter_1} -1) ... 1} \dots  
\rw{(\underline{\lowerbarletter_r} -1) ... 1} 
\rw{(\underline{\lowerbarletter_{r+1}}-1) ... 1} 
\rw{n ... (n-\overline{\upperbarletter_s} +2)} \dots  
\rw{n ... (n- \overline{\upperbarletter_1} +2)}\,.
\]
We will refer to each maximal decreasing run of $\RC$ as the ``run of $\lowerbarletter_i$ (resp. $\upperbarletter_i$)'', denoted $\RC[\lowerbarletter_i]$ (resp. $\RC[\upperbarletter_i]$).  With this notation we can write  \[
\mathcal{R}_c =\RC[\lowerbarletter_1] \dots \RC[\lowerbarletter_r] \, \RC[\lowerbarletter_{r+1}] \, \RC[\upperbarletter_s] \dots \RC[\upperbarletter_1].
\]
For $1\leq k \leq r+1$ and $1\leq j \leq d_k-1$ we will let $d^{(k)}_j$ be the concatenation
\[
\RC[\lowerbarletter_1] \dots \RC[\lowerbarletter_{k-1}]
\,\rw{(d_k-1)\dots (d_k-j)}.
\]
In particular, we define $d^{(k)} := d^{(k)}_{d_k-1}=\RC[\lowerbarletter_1] \dots \RC[\lowerbarletter_{k}]$ for $1\leq k \leq r+1$ and $d^{(k)}_0 := d^{(k-1)} = d^{(k-1)}_{d_{k-1}-1}$ for $2 \leq k \leq r+1$. 

Similarly, for $1\leq k \leq s$ and $1\leq j \leq u_k-1$ we will use the notation $u^{(k)}_j$ to be the concatenation of 
\[
\RC[\lowerbarletter_1] \dots \RC[\lowerbarletter_{r+1}] \RC[\upperbarletter_s]\dots \RC[\upperbarletter_{k+1}]\, \rw{n\dots (n+1-j)}.
\]
In particular, define $u^{(k)} = u^{(k)}_{u_k-1}=\RC[\lowerbarletter_1] \dots \RC[\lowerbarletter_{r+1}] \RC[\upperbarletter_s]\dots \RC[\upperbarletter_{k}]$ for $1\leq k \leq s$ and $u_0^{(k)} := u^{(k+1)} := u^{(k+1)}_{u_{k+1}-1}$ for $1\leq k \leq s-1$.

\begin{example}
\label{ex.notation}
Continuing from Example~\ref{eg.basis}, we have $d_1 = 2, d_2 = 5, d_3 = 7, d_4 = 8$ and $u_1 = 3, u_2 = 4, u_3 = 6$. So we get $d^{(3)}_2 = \rw{1\ 4321\ 65}$ and $u^{(2)}_3 = \rw{1\ 4321\ 654321\ 7654321\ 76543\ 765}$.
\end{example}

To prove Proposition~\ref{prop.square}, we will need to show that for $1\leq i \leq \binom{n+1}{2}$, there is a 1 in the $i$\textsuperscript{th}-to-last  position of $\Pi_c(X(\prefixR{i}))$ and there are 0's in all earlier positions of $\Pi_c(X(\prefixR{i}))$.
We break this analysis into two cases, based on whether $b_i = d_j^{(k)}$ or $b_i = u_j^{(k)}$ in Lemmas~\ref{lem.lower} and \ref{lem.upper} respectively.

\begin{remark}\label{rmk:Described^(k)andu^(k)}
Here we summarize some computations which are useful for the next lemmas. 
\begin{enumerate}[(a)]
\item We can describe the permutation matrix for $d^{(k)}$ as follows:
\begin{itemize}
    \item In column $d_a$, if $a \leq k$, there is a 1 in row $k+1-a$ and if $a > k$, there is a 1 in row $d_a$\,;
    \item In column $u_a$, there is a 1 in row $u_a + \vert \{ d_b > u_a : d_b \leq d_k\} \vert$. Using Lemma~\ref{lem:content_of_interval}, we can show this is equal to $k+a+1$. 
\end{itemize}

\item \label{rmk:Described^(k)andu^(k):itm:2}
We can describe the permutation matrix for $u^{(k)}$ as follows
\begin{itemize}
    \item In column $j \geq u_k$, there is a $1$ in row $n+2-j$;
    \item In column $d_a < u_k$, there is a $1$ in row $(r+2-a) + (s-k+1) = r+s+3-a-k = n+2-a-k$\,;
    \item In column $u_a < u_k$, there is a $1$ in row $(n+1-s+a) + (s-k+1) = n+2+a-k$\,.
\end{itemize}
To see the above is true, we can verify this for $k = 1$  and then do induction on $k$.
\end{enumerate}

\end{remark}

\begin{example}
Continuing from Example~\ref{ex.notation}, we compute the changes of the permutation matrix from $d^{(2)}$ to $d^{(3)}$. In this process we multiply $d^{(2)}$ on the right by $\rw{654321}$. This is demonstrated in Figure~\ref{fig.lower}. We also compute the changes of the permutation matrix from $u^{(3)}$ to $u^{(2)}$. In this process we multiply $u^{(3)}$ on the right by $\rw{765}$. This is demonstrated in Figure~\ref{fig.upper}.

\begin{figure}[htb!]
\begin{center}
\begin{minipage}{0.49\textwidth}
\begin{tabular}{|c|c|c|c|c|c|c|c|}
\hline
  & \cellcolor{green!50} & $\X$  & $\X$ & \cellcolor{green!50}\Circled{1} & $\X$ &  \cellcolor{green!50} &  \cellcolor{green!50}\\ \hline
  &  \Circled{1}  & $\X$  & $\X$ & \cellcolor{green!50} & $\X$ & \cellcolor{green!50} &  \cellcolor{green!50}\\ \hline
\Circled{1}  &    &    & $\X$ & \cellcolor{green!50} & \cellcolor{green!50} & \cellcolor{green!50} & \cellcolor{green!50} \\ \hline
  &    &  \Circled{1}  &   &  & \cellcolor{green!50} & \cellcolor{green!50} & \cellcolor{green!50} \\ \hline
  &    &    & \Circled{1}  &    & \cellcolor{green!50} & \cellcolor{green!50} & \cellcolor{green!50} \\ \hline
  &  \cellcolor{yellow!50} &   &   &  $\X$  & \Circled{1}  & \cellcolor{green!50} & \cellcolor{green!50} \\ \hline
\cellcolor{yellow!50} &  \cellcolor{yellow!50} & \cellcolor{yellow!50} &   &  $\X$  &   &  \Circled{1}   & \cellcolor{green!50} \\ \hline
\cellcolor{yellow!50} & $\X$  & \cellcolor{yellow!50}  & \cellcolor{yellow!50} &  $\X$  &   & $\X$   & \Circled{1}   \\ \hline
\end{tabular}
\end{minipage}\hfill
\begin{minipage}{0.49\textwidth}
\begin{tabular}{|c|c|c|c|c|c|c|c|}
\hline
  & \cellcolor{green!50} & $\X$  & $\X$ & \cellcolor{green!50} & $\X$ &  \cellcolor{green!50}\Circled{1}&  \cellcolor{green!50}\\ \hline
  &    & $\X$  & $\X$ & \cellcolor{green!50}\Circled{1} & $\X$ & \cellcolor{green!50} &  \cellcolor{green!50}\\ \hline
  &  \Circled{1}  &    & $\X$ & \cellcolor{green!50} & \cellcolor{green!50} & \cellcolor{green!50} & \cellcolor{green!50} \\ \hline
\Circled{1}  &    &    &   &  & \cellcolor{green!50} & \cellcolor{green!50} & \cellcolor{green!50} \\ \hline
  &    &  \Circled{1}   &  &    & \cellcolor{green!50} & \cellcolor{green!50} & \cellcolor{green!50} \\ \hline
  &  \cellcolor{yellow!50} &   & \Circled{1}   &  $\X$  &  & \cellcolor{green!50} & \cellcolor{green!50} \\ \hline
\cellcolor{yellow!50} &  \cellcolor{yellow!50} & \cellcolor{yellow!50} &   &  $\X$  &  \Circled{1}  &    & \cellcolor{green!50} \\ \hline
\cellcolor{yellow!50} & $\X$  & \cellcolor{yellow!50}  & \cellcolor{yellow!50} &  $\X$  &   & $\X$   & \Circled{1}   \\ \hline
\end{tabular}
\end{minipage}
\caption{On the left we have the permutation matrix associated to the reduced word $d^{(2)}$ and on the right we have the same for $d^{(3)}$, where $c = [1432657]$ as in previous examples. Comparing these illustrates the main point of Lemma \ref{lem.lower}.}
 \label{fig.lower}
\end{center}
\begin{center}
\begin{minipage}{0.49\textwidth}
\begin{tabular}{|c|c|c|c|c|c|c|c|}
\hline
& \cellcolor{green!50} & $\X$  & $\X$ & \cellcolor{green!50} & $\X$ &  \cellcolor{green!50}& \Circled{1} \cellcolor{green!50}\\ \hline
&    & $\X$  & $\X$ & \cellcolor{green!50} & $\X$ & \cellcolor{green!50}\Circled{1} &  \cellcolor{green!50}\\ \hline
  &    &    & $\X$ & \cellcolor{green!50} & \cellcolor{green!50}\Circled{1} & \cellcolor{green!50} & \cellcolor{green!50} \\ \hline
  &    &    &   &\Circled{1}  & \cellcolor{green!50} & \cellcolor{green!50} & \cellcolor{green!50} \\ \hline
  &  \Circled{1}  &    &   &    & \cellcolor{green!50} & \cellcolor{green!50} & \cellcolor{green!50} \\ \hline
\Circled{1}  &  \cellcolor{yellow!50} &   &   &  $\X$  &  & \cellcolor{green!50} & \cellcolor{green!50} \\ \hline
\cellcolor{yellow!50} &  \cellcolor{yellow!50} & \cellcolor{yellow!50}\Circled{1} &   &  $\X$  &   &     & \cellcolor{green!50} \\ \hline
\cellcolor{yellow!50} & $\X$  & \cellcolor{yellow!50}  & \cellcolor{yellow!50}\Circled{1} &  $\X$  &   & $\X$   &    \\ \hline
\end{tabular}
\end{minipage}\hfill
\begin{minipage}{0.49\textwidth}
\begin{tabular}{|c|c|c|c|c|c|c|c|}
\hline
& \cellcolor{green!50} & $\X$  & $\X$ & \cellcolor{green!50} & $\X$ &  \cellcolor{green!50}& \Circled{1} \cellcolor{green!50}\\ \hline
&    & $\X$  & $\X$ & \cellcolor{green!50} & $\X$ & \cellcolor{green!50}\Circled{1} &  \cellcolor{green!50}\\ \hline
  &    &    & $\X$ & \cellcolor{green!50} & \cellcolor{green!50}\Circled{1} & \cellcolor{green!50} & \cellcolor{green!50} \\ \hline
  &    &    &   &\Circled{1}  & \cellcolor{green!50} & \cellcolor{green!50} & \cellcolor{green!50} \\ \hline
  &    &    & \Circled{1}  &    & \cellcolor{green!50} & \cellcolor{green!50} & \cellcolor{green!50} \\ \hline
  &  \cellcolor{yellow!50}\Circled{1} &   &   &  $\X$  &  & \cellcolor{green!50} & \cellcolor{green!50} \\ \hline
\cellcolor{yellow!50}\Circled{1} &  \cellcolor{yellow!50} & \cellcolor{yellow!50} &   &  $\X$  &   &     & \cellcolor{green!50} \\ \hline
\cellcolor{yellow!50} & $\X$  & \cellcolor{yellow!50} \Circled{1}  & \cellcolor{yellow!50} &  $\X$  &   & $\X$   &    \\ \hline
\end{tabular}
\end{minipage}
\caption{On the left we have the permutation matrix associated to the reduced word $u^{(3)}$ and on the right we have the same for $u^{(2)}$, where $c = [1432657]$ as in previous examples. Comparing these illustrates the main point of Lemma \ref{lem.upper}.}
\label{fig.upper}
\end{center}
\end{figure}
\end{example}

\begin{lem}
\label{lem.lower}
Let $1 \leq k \leq r+1$ and $1 \leq j \leq d_k-1$. If $i = j + \sum_{a = 1}^{k-1}(d_a -1)$, that is, if $\prefixR{i} = d^{(k)}_j$, then the $\left(\binom{n+1}{2}-i+1\right)$\textsuperscript{th} entry of the vector $\Pi_c(X(\prefixR{i}))$ is $1$,  
and all earlier entries of $\Pi_c(\prefixR{i})$ are 0.
\end{lem}

\begin{proof}
It suffices to show that for all $1\leq i \leq \sum_{a=1}^{r+1}(d_a-1)$, we have
\begin{enumerate}[(a)]
\item $X(\prefixR{i})$ is $0$ in all entries of $\swpi$, and
\item $X(\prefixR{i})(d_k-j,d_k) = 1$ and this is the first nonzero entry in $\nepi$.
\end{enumerate}

We will prove this by induction on $i$. One can check the statements when $i = 1$ directly. Suppose the claim is true for $i-1$, that is, $X(b_{i-1})(d_k - (j-1),d_k) = 1$ and is the first nonzero entry in $\Pi_c(X(b_{i-1}))$. Note $b_{i-1} = d^{(k)}_{j-1}$.

Since $b_i$ and $b_{i-1}$ differ by a transposition, $s_{\mathcal{R}(i)}$, the difference between $X(b_{i})$ and $X(b_{i-1})$ is that two 1's and two 0's swap places in rows $\mathcal{R}(i)$ and $\mathcal{R}(i)+1$.  We have $\mathcal{R}(i)=d_k-j$, so inducting on~$j$ we can see $X(b_i)(d_k-j,d_k) = 1$, as desired.

It remains to verify that we have not moved a $1$ to an entry recorded by $\Pi_c$ with smaller index than $\binom{n+1}{2} - i$. We compute the change done by the transposition $(d_k-j,d_k-j+1)$ and discuss the two cases below. Note everything to the right of column $d_k$ remain unchanged, so we don't need to consider these columns.

Let $p$ be such that $X(b_{i})(d_k-j+1,p)  = 1$, or equivalently, $X(b_{i-1})(d_k-j,p) = 1$.
\begin{enumerate}
\item Suppose $p = d_a$ where $0\leq a < k$. In this case $k-a = d_k-j$ therefore $j = d_k-k+a$. The entry $(d_k-j,d_a)$ moves downwards when we apply the transposition. If this entry moves to another spot in $\nepi$ then it moves to a spot lower in the same column and thus is recorded later than the entry $(d_k-j+1,p)$ in $\Pi_c$. Otherwise, it moves to a spot on or below the main diagonal, so it is either not recorded by $\Pi_c$ or it is in $\swpi$. Suppose this spot is in $\swpi$. In other words, the entry $(k+1-a,d_a)$ is in $\swpi$.  Then by Proposition~\ref{prop.swpi}\ref{it.d} we have
\[
k+1-a = n+2-t-a\quad \text{ for some } 1\leq t \leq s\,.
\]
This gives us $k + t = n+1$ which contradicts the fact that $0\leq k \leq r+1$ and $r+s = n-1$.

\item If $p = u_a$ where $u_a < d_k$. In this case
\[
d_k-j + 1 = u_a + \#\{d_b > u_a \mid b \leq k\}\,.
\]
This entry is clearly below the main diagonal, so it will not be in $\nepi$ and we just need to show that it will not be in $\swpi$.

The number $\#\{d_b > u_a\}$ is $(n+1) - u_a - (s-a)$, according to Lemma~\ref{lem:content_of_interval} (recall $d_{r+1} = n+1$), so 
\[
u_a + \#\{d_b > u_a: b \leq k\} \leq u_a + ((n+1) - u_a - (s-a)) = n+1-s+a\,.
\]
Therefore, $d_k-j+1 \leq n+1-s+a$. Recall that the entries in both column $u_a$ and $\swpi$ are in columns $n+2-t+a$ for $a < t \leq s$. Since $n+1-s+a < n+2-t+a$, we conclude $(d_k-j,u_a) \notin \swpi$.
\end{enumerate}
\end{proof}

\begin{lem}
\label{lem.upper}
Let $1 \leq k \leq s$ and $1 \leq j \leq u_k -1$. If $i = \left(\sum_{a = 1}^{r+1}(d_a -1)\right) + j+  \sum_{a = s}^{k+1}(u_a-1)$, that is, if  $\prefixR{i} = u^{(k)}_j$, then $\Pi_c(X(\prefixR{i}))$ is a vector where the $\left(\binom{n+1}{2}-i+1\right)$\textsuperscript{th} entry of the vector $\Pi_c(X(\prefixR{i}))$  is $1$,  
and all earlier entries of $\Pi_c(X(\prefixR{i}))$ are zero.
\end{lem}

\begin{proof}
It suffices to show that the first $1$ in $X(b_i)$ that is selected by $\Pi_c$ comes from position
\[
\begin{cases}
    (n+2-j,c^j(u_k)), & \text{if } 1\leq j \leq m = \mu(u_k)\,,\\
    (u_k-j+m,u_k), & \text{if } m+1 \leq j \leq u_k-1\,.
 \end{cases}
\]
We induct first on $k$, working backwards from $s$ to $1$, and then on $j$, from $j =1$ to $j = u_k-1$, where we recall that $u_0^{(k)} = u^{(k+1)}_{u_{k+1}-1}$. For convenience of notation, we will set $u_{s+1} = n$ in this proof. As a base case, we know the statement holds for $u^{(s)}_0 = u^{(s+1)} = d^{(r+1)}=n+1$ from Lemma~\ref{lem.lower}.

Suppose the statement holds for $u_{j-1}^{(k)}$.  We compute the changes in the permutation matrix between
$u_{j-1}^{(k)}$ and $u_j^{(k)}$.
We obtain $u_j^{(k)}$ from $u_{j-1}^{(k)}$ by applying $s_{n+1-j}$, which exchanges the $(n+1-j)$\textsuperscript{th} row and the $(n+2-j)$\textsuperscript{th} row. By Remark~\ref{rmk:Described^(k)andu^(k)}\ref{rmk:Described^(k)andu^(k):itm:2}, this will cause a $1$ in column $q$ to move from the $(n+1-j)$\textsuperscript{th} row down to the $(n+2-j)$\textsuperscript{th} row for some $q< u_k$, and a $1$ in column $u_k$ to move up from the $(n+2-j)$\textsuperscript{th} row to the $(n+1-j)$\textsuperscript{th} row. We consider two cases regarding $j$.

\begin{enumerate}
\item Suppose $1\leq j \leq m = \mu(u_k)$. 
In this case, we claim that $q = c^j(u_{k})$. We discuss the two cases.
\begin{enumerate}
\item When $k+1 > j$, we check column $u_{k-j}$ in $X(b_{i-1}) = X(u^{(k)}_{j-1})$. By definition of $u^{(k+1)}$ and $u^{(k)}$, this column is the same as in $X(u^{(k+1)})$ at this point. According to Remark~\ref{rmk:Described^(k)andu^(k)}, there is a $1$ in row 
\[
n+2+(k-j) -(k+1)= n+1-j.
\]
Therefore in this case we have $q = u_{k-j} = c^{j}(u_{k})$.
\item When $k+1 \leq j$, we check column $d_{j-k}$ in $X(b_{i-1}) = X(u^{(k)}_{j-1})$. This column is the same as in $X(u^{(k+1)})$ at this point. According to Remark~\ref{rmk:Described^(k)andu^(k)}, there is a $1$ in row 
\[
n+2-(j-k)-(k+1) = n+1-j.
\]
Therefore in this case we also have $q = d_{j-k} = c^{j}(u_{k})$.
\end{enumerate}
In particular, the above claim indicates that $(n+2-j,q)$ is an entry in $\swpi$ that is associated to $u_k$.

Since the $1$ in column $u_k$ has not moved up to $\nepi$, the position $(n+1-j,u_k)$ is either not recorded by $\Pi_c$, or it is in $\swpi$. In the latter case, Proposition~\ref{prop.swpi}\ref{it.u} indicates it is associated to some $u_t$ with $t>k$.
Together with the inductive hypothesis, we conclude that the first $1$ in $\Pi_c(b_i)$ is indeed at $(n+2-j,c^j(u_k))$.

\item If $m = n+1-u_k < u_k-1$, then we must also consider $j$ such that $m+1 \leq j \leq u_k-1$. We then have $u_k-j+m = n+1-j$.
We need to show that $(n+2-j,q)$ is not an entry in~$\swpi$ associated to some $u_t$ with $t \leq k$.

\begin{enumerate}
\item When $q = d_a$, $X(u^{(k+1)})$ has a 1 in column $q$ and row $(n+1-a-k)$ by Remark~\ref{rmk:Described^(k)andu^(k)}.  Since column $q$ in $X(b_{i-1}) = X(u^{(k)}_{j-1})$ is the same as in $X(u^{(k+1)})$, $n+2-a-k = n+2-j$, which implies $j = a+k$. Suppose $(n+2-j, d_a) \in \swpi$ and is associated to some $u_t$ where $t\leq k$ in Definition~\ref{defn.projection}. Then by Proposition~\ref{prop.swpi} we have
\[
n+2-j = n+2-t-a\,.
\]
This means $ t = k$. However the highest row in $\swpi$ associated to $u_k$ in Definition~\ref{defn.projection} is $n+2-m > n+2-j$. This is a contradiction, so if $(n+2-j, d_a) \in \swpi$ it is associated to some $u_t$ where $t>k$ in Definition~\ref{defn.projection}.

\item When $q = u_a<u_k$, $X(u^{(k+1)})$ has a 1 in column $q$ and row $(n+1+a-k)$ by Remark~\ref{rmk:Described^(k)andu^(k)}.  Since column $q$ in $X(b_{i-1}) = X(u^{(k)}_{j-1})$ is the same as in $X(u^{(k+1)})$, $n+2+a-k = n+2-j$, which implies $j = k-a$.  Suppose $(n+2-j,u_a)\in \swpi$ determined by some $u_t$ with $t\leq k$, then by Proposition~\ref{prop.swpi} we have
\[
n+2-j = n+2-t+a\,.
\]
Again we have $t = k$. As in the previous case, this is impossible.
\end{enumerate}
So, $(n+2-j,q)$ is not an entry in $\swpi$ associated to some $u_t$ with $t \leq k$.  This means that the 1 in column $u_k$ is the first 1 to appear in $\Pi_c(X(u^{(k)}_j))$.
Since $X(u^{(k)}_{j-1})$ has a 1 in column~$u_k$ and row $n+2-j$, as $j$ increases this $1$ will move up one row at a time and it will stay the first $1$ in $\Pi_c$.
\end{enumerate}
\end{proof}

\begin{proof}[Proof of Proposition~\ref{prop.square}]
By Lemma~\ref{lem.lower} and Lemma~\ref{lem.upper}, the matrix whose columns are $\Pi_c(X(\prefixR{i}))$
for $1\leqslant i \leqslant \binom{n+1}{2}$ 
is a lower antidiagonal triangular matrix with $1$'s along the antidiagonal.  By Remark~\ref{rmk:o(bi)},  the matrix whose columns are $o(\prefixR{i})$
for $1\leqslant i \leqslant \binom{n+1}{2}$ 
is also a lower antidiagonal triangular matrix with $1$'s along the antidiagonal. The statement follows from linear algebra.
\end{proof}

\subsection{Main Results}
\label{subsec:SquareToRectangle}

Let $\mathcal{U}_c$ be the 
$\binom{n+1}{2}\times \binom{n+1}{2}$ lower-triangular matrix $\mathcal{U}_c$ from Proposition~\ref{prop.square}. The goal of this section is to prove the following theorem.

\begin{theorem}\label{thm.square.to.rectangle}
For each $c$-singleton $w$, 
we have $\mathcal{U}_c \circ \Pi_c(\PermMatrix{w}) = o(w)$.    
\end{theorem}

Let $1,2,..., \Coxeterlength(w_0)$ denote the poset elements of the heap  $\HeapDiagonalReading$ of $\DiagonalReadingWord$, as in Definition \ref{defn:heap of a reduced word}. With this notation, we can think of the order ideal $f(w)$ 
of $\HeapDiagonalReading$ 
given in  Proposition~\ref{prop:vertex-bijection} (see Remark~\ref{rem:prop:vertex-bijection})
as a subset of $\{1,2,..., \Coxeterlength(w_0)\}$.  To prove Theorem~\ref{thm.square.to.rectangle}, we first prove an important identity for $c$-singletons $w$ where $1\in f(w)$, see Lemma~\ref{lem.easy}. Lemma~\ref{lem.complex} allows us to extend this to $c$-singletons $w$ where $1\not\in f(w)$.  We then prove Theorem~\ref{thm.square.to.rectangle} by combining these cases with Proposition~\ref{prop.square}.

Recall that $D_{x}$ denotes the diagonal in $\HeapDiagonalReading$ containing $x$, and $A_{x}$ denotes the antidiagonal in $\HeapDiagonalReading$ containing ~$x$.

\begin{notation}\label{nota:contains.1}
Suppose $w$ is a $c$-singleton such that $f(w)$ contains the poset element $1$ in $\HeapDiagonalReading$. 
We construct a decreasing sequence of positive integers $a_0>a_1>\dots>a_{2k}$ (where $k\geq 0$) iteratively as follows:
\begin{enumerate}[I.]
\item Let $a_0$ be the largest number in $f(w)$.
\item Suppose $i\geq 1$, and we have 
assigned a number in $f(w)$ to $a_{2(i-1)}$ which lives in $\HeapDiagonalReading$. 
If $f(w)$ contains all the integers $1, 2, \dots, a_{2(i-1)}$, we stop and let $k=i-1$. 
Otherwise, we append two more integers $a_{2i-1}$ and $a_{2i}$ to our decreasing sequence. 
\begin{itemize} 
\item 
Let $a_{2i-1}$ be the largest integer such that $a_{2i-1}<a_{2(i-1)}$  and  $a_{2i-1}$ is \emph{not} in $f(w)$. Note that $a_{2i-1}$ is the northwest vertex of $D_{a_{2i-1}}$ which is below $D_{a_{2(i-1)}}$, the diagonal containing $a_{2(i-1)}.$ 
\item 
Let $a_{2i}$ be the largest integer such that $a_{2i}<a_{2i-1}$ and $a_{2i}$ is in $f(w)$ (we know such integer exists because the number $1$ is in $f(w)$ by assumption). 
Note that the diagonal $D_{a_{2i}}$ containing $a_{2i}$ is either equal to $D_{a_{2i-1}}$ or is 
below $D_{a_{2i-1}}$.
\end{itemize}
\end{enumerate}
\end{notation}

\begin{example}\label{ex:nota:contains.1.PartI}
Consider $H_c$ for $c=\rw{21365487}$, whose Hasse diagram is depicted in Figure~\ref{fig:ex:nota:contains.1} (left) and heap diagram is depicted in Figure \ref{fig:ex:nota:contains.1} (right). 
Let $w$ denote the $c$-singleton with $f(w)=\{1, 6, 7, 12, 13,14,20,21,27 \}$, the order ideal highlighted in round disks in the figure. 
Then the decreasing sequence defined in Notation~\ref{nota:contains.1} is
\[\begin{matrix}
&a_0, &a_1, &a_2, &a_3, &a_4, &a_5, &a_6, &a_7, &a_8 =\\
&27, &26, &21, &19, &14, &11, &7, &5, &1
\end{matrix}
\]
\end{example}

\def\myscale{.85}
\def\mycirclecolor{violet!10}
\def\mycirclesize{2ex}
\newcommand*\RectangledFilled[1]{\tikz[baseline=(char.base)]{
\node[fill=yellow!50,shape=rectangle,draw=yellow!50,inner sep=2.1pt] (char) {#1};}}
\begin{figure}
\begin{tikzpicture}[scale=\myscale]
\draw[\mycirclecolor,fill=\mycirclecolor] (1,1) circle (\mycirclesize);
\node (s2r2) at (1,1) [draw, star, very thick, densely dotted] {$1$}; 
\draw (2.5,1) node {$=a_8=x$};

\draw[\mycirclecolor,fill=\mycirclecolor] (5,1) circle (\mycirclesize);
\node (s6r2) at (5,1) {$6$}; 

\draw[\mycirclecolor,fill=\mycirclecolor] (7,1) circle (\mycirclesize);
\node (s8r2) at (7,1)  {$12$};

\node (s1r3) at (0,2) {\RectangledFilled{$2$}}; 

\node (s3r3) at (2,2) {\RectangledFilled{$3$}};
\draw (2.55,2) node {$=y'$};

\draw[\mycirclecolor,fill=\mycirclecolor] (4,2) circle (\mycirclesize);
\node (s5r3) at (4,2) [draw, star, very thick, densely dotted] {$7$}; 
\draw (4.8,2) node {$=a_6$};

\draw[\mycirclecolor,fill=\mycirclecolor] (6,2) circle (\mycirclesize);
\node (s7r3) at (6,2) {$13$}; 

\node (s2r4) at (1,3) {\RectangledFilled{$4$}}; 
\node (s4r4) at (3,3) {$8$}; 

\draw[\mycirclecolor,fill=\mycirclecolor] (5,3) circle (\mycirclesize);
\node (s6r4) at (5,3) [draw, star, very thick, densely dotted] {$14$}; 
\draw (5.9,3) node {$=a_4$};

\draw[\mycirclecolor,fill=\mycirclecolor] (7,3) circle (\mycirclesize);
\node (s8r4) at (7,3) {$20$}; 

\node (s1r5) at (0,2+2) [draw, rectangle, very thick, densely dotted] {\RectangledFilled{$5$}}; 
\draw (s1r5); 
\draw (-1.2,2+2) node {$y=a_7=$};

\node (s3r5) at (2,2+2) {$9$}; 
\node (s5r5) at (4,2+2) {$15$} ;

\draw[\mycirclecolor,fill=\mycirclecolor] (6,2+2) circle (\mycirclesize);
\node (s7r5) at (6,2+2) [draw, star, very thick, densely dotted] {$21$}; 
\draw (6.9,2+2) node {$=a_2$};

\node (s2r6) at (1,3+2) {$10$}; 
\node (s4r6) at (3,3+2) {$16$}; 
\node (s6r6) at (5,3+2) {$22$}; 

\draw[\mycirclecolor,fill=\mycirclecolor] (7,3+2) circle (\mycirclesize);
\node (s8r6) at (7,3+2) [draw, star, very thick, densely dotted] {$27$}; 
\draw (7.8,5) node {$=a_0$};

\node (s1r7) at (0,2+4) [draw, rectangle, very thick, densely dotted ]{$11$}; 
\draw (-.9,2+4) node {$a_5=$};

\node (s3r7) at (2,2+4) {$17$}; 
\node (s5r7) at (4,2+4) {$23$}; 
\node (s7r7) at (6,2+4) {$28$}; 
\node (s2r8) at (1,3+2+2) {$18$}; 
\node (s4r8) at (3,3+2+2) {$24$}; 
\node (s6r8) at (5,3+2+2) {$29$}; 
\node (s8r8) at (7,3+2+2) {$32$};

\node (s1r9) at (0,2+4+2) [draw, rectangle, very thick, densely dotted] {$19$}; 
\draw (-.9,2+4+2) node {$a_3=$};

\node (s3r9) at (2,2+4+2) {$25$}; 
\node (s5r9) at (4,2+4+2) {$30$}; 
\node (s7r9) at (6,2+4+2) {$33$}; 

\node (s2r10) at (1,3+2+2+2)  [draw, rectangle, very thick, densely dotted]  {$26$}; 
\draw (.1,3+2+2+2) node {$a_1=$};

\node (s4r10) at (3,3+2+2+2) {$31$}; 
\node (s6r10) at (5,3+2+2+2) {$34$}; 
\node (s8r10) at (7,3+2+2+2) {$36$}; 

\node (s5r11) at (4,2+4+2+2) {$35$}; 

\draw [thick, shorten <=-2pt, shorten >=-2pt] (s1r3) -- (s2r2);

\draw [thick, shorten <=-2pt, shorten >=-2pt]
(s2r2) -- (s3r3);

\draw [,thick, shorten <=-2pt, shorten >=-2pt]
(s5r3) -- (s6r2);

\draw [,thick, shorten <=-2pt, shorten >=-2pt]
(s6r2) -- (s7r3);

\draw [,thick, shorten <=-2pt, shorten >=-2pt]
(s7r3) -- (s8r2);

\draw [thick, shorten <=-2pt, shorten >=-2pt] (s1r3) -- (s2r4) ;

\draw [thick, shorten <=-2pt, shorten >=-2pt]
(s2r4) -- (s3r3);

\draw [thick, shorten <=-2pt, shorten >=-2pt]
(s3r3) -- (s4r4);

\draw [thick, shorten <=-2pt, shorten >=-2pt]
(s4r4) -- (s5r3);

\draw [thick, shorten <=-2pt, shorten >=-2pt]
(s5r3) -- (s6r4);

\draw [thick, shorten <=-2pt, shorten >=-2pt]
(s6r4) -- (s7r3);

\draw [thick, shorten <=-2pt, shorten >=-2pt]
(s7r3) -- (s8r4);

\draw [, thick, shorten <=-2pt, shorten >=-2pt] (s1r5) -- (s2r4);

\draw [, thick, shorten <=-2pt, shorten >=-2pt]
(s2r4) -- (s3r5);

\draw [, thick, shorten <=-2pt, shorten >=-2pt]
(s3r5) -- (s4r4);

\draw [, thick, shorten <=-2pt, shorten >=-2pt]
(s4r4) -- (s5r5);

\draw [, thick, shorten <=-2pt, shorten >=-2pt]
(s5r5) -- (s6r4);

\draw [, thick, shorten <=-2pt, shorten >=-2pt]
(s6r4) -- (s7r5);

\draw [, thick, shorten <=-2pt, shorten >=-2pt]
(s7r5) -- (s8r4);

\draw [, thick, shorten <=-2pt, shorten >=-2pt] (s1r5) -- (s2r6);

\draw [, thick, shorten <=-2pt, shorten >=-2pt] 
(s2r6) -- (s3r5);

\draw [, thick, shorten <=-2pt, shorten >=-2pt]
(s3r5) -- (s4r6);

\draw [, thick, shorten <=-2pt, shorten >=-2pt]
(s4r6)
-- (s5r5);

\draw [, thick, shorten <=-2pt, shorten >=-2pt]
(s5r5) -- (s6r6); 

\draw [, thick, shorten <=-2pt, shorten >=-2pt]
(s6r6) -- (s7r5); 

\draw [, thick, shorten <=-2pt, shorten >=-2pt]
(s7r5) -- (s8r6);

\draw [, thick, shorten <=-2pt, shorten >=-2pt] (s1r7) -- (s2r6);

\draw [, thick, shorten <=-2pt, shorten >=-2pt]
(s2r6)
-- (s3r7);

\draw [, thick, shorten <=-2pt, shorten >=-2pt]
(s3r7) 
-- (s4r6);

\draw [, thick, shorten <=-2pt, shorten >=-2pt]
(s4r6) 
-- (s5r7);

\draw [, thick, shorten <=-2pt, shorten >=-2pt]
(s5r7) 
-- (s6r6);

\draw [, thick, shorten <=-2pt, shorten >=-2pt]
(s6r6) 
-- (s7r7);

\draw [, thick, shorten <=-2pt, shorten >=-2pt]
(s7r7)
--
(s8r6);

\draw [, thick, shorten <=-2pt, shorten >=-2pt] (s1r7) -- (s2r8);

\draw [, thick, shorten <=-2pt, shorten >=-2pt] 
(s2r8)
-- (s3r7);

\draw [, thick, shorten <=-2pt, shorten >=-2pt] 
(s3r7) 
-- (s4r8);

\draw [, thick, shorten <=-2pt, shorten >=-2pt] 
(s4r8) 
-- (s5r7);

\draw [, thick, shorten <=-2pt, shorten >=-2pt] 
(s5r7) 
-- (s6r8);

\draw [, thick, shorten <=-2pt, shorten >=-2pt] 
(s6r8) 
-- (s7r7);

\draw [, thick, shorten <=-2pt, shorten >=-2pt] 
(s7r7) 
-- 
(s8r8);

\draw [, thick, shorten <=-2pt, shorten >=-2pt] (s1r9) -- (s2r8);

\draw [, thick, shorten <=-2pt, shorten >=-2pt]
(s2r8) -- (s3r9);

\draw [, thick, shorten <=-2pt, shorten >=-2pt]
(s3r9) -- (s4r8);

\draw [, thick, shorten <=-2pt, shorten >=-2pt]
(s4r8) -- (s5r9);

\draw [, thick, shorten <=-2pt, shorten >=-2pt]
(s5r9) -- (s6r8);

\draw [, thick, shorten <=-2pt, shorten >=-2pt]
(s6r8) -- (s7r9);

\draw [, thick, shorten <=-2pt, shorten >=-2pt] (s7r9) -- (s8r8);

\draw [, thick, shorten <=-2pt, shorten >=-2pt] (s7r9) -- 
(s8r10);

\draw [, thick, shorten <=-2pt, shorten >=-2pt] (s1r9) -- (s2r10);

\draw [, thick, shorten <=-2pt, shorten >=-2pt] 
(s2r10) 
-- (s3r9);

\draw [, thick, shorten <=-2pt, shorten >=-2pt] 
(s3r9) 
-- (s4r10);

\draw [, thick, shorten <=-2pt, shorten >=-2pt] 
(s4r10) 
-- (s5r9);

\draw [, thick, shorten <=-2pt, shorten >=-2pt] 
(s5r9) 
-- (s6r10);

\draw [, thick, shorten <=-2pt, shorten >=-2pt] 
(s6r10) 
-- (s7r9);

\draw [, thick, shorten <=-2pt, shorten >=-2pt] 
(s7r9) 
-- 
(s8r10);

\draw [, thick, shorten <=-2pt, shorten >=-2pt] (s4r10) -- (s5r11);

\draw [, thick, shorten <=-2pt, shorten >=-2pt]
(s5r11) -- (s6r10);

\end{tikzpicture}
\begin{tikzpicture}[scale=\myscale]
\node (s2r2) at (1,1) {$s_2$};
\node (s6r2) at (5,1) {$s_6$};
\node (s8r2) at (7,1) {$s_8$}; 

\node (s1r3) at (0,2) {$s_1$};
\node (s3r3) at (2,2) {$\Circled{s_3}=s_p$};
\node (s5r3) at (4,2) {$s_5$};
\node (s7r3) at (6,2) {$s_7$};
\node (s2r4) at (1,3) {$s_2$};
\node (s4r4) at (3,3) {$s_4$};
\node (s6r4) at (5,3) {$s_6$};
\node (s8r4) at (7,3) {$s_8$}; 

\node (s1r5) at (0,2+2) {$s_1$};
\node (s3r5) at (2,2+2) {$s_3$};
\node (s5r5) at (4,2+2) {$s_5$};
\node (s7r5) at (6,2+2) {$s_7$};
\node (s2r6) at (1,3+2) {$s_2$};
\node (s4r6) at (3,3+2) {$s_4$};
\node (s6r6) at (5,3+2) {$s_6$};
\node (s8r6) at (7,3+2) {$s_8$}; 

\node (s1r7) at (0,2+4) {$s_1$};
\node (s3r7) at (2,2+4) {$s_3$};
\node (s5r7) at (4,2+4) {$s_5$};
\node (s7r7) at (6,2+4) {$s_7$};
\node (s2r8) at (1,3+2+2) {$s_2$};
\node (s4r8) at (3,3+2+2) {$s_4$};
\node (s6r8) at (5,3+2+2) {$s_6$};
\node (s8r8) at (7,3+2+2) {$s_8$};

\node (s1r9) at (0,2+4+2) {$s_1$};
\node (s3r9) at (2,2+4+2) {$s_3$};
\node (s5r9) at (4,2+4+2) {$s_5$};
\node (s7r9) at (6,2+4+2) {$s_7$};

\node (s2r10) at (1,3+2+2+2) {$s_2$};
\node (s4r10) at (3,3+2+2+2) {$s_4$}; 
\node (s6r10) at (5,3+2+2+2) {$s_6$}; 
\node (s8r10) at (7,3+2+2+2) {$s_8$}; 

\node (s5r11) at (4,2+4+2+2) {$s_5$};

\draw [black, thick, shorten <=-2pt, shorten >=-2pt] (s1r3) -- (s2r2) -- (s3r3)  (s5r3) -- (s6r2) -- (s7r3) -- 
(s8r2)
node[pos=0.5, left=5mm]{};

\draw [black, thick, shorten <=-2pt, shorten >=-2pt] (s1r3) -- (s2r4) -- (s3r3) -- (s4r4) -- (s5r3) -- (s6r4) -- (s7r3) -- 
(s8r4)
node[pos=0.5, left=5mm]{};
\draw [black, thick, shorten <=-2pt, shorten >=-2pt] (s1r5) -- (s2r4) -- (s3r5) -- (s4r4) -- (s5r5) -- (s6r4) -- (s7r5) -- 
(s8r4)
node[pos=0.5, left=5mm]{};

\draw [black, thick, shorten <=-2pt, shorten >=-2pt] (s1r5) -- (s2r6) -- (s3r5) -- (s4r6) -- (s5r5) -- (s6r6) -- (s7r5) --
(s8r6)
node[pos=0.5, left=5mm]{};
\draw [black, thick, shorten <=-2pt, shorten >=-2pt] (s1r7) -- (s2r6) -- (s3r7) -- (s4r6) -- (s5r7) -- (s6r6) -- (s7r7)  --
(s8r6)
node[pos=0.5, left=5mm]{};

\draw [black, thick, shorten <=-2pt, shorten >=-2pt] (s1r7) -- (s2r8) -- (s3r7) -- (s4r8) -- (s5r7) -- (s6r8) -- (s7r7) -- 
(s8r8)
node[pos=0.5, left=5mm]{};

\draw [black, thick, shorten <=-2pt, shorten >=-2pt] (s1r9) -- (s2r8) -- (s3r9) -- (s4r8) -- (s5r9) -- (s6r8) -- (s7r9) -- 
(s8r8)
node[pos=0.5, left=5mm]{};

\draw [black, thick, shorten <=-2pt, shorten >=-2pt] (s7r9) -- 
(s8r10)
node[pos=0.5, left=5mm]{};

\draw [black, thick, shorten <=-2pt, shorten >=-2pt] (s1r9) -- (s2r10) -- (s3r9) -- (s4r10) -- (s5r9) -- (s6r10) -- (s7r9) -- 
(s8r10)
node[pos=0.5, left=5mm]{};

\draw [black, thick, shorten <=-2pt, shorten >=-2pt] (s4r10) -- (s5r11) -- (s6r10)
node[pos=0.5, left=5mm]{};
\end{tikzpicture}
\caption{Heap of $\rw{21365487}$-sorting word of the longest element $w_0$ in $A_8$, used in Examples \ref{ex:nota:contains.1.PartI}, \ref{ex:nota:contains.1.PartII}, \ref{ex:nota:contains.1.PartIII}, and \ref{ex:complex}} 
\label{fig:ex:nota:contains.1}    
\end{figure}

\begin{lemma}
\label{lem.eqn.o}
Let $w$ be a $c$-singleton such that $f(w)$ contains the poset element $1$ of $\HeapDiagonalReading$, and let $a_0 > a_1 > \dots > a_{2k}$ denote the decreasing sequence  defined in Notation \ref{nota:contains.1}. 
Then 
we have
\begin{equation}
o(w) = o(\prefixR{a_0}) - \sum_{i=1}^{k}[o(\prefixR{a_{2i-1}})-o(\prefixR{a_{2i}})]\,.
\label{eqn.o}
\end{equation}
\end{lemma}
\def\mysmall{i}
\def\mybig{j}
\begin{proof}
Note that if $\mybig>\mysmall$, then $o(\prefixR{\mybig})-o(\prefixR{\mysmall})$ is the vector with $1$'s in the positions $\mysmall+1$ through $\mybig$ (inclusive) and $0$'s elsewhere.
That is, the transpose of $o(\prefixR{\mybig})-o( \prefixR{\mysmall})$ is 
\begin{align*}
( 0 \, 0 \,  \dots  \, 0 \,  
\overbrace{
1 \, 1 \, \dots \, 1
}
^{\substack{
 \text{positions } \mysmall+1\\ \text{to } \mybig}} 
\,
0 \, 0 \, \dots \, 0 )
\end{align*}
The relation \eqref{eqn.o} 
follows from Notation~\ref{nota:contains.1}.
\end{proof}

\def\myWV{wv}

\begin{notation}\label{nota:Various_words}
Suppose $w$ is a $c$-singleton such that $f(w)$ contains the poset element $1$ in $\HeapDiagonalReading$ and $f(w)$ is not of the form $\{1,2,\ldots,x\}$. 
Define the sequence $a_0 > a_1 > \cdots > a_{2k}$ as in Notation~\ref{nota:contains.1}; necessarily, $k \geq 1$. 
\begin{itemize}
\item 
Let $x:= a_{2k}$ and $y:= a_{2k-1}$.
\item Let $v$ denote the permutation such that $\prefixR{y}=\prefixR{x} v$ with reduced word  
$\mathbf{v}\coloneqq \rw{\DiagWordSubseq{x+1}
\dots \DiagWordSubseq{y}}$ corresponding to $\{x+1,...,y \}\subset H_c$. 

\item 
Let $\mathbf{w}$ denote the reduced word of $w$ constructed by taking the subsequence $\DiagonalReadingWord$ corresponding to the order ideal $f(w)$.
\item 
Let $\wsuff$ denote the permutation such that 
$w = \prefixR{x} ~ \wsuff$, and let $\mathbf{\wsuff}$ be the reduced word of $\wsuff$ which is a 
suffix of $\mathbf{w}$.
\end{itemize}
\end{notation}

\begin{example}\label{ex:nota:contains.1.PartII}
Let $c$ and $w$ be as in Example~\ref{ex:nota:contains.1.PartI}. 
Following Notation~\ref{nota:Various_words}, we have the following. 
\begin{itemize}
\item 
$x=a_8=1$ and $y=a_7=5$ 
\item $\mathbf v=\rw{\DiagWordSubseq{x+1}...\DiagWordSubseq{y}}=\rw{\DiagWordSubseq{2}...\DiagWordSubseq{5}}=\rw{1321}$. The elements of the poset corresponding to the letters of $\mathbf{v}$ are highlighted in rectangles in Figure \ref{fig:ex:nota:contains.1}.

\item The reduced word $\mathbf{w}$ for the permutation $w$ is  \[\rw{\DiagWordSubseq{1} \cdot  \DiagWordSubseq{6}\DiagWordSubseq{7} \cdot  \DiagWordSubseq{12}\DiagWordSubseq{13}\DiagWordSubseq{14} \cdot \DiagWordSubseq{20}\DiagWordSubseq{21} \cdot \DiagWordSubseq{27}}=\rw{2 65 876 87 8}.\] 

\item Since $\prefixR{x}=\prefixR{1}=s_2$, the reduced word $\mathbf{\wsuff}$ for the permutation $\wsuff$ is  \[\rw{\DiagWordSubseq{6}\DiagWordSubseq{7} \cdot  \DiagWordSubseq{12}\DiagWordSubseq{13}\DiagWordSubseq{14} \cdot \DiagWordSubseq{20}\DiagWordSubseq{21} \cdot \DiagWordSubseq{27}}=\rw{65 876 87 8}.\] 
\end{itemize}
\end{example}

\begin{lemma}\label{lem.contains.1}
Assume the situation and notation given in 
Notation~\ref{nota:Various_words}. 
View the word $\mathbf{v}$ 
as a concatenation of maximal decreasing runs (each supported on a single diagonal). 
Let $p\geq 1$ be the integer such that  $\rw{p(p-1)\dots }$ is the rightmost maximal decreasing run of $\mathbf{v}$.

Then the following results hold. 
\begin{enumerate} [(1)]
\item 
\label{lem.contains.1.support.of.v.is.up.to.p} 
$\supp(v) \subseteq \{ 1,...,p\}$.

\item 
\label{lem.contains.1.support.of.w''}
We have $\supp(\wsuff)\subseteq \{p+2,\dots,n\}$; 
in particular,  
we have $w(j) = \prefixR{x}(j)$ for each $1 \leq j \leq p+1$.

\item \label{lem.contains.1.wv is c-singleton containing 1.itm a}
$\myWV$ is a $c$-singleton such that $f(\myWV)$ contains $1$.

\item \label{lem.contains.1.wv is c-singleton containing 1.itm b}
Furthermore, if we construct a decreasing sequence $a_0' > a_1 ' > ...$ for $\myWV$ following the iterative algorithm given in Notation \ref{nota:contains.1}, we get precisely \[a_0 > a_1 > ... > a_{2k-2},\] that is, the same sequence for $w$ but with the last two integers removed.

\end{enumerate}
\end{lemma}

\begin{proof} 
It follows from Notation~\ref{nota:Various_words} that  
$f(\prefixR{y})=f(\prefixR{x}) \sqcup \{ x+1, ..., y\}$. 
By construction, 
each of the poset elements $1,\dots, x$ is either in $D_y$ or in a diagonal below $D_y$. 
Let $y'$ denote the poset element in $\HeapDiagonalReading$ (corresponding to the simple reflection $s_p$) which is the first number in the rightmost maximal decreasing run of $\mathbf{v}$. 
Note that $D_{y'}=D_y$.

\begin{enumerate}
\item [\ref{lem.contains.1.support.of.v.is.up.to.p}]
We consider two cases. 

\begin{enumerate}[({Case} I)]
\item 
% CASE I: 
Suppose $D_x = D_y$.  
Then $\{ x+1, ..., y\}$ are in the same diagonal, and thus $\mathbf{v}$ consists of a single decreasing run $\rw{p(p-1)\dots}$. 

\item \label{lem.contains.1:caseII}
% CASE II: 
Otherwise, $D_x$ is below $D_y$. 
Then $D_y$ contains no element of $f(w)$. 
This means $y'$ is the southeast-most vertex of $D_y$ and $y'$ corresponds to the simple reflection $s_p$ in the heap diagram of $H_c$. 

We cannot have $p=1$ because then it would be impossible to have a diagonal below $D_y$ by Lemma~\ref{lem:southeast vertex of a diagonal whose label is p not n is the minimum element of Heap with label p}\ref{lem:southeast vertex of a diagonal whose label is p not n is the minimum element of Heap with label p:itm2}. 
We also cannot have $p=n$ because of the following argument: for the sake of contradiction, suppose $p=n$. Since the element $a_0$ is comparable to every poset element in $\{1,\dots,a_0 \}$ labeled $n$ (by Lemma \ref{lem:j is comparable to smaller integers i labeled n}), we have that $y' \HeapLessThan a_0$ in $\HeapDiagonalReading$. Since $f(w)$ contains $a_0$ but not $y'$, this contradicts the fact that $f(w)$ is an order ideal.

Since $y'$ is the southeast vertex of a diagonal in $\HeapDiagonalReading$ and since $p<n$, Lemma \ref{lem:southeast vertex of a diagonal whose label is p not n is the minimum element of Heap with label p}\ref{lem:southeast vertex of a diagonal whose label is p not n is the minimum element of Heap with label p:itm1} tells us that  $y'$ is the minimum element of $\HeapDiagonalReading$ with label $p$.

If $\mathbf{v}$ has other decreasing runs, they correspond to diagonals below $D_{y'}$, so by Lemma \ref{lem:southeast vertex of a diagonal whose label is p not n is the minimum element of Heap with label p}\ref{lem:southeast vertex of a diagonal whose label is p not n is the minimum element of Heap with label p:itm2} they consist of letters smaller than $p$. 
Therefore, $\supp(v) \subseteq \{ 1,...,p\}$. 
\end{enumerate}

\smallskip

\item [\ref{lem.contains.1.support.of.w''}] 
First, we will prove that, 
for any $z$ in $\HeapDiagonalReading$ which corresponds to a letter in $\mathbf{\wsuff}$, the vertex $z$ is in a diagonal above $D_y$ and in an antidiagonal below $A_{y'}$.

Let $z$ be such a poset element in $\HeapDiagonalReading$. 
We know that $z$ is in a diagonal above 
$D_y$ by definition of the algorithm given in Notation \ref{nota:contains.1}. 

To show that $z$ is in an antidiagonal below 
$A_{y'}$, suppose otherwise. Then $z$ is in either in $A_{y'}$ or in an antidiagonal above $A_{y'}$. 
So there must be some element $y''\in \{ y', ..., y\}$ in $D_y$ such that $y''<z$ for.  This contradicts the fact that $f(w)$ is an order ideal, since $y''\notin f(w)$ and $z\in f(w)$.

Since $D_{y'}=D_y$, we have that $z$ is in a diagonal above $D_{y'}$ and $z$ is in an antidiagonal below $A_{y'}$. 
Since $y'$ is a poset element of $\HeapDiagonalReading$ with label $p$, Lemma \ref{lem:if z is in a diagonal to the right and in an antidiagonal to the right} tells us that $z$ has label in $\{p+2,...,n\}$. 
Thus the letters in $\mathbf{\wsuff}$ belong to the set $\{p+2,\dots,n\}$.  
In particular, since $w=\prefixR{x}\wsuff$, 
this implies $w(j) = \prefixR{x}(j)$ for each $1 \leq j \leq p+1$. 

\smallskip

\item[\ref{lem.contains.1.wv is c-singleton containing 1.itm a}]
Observe that
$f(\prefixR{x})\sqcup \{ x+1,...,y\}$ is an order ideal because it is equal to $f(\prefixR{y})$. 
Furthermore, 
none of the newly added poset elements $x+1, ..., y$ is related by a covering relation to any poset element $z$ corresponding to a letter in $\mathbf{\wsuff}$ because 
$\supp(v) \subset \{ 1, ..., p\}$ while 
$\supp(\wsuff) \subset \{ p+2, ..., n\}$, due to parts \ref{lem.contains.1.support.of.v.is.up.to.p} and \ref{lem.contains.1.support.of.w''}, respectively. 
Since $f(w)$ is an order ideal consisting of $f(\prefixR{x})$ and the elements corresponding to $\mathbf{\wsuff}$, we can conclude that 
$f(w) \sqcup \{ x+1, ...,y\}$ 
is an order ideal; it contains $1$ because $f(w)$ does.  
Hence $\myWV$ 
is a $c$-singleton containing $1$.

\item[\ref{lem.contains.1.wv is c-singleton containing 1.itm b}]
Finally, since $f(w)$ and $f(\myWV)$ are identical 
except that \begin{align*}
f(\myWV)-f(w) &= 
\{ x+1, ..., y\}\\
&=
\{ (a_{2k})+1, ..., a_{2k-1}\},
\end{align*}
every step in the iteration produces the same $a_0 > a_1 > ... > a_{2(k-1)}$.
Since $f(\myWV)$ contains all poset elements $1,...,y=a_{2k-1}$, and also all poset elements $y+1, ... , a_{2(k-1)}$, the algorithm tells us to that the last entry of the sequence for $f(\myWV)$ is $a_{2(k-1)}$. 
\end{enumerate}
\end{proof}

\begin{example}
\label{ex:nota:contains.1.PartIII}
Continuing with Examples \ref{ex:nota:contains.1.PartI} and \ref{ex:nota:contains.1.PartII}, Lemma~\ref{lem.contains.1} tells us that we have the following. 
\begin{enumerate}
\item [\ref{lem.contains.1.support.of.v.is.up.to.p}]
The rightmost maximal decreasing run of $\mathbf v$ is $\rw{321}$, so $p=3$. 
The poset element $y'$ from the proof of Lemma~\ref{lem.contains.1} is also equal to the number $3$; in general $p$ and $y'$ need not be equal.

\item [\ref{lem.contains.1.support.of.w''}]
We see that $\supp(\wsuff)\subset \{p+2,...,n\}=\{5,6,7,8 \}$. 
\item [\ref{lem.contains.1.wv is c-singleton containing 1.itm a}]
We have 
$f(wv)=f(w) \sqcup \{ 2,3,4,5 \}$, all the highlighted vertices (in both round disks and rectangles). 
The permutation $\myWV$ is 
a $c$-singleton containing the poset element~$1$. 
\item[\ref{lem.contains.1.wv is c-singleton containing 1.itm b}] The decreasing sequence $a_0 > ... > a_6$ for $\myWV$ is $27 > 26 > 21 > 19 > 14 > 11 > 7$. 
\end{enumerate}    
\end{example}

\begin{lemma}\label{lem.easy}
Let $w$ be a $c$-singleton such that $f(w)$ contains the poset element $1$ in $\HeapDiagonalReading$. Then 
\begin{equation}
\PermMatrix{w}
= \PermMatrix{\prefixR{a_0}} - 
\sum_{i=1}^{k} [ \PermMatrix{\prefixR{a_{2i-1}}} -\PermMatrix{\prefixR{a_{2i}}}
]\,
\label{eqn.m}
\end{equation}
where 
$a_0 > a_1 > ... > a_{2k}$  are as defined in Notation \ref{nota:contains.1}. 
\end{lemma}

\begin{proof}
We prove this by induction on $k$. When $k = 0$, we have $w = \prefixR{a_0}$, and the result holds.

Suppose the equation holds for all $c$-singletons $w$ such that $f(w)$ contains $1$ for some $k \geq 0$.
Let $w$ be a $c$-singleton such that $f(w)$ contains $1$, and its decreasing sequence of integers $a_0 > a_1 > ... > a_{2(k+1)}$  are as defined in Notation \ref{nota:contains.1}. 
We need to show that 
\begin{equation}
\label{eq: conclusion X(w) for k+1}
\PermMatrix{w} = 
\PermMatrix{\prefixR{a_0}} - \sum_{i=1}^{k+1} [\PermMatrix{\prefixR{a_{2i-1}}} -\PermMatrix{\prefixR{a_{2i}}}
]\,.
\end{equation}

Let $v$ be the permutation such that 
$\prefixR{a_{2k+1}} = \prefixR{a_{2k+2}} v$, and  
let $w' = w v$. 
By Lemma \ref{lem.contains.1}\ref{lem.contains.1.wv is c-singleton containing 1.itm a} and~\ref{lem.contains.1.wv is c-singleton containing 1.itm b}, 
$w'$ is a $c$-singleton such that $f(w')$ contains $1$ and the decreasing sequence of integers for $w'$  is 
$a_0 > a_1 > ... > a_{2k}$.

By the inductive hypothesis, Equation~\eqref{eqn.m} holds for $w'$:
\begin{equation}
\label{eq:X(wv)}
\PermMatrix{w'} = \PermMatrix{\prefixR{a_0}} - \sum_{i=1}^{k}[\PermMatrix{\prefixR{a_{2i-1}}}-\PermMatrix{\prefixR{a_{2i}}}]\,.
\end{equation}

We now compare the four matrices 
$\PermMatrix{w'}, \PermMatrix{w}, \PermMatrix{\prefixR{2k+1}}$, and  $\PermMatrix{\prefixR{2k+2}}$. 
Note that there exists $p$ such that the following holds: 

\begin{itemize} 
\item 
$\supp(v) \subset \{ 1, \dots, p\}$, by Lemma \ref{lem.contains.1}\ref{lem.contains.1.support.of.v.is.up.to.p}, 
and 
\item 
$w(j)=\prefixR{2k+2} (j)$ for each 
$j=1,\dots,p+1$, by Lemma \ref{lem.contains.1}\ref{lem.contains.1.support.of.w''}.
\end{itemize}

Since $\supp(v)\subset \{ 1, ..., p \}$, 
the permutation $v$ sends the set $\{ 1, ..., p, p+1 \}$ to itself, so 
\[w'(j)=w(v(j))=\prefixR{2k+2}(v (j)) = \prefixR{2k+1}(j)   \text{ for each 
$j=1,\dots,p+1$}. 
\]
Thus, the first $p+1$ rows of $\PermMatrix{w}$ and $\PermMatrix{\prefixR{2k+2}}$ are the same, and the first $p+1$ rows of $\PermMatrix{w'}$ and $\PermMatrix{\prefixR{2k+1}}$ are the same.  This means the first $p+1$ rows of $\PermMatrix{w'}-\PermMatrix{w}$ and $\PermMatrix{\prefixR{2k+1}}-\PermMatrix{\prefixR{2k+2}}$ are the same.

Furthermore, since $\supp(v) \subset \{ 1,\dots, p\}$, 
we have \[\text{$w'(j)=w(j)$ and $\prefixR{2k+1}(j)=\prefixR{2k+2}(j)$ for each $j \geq p+2$.} \]  
So, rows $p+2$ through $n+1$ of the matrices $\PermMatrix{w'}-\PermMatrix{w}$ and $\PermMatrix{\prefixR{2k+1}}-\PermMatrix{\prefixR{2k+2}}$ are all $0$'s.
Therefore, we have
\begin{equation*}
\PermMatrix{w'} - \PermMatrix{w} = \PermMatrix{\prefixR{2k+1}} - \PermMatrix{\prefixR{2k+2}}.
\end{equation*}

Then 
\begin{align*}
\PermMatrix{w} &= 
\PermMatrix{w'} - [ \PermMatrix{w'} - \PermMatrix{w}] 
\\
&=
\PermMatrix{w'} - [ \PermMatrix{\prefixR{2k+1}} - \PermMatrix{\prefixR{2k+2}} ]\\
&=
\PermMatrix{\prefixR{a_0}} - \left( \sum_{i=1}^{k}[\PermMatrix{\prefixR{a_{2i-1}}}-\PermMatrix{\prefixR{a_{2i}}}] \right) 
- 
[ \PermMatrix{\prefixR{2k+1}} - \PermMatrix{\prefixR{2k+2}} ]
\text{ by \eqref{eq:X(wv)}}
\\
&=
\PermMatrix{\prefixR{a_0}} - \sum_{i=1}^{k+1}[\PermMatrix{\prefixR{a_{2i-1}}}-\PermMatrix{\prefixR{a_{2i}}}], \,
\end{align*}
completing our argument for \eqref{eq: conclusion X(w) for k+1}.
\end{proof}

\def\myLabelFirstElt{q}
\begin{lemma}\label{lem.complex}
Let $w$ be a $c$-singleton such that the order ideal $f(w)$ does \emph{not} contain the poset element $1$ in $\HeapDiagonalReading$.  
Let $\myLabelFirstElt:= \DiagWordSubseq{1}$, so that $\prefixR{1} = s_\myLabelFirstElt$. 
Then the following holds.
\begin{enumerate}[(1)]
\item \label{itm:lem.complex:1}
$\supp(w) \subset \{\myLabelFirstElt+2,\dots, n\}$, that is,  $w(j) = j$ for $1\leq j \leq \myLabelFirstElt+1$.
\item  \label{itm:lem.complex:2}
Let $w' = ws_\myLabelFirstElt$. Then $w'$ is a $c$-singleton and $f(w')$ contains the poset element $1$ of $\HeapDiagonalReading$.
\item 
\label{itm:o for w}
Let the decreasing sequence $a_0 > a_1 > \dots > a_{2k}$ be the sequence for $w'$ defined in Notation~\ref{nota:contains.1}. Then
\begin{equation*}
o(w) = o(\prefixR{a_0}) - \sum_{i=1}^{k}[o(\prefixR{a_{2i-1}})-o(\prefixR{a_{2i}})] - o(\prefixR{1})\,.
\end{equation*}
\item \label{itm:lem.complex:Two matrices only differ along the main diagonal}
The two matrices $\PermMatrix{w}$ and 
\begin{equation*}
\PermMatrix{\prefixR{a_0}} - \sum_{i=1}^{k}[\PermMatrix{\prefixR{a_{2i-1}}}-\PermMatrix{\prefixR{a_{2i}}}] - \PermMatrix{\prefixR{1}}
\end{equation*}
only differ on the main diagonal.
\end{enumerate}
\end{lemma}

\begin{proof}

First we prove \ref{itm:lem.complex:1}. Since the order ideal $f(w)$ does not contain the poset element $1$ in $\HeapDiagonalReading$, every element of $f(w)$ is in a diagonal above $D_1$ and in an antidiagonal below $A_1$. Since the label of poset element 1 is $q$, by Lemma~\ref{lem:if z is in a diagonal to the right and in an antidiagonal to the right} the label of every element of $f(w)$ is in $\{\myLabelFirstElt+2,\dots, n\}$.

Claim \ref{itm:lem.complex:2} is immediate. 
For \ref{itm:o for w}, we have
$$o(w') = o(\prefixR{a_0}) - \sum_{i=1}^{k}[o(\prefixR{a_{2i-1}})-o(\prefixR{a_{2i}})]$$
by Lemma~\ref{lem.eqn.o}.  Then the claim follows from the fact that $o(w') = o(w) + o(\prefixR{1})$.

It remains to prove part \ref{itm:lem.complex:Two matrices only differ along the main diagonal}. 
We subtract the two matrices: 
\begin{align}
\label{eq:differ at the main diagonal}
&\PermMatrix{w} - \left[\PermMatrix{\prefixR{a_0}} - \sum_{i=1}^{k}[\PermMatrix{\prefixR{a_{2i-1}}}-\PermMatrix{\prefixR{a_{2i}}}] - \PermMatrix{\prefixR{1}}\right]\\
\nonumber
=& \PermMatrix{w} - \left[ \PermMatrix{w'} - \PermMatrix{\prefixR{1}} \right] \text{ by Lemma \ref{lem.easy}, since $f(w')$ contains the poset element 1}
\\
\nonumber
=& \PermMatrix{\prefixR{1}}-\PermMatrix{w'}+\PermMatrix{w} 
\end{align}

Since $w' = ws_\myLabelFirstElt$, and since $\myLabelFirstElt$ and $\myLabelFirstElt+1$ are fixed points of $w$, 
we have 
\begin{align*}
w'(\myLabelFirstElt)&=\myLabelFirstElt+1, 
&
w(\myLabelFirstElt) &= \myLabelFirstElt, 
\\
w'(\myLabelFirstElt+1)&=\myLabelFirstElt,
&
w(\myLabelFirstElt+1)&=\myLabelFirstElt+1,
\end{align*}
$$w'(j)=w(j)  \text{ for each 
$j \neq \myLabelFirstElt, \myLabelFirstElt+1$}.$$
This means $\PermMatrix{w'}-\PermMatrix{w}$ is all zeros except for the $2\times 2$ square along the diagonal in rows $\myLabelFirstElt$ and $\myLabelFirstElt+1$.  This square has $-1$'s on the diagonal and $1$'s as the other two elements.

Since $\prefixR{1}=s_\myLabelFirstElt$, the matrix $\PermMatrix{\prefixR{1}}$ is  the identity matrix with the rows $\myLabelFirstElt$ and $\myLabelFirstElt+1$ swapped. Subtracting $\PermMatrix{w'}-\PermMatrix{w}$ from $\PermMatrix{\prefixR{1}}$, we get that $\PermMatrix{\prefixR{1}}-\PermMatrix{w'}+\PermMatrix{w}$ is the identity matrix.  Therefore, the matrix \eqref{eq:differ at the main diagonal} has 0's everywhere except along the main diagonal.

\end{proof}

\begin{example}\label{ex:complex} 
As in Examples~\ref{ex:nota:contains.1.PartI}, \ref{ex:nota:contains.1.PartII}, and \ref{ex:nota:contains.1.PartIII}, consider $\HeapDiagonalReading$ for $c=\rw{21365487}$.
Its heap diagram is given in Figure \ref{fig:ex:nota:contains.1} (right).
The $c$-singleton $w$ with reduced word $\rw{65 876 87 8}$ and one-line notation $123479865$ satisfies Lemma \ref{lem.complex} because it does not contain the poset element $1$ in $\HeapDiagonalReading$. 
In this example, $\myLabelFirstElt = 2$, so $b_1 = s_2$; and $w' = \rw{65 876 87 8 2}$.  
We see that

\begin{align*}
\PermMatrix{b_1} - \PermMatrix{w'} + \PermMatrix{w} 
&= 
\begingroup % keep the change local
\setlength\arraycolsep{2.5pt}
\begin{pmatrix}
1 & 0 & 0 & 0 & 0 & 0 & 0 & 0 & 0 \\
0 & 0 & 1 & 0 & 0 & 0 & 0 & 0 & 0 \\
0 & 1 & 0 & 0 & 0 & 0 & 0 & 0 & 0 \\
0 & 0 & 0 & 1 & 0 & 0 & 0 & 0 & 0 \\
0 & 0 & 0 & 0 & 1 & 0 & 0 & 0 & 0 \\
0 & 0 & 0 & 0 & 0 & 1 & 0 & 0 & 0 \\
0 & 0 & 0 & 0 & 0 & 0 & 1 & 0 & 0 \\
0 & 0 & 0 & 0 & 0 & 0 & 0 & 1 & 0 \\
0 & 0 & 0 & 0 & 0 & 0 & 0 & 0 & 1
\end{pmatrix}
\endgroup
-
\begingroup % keep the change local
\setlength\arraycolsep{2.5pt}
\begin{pmatrix}
1 & 0 & 0 & 0 & 0 & 0 & 0 & 0 & 0 \\
0 & 0 & 1 & 0 & 0 & 0 & 0 & 0 & 0 \\
0 & 1 & 0 & 0 & 0 & 0 & 0 & 0 & 0 \\
0 & 0 & 0 & 1 & 0 & 0 & 0 & 0 & 0 \\
0 & 0 & 0 & 0 & 0 & 0 & 1 & 0 & 0 \\
0 & 0 & 0 & 0 & 0 & 0 & 0 & 0 & 1 \\
0 & 0 & 0 & 0 & 0 & 0 & 0 & 1 & 0 \\
0 & 0 & 0 & 0 & 0 & 1 & 0 & 0 & 0 \\
0 & 0 & 0 & 0 & 1 & 0 & 0 & 0 & 0
\end{pmatrix} 
\endgroup
+
\begingroup % keep the change local
\setlength\arraycolsep{2.5pt}
\begin{pmatrix}
1 & 0 & 0 & 0 & 0 & 0 & 0 & 0 & 0 \\
0 & 1 & 0 & 0 & 0 & 0 & 0 & 0 & 0 \\
0 & 0 & 1 & 0 & 0 & 0 & 0 & 0 & 0 \\
0 & 0 & 0 & 1 & 0 & 0 & 0 & 0 & 0 \\
0 & 0 & 0 & 0 & 0 & 0 & 1 & 0 & 0 \\
0 & 0 & 0 & 0 & 0 & 0 & 0 & 0 & 1 \\
0 & 0 & 0 & 0 & 0 & 0 & 0 & 1 & 0 \\
0 & 0 & 0 & 0 & 0 & 1 & 0 & 0 & 0 \\
0 & 0 & 0 & 0 & 1 & 0 & 0 & 0 & 0
\end{pmatrix}
\endgroup
\\
&=
\begingroup % keep the change local
\setlength\arraycolsep{2.5pt}
\begin{pmatrix}
1 & 0 & 0 & 0 & 0 & 0 & 0 & 0 & 0 \\
0 & 1 & 0 & 0 & 0 & 0 & 0 & 0 & 0 \\
0 & 0 & 1 & 0 & 0 & 0 & 0 & 0 & 0 \\
0 & 0 & 0 & 1 & 0 & 0 & 0 & 0 & 0 \\
0 & 0 & 0 & 0 & 1 & 0 & 0 & 0 & 0 \\
0 & 0 & 0 & 0 & 0 & 1 & 0 & 0 & 0 \\
0 & 0 & 0 & 0 & 0 & 0 & 1 & 0 & 0 \\
0 & 0 & 0 & 0 & 0 & 0 & 0 & 1 & 0 \\
0 & 0 & 0 & 0 & 0 & 0 & 0 & 0 & 1
\end{pmatrix}
\endgroup
\end{align*}
which is the identity matrix.
\end{example}

\begin{proof}[Proof of Theorem~\ref{thm.square.to.rectangle}]
Let $w$ be any $c$-singleton. 
If $f(w)$ contains the element $1$ of the poset $H_c$,  
let the sequence $a_0 > a_1 > \dots > a_{2k}$ be as defined in Notation \ref{nota:contains.1}.
Then, by Lemma~\ref{lem.easy}, we have
\begin{equation*}
\PermMatrix{w} = 
\PermMatrix{\prefixR{a_0}} - 
\sum_{i=1}^{k}[\PermMatrix{\prefixR{a_{2i-1}}} -\PermMatrix{\prefixR{a_{2i}}}
]\,.
\end{equation*}
We apply $U_c \circ \Pi_c$ to both sides of the above equation: 
\begin{align*}
U_c \circ \Pi_c (\PermMatrix{w})  
&= U_c \circ \Pi_c (\PermMatrix{\prefixR{a_0}}) - \sum_{i=1}^{k}[U_c \circ \Pi_c (\PermMatrix{\prefixR{a_{2i-1}}})
-
U_c \circ \Pi_c (\PermMatrix{\prefixR{a_{2i}}}) ]
\\ 
&=
o(\prefixR{a_0}) - \sum_{i=1}^{k}[  o(\prefixR{a_{2i-1}})-
o(\prefixR{a_{2i}}) ] \text{ by Proposition~\ref{prop.square}}
\\
&= o(w) \text{ by Lemma~\ref{lem.eqn.o}}
\end{align*}

Otherwise, suppose $f(w)$ does not contain the element $1$ of the poset $H_c$. 
Then, by Lemma~\ref{lem.complex}\ref{itm:lem.complex:Two matrices only differ along the main diagonal} and the fact that $\Pi_c$ does not take entries in the main diagonal, the statement is also true.
\end{proof}

We will now prove our main result, which is about $\heap(\sort(w_0))$.
Recall from Proposition~\ref{prop:diagonal reading word} that $\heap(\sort(w_0))$ and $\HeapDiagonalReading$
have the same heap diagram (see Remark~\ref{rem:prop:vertex-bijection}).  

\begin{theorem}\label{thm:main}
The $c$-Birkhoff polytope $\bir(c)$ is unimodularly equivalent to the order polytope $\ord(H)$ where $H=\heap(\sort(w_0))$.
\end{theorem}
\begin{proof}
The projection $\Pi_c$ is injective and preserves lattice points (Theorem~\ref{thm.lattice.preserving.projection}), so it is a unimodular transformation on $\bir(c)$.  
By Proposition~\ref{prop.square}, the transformation  $\mathcal{U}_c$ has determinant $1$. 
Thus $\mathcal{U}_c \circ \Pi_c$ is a unimodular transformation on $\bir(c)$.  From Theorem~\ref{thm.square.to.rectangle} we know $\mathcal{U}_c \circ \Pi_c$ sends vertices of $\bir(c)$ to vertices of $\ord(H)$.  
Since the vertices of $\bir(c)$ are the $c$-singleton permutation matrices (Remark~\ref{rem:vertices of birk c}), there are the same number of vertices in $\bir(c)$ and $\ord(H)$ by Remark~\ref{rmk:vertex-bij}.  
All points of $\ord(H)$ are in the convex hull of its vertices, so we have that $\ord(H)$ is the image of $\bir(c)$ under $\mathcal{U}_c \circ \Pi_c$.  
Therefore, $\mathcal{U}_c \circ \Pi_c$ is a unimodular transformation which sends $\bir(c)$ to $\ord(H)$.
\end{proof}

If $\text{dim}(\ord)$ denotes the dimension of an integral polytope $\ord$ and $\text{Vol}(\ord)$ denotes the usual relative volume of $\ord$,  the \emph{normalized volume} of $\ord$ is equal to $\text{dim}(\ord)! ~ \text{Vol}(\ord)$. The following corollary recovers, and generalizes, a result of Davis and Sagan in \cite{DS18}.

\begin{corollary}\label{cor:volume}
The normalized volume of the $c$-Birkhoff polytope counts the following:
\begin{enumerate}[(1)]
\item 
linear extensions of $\Heapwnot$
\item \label{cor:volume:itm2}
reduced words in the commutation class of $\sort(w_0)$
\item 
maximal chains in $\LatticeC$ 
\item 
longest chains in the  $c$-Cambrian lattice
\item \label{cor:volume:itm5}
maximal chains in the lattice of permutations which avoid the four patterns $31\underline{2}$,  
$\overline{2}31$, $13\underline{2}$, and $\overline{2}13$, as a sublattice of the weak order on the symmetric group $A_n$

\item \label{cor:volume:itm6}
ways to add indecomposable projective representations
and perform the knitting algorithm for constructing the Auslander--Reiten quiver of the quiver $Q(c)$

\end{enumerate}
\end{corollary}
\begin{proof} 
Let $H=\Heapwnot$. Since $\bir(c)$ and $\ord(H)$ are unimodularly equivalent, they have the same volume. 
By 
Theorem~\ref{thm:BasicFactsOrderPolytope}\ref{itm:thm:BasicFactsOrderPolytope:volume}, the normalized volume of $\ord(H)$ is equal to the number of linear extensions of $H$.

By Lemma~\ref{lem:threesets}, each of the sets \ref{cor:volume:itm2} through \ref{cor:volume:itm5}
is in bijection with the linear extensions of $H$. 
Furthermore, by Remark~\ref{rem:combinatorial AR quiver}, set \ref{cor:volume:itm6} is also in bijection with the linear extensions of $H$. 

\end{proof} 

\begin{remark}\label{rmk:RelationsAreEverything}
In Corollary~\ref{cor:Affc has dimension at most n+1 choose 2}, we gave an upper bound of $\binom{n+1}{2}$ on the dimension of $\aff(c)$.  
A consequence of Theorem~\ref{thm:main} is that $\bir(c)$ and $\ord(\Heapwnot)$ have the same dimension.
Since the latter is full-dimensional, we know this dimension is $\binom{n+1}{2}$. This implies that $\aff(c)$ is also $\binom{n+1}{2}$-dimensional. In particular, the relations in Proposition~\ref{prop:IndependenceOfRelations} are a maximal set of independent relations on $\aff(c)$. 
\end{remark}

\section{Examples}
\label{Sec:ex}
\subsection{Tamari orientation}
\label{subsec.tamari}
We consider the Tamari Coxeter element $c = s_1s_2 \cdots s_n$ throughout this subsection. 
The statement of Theorem \ref{thm:main} in this case gives an affirmative answer to Question 5.1 of Davis and Sagan's paper \cite{DS18}.

The Tamari case is special as there are no upper-barred numbers. 
This means that, by Proposition~\ref{prop:pattern avoidance characterization type A}, the $c$-singletons in the Tamari case are exactly $132$ and $312$ avoiding permutations. Note that there are $2^n$ permutations in $A_n$ which avoid $132$ and $312$ (see \cite{simion1985restricted}), hence there are $2^n$ vertices of the $c$-Birkhoff polytope and $2^n$ order ideals of $\Heapwnot$. Applying Algorithm~\ref{alg:algorithm for heap using diagonal reading word} in this case produces a heap as in Figure \ref{fig:heapTamari}, which shows the heap for $A_4$ on the left and the heap for general $A_n$ on the right.

The relations on $\aff(c)$ are simpler to describe in the Tamari case. The statement of  Proposition~\ref{prop:ZeroRelationsOnMatrix} reduces to the following: for all $X\in\aff(c)$, the entries strictly below both the main diagonal and main antidiagonal are $0$. Figure \ref{fig:ZerosAndPicTamari} shows the case of $A_7$; the entries that are always 0 are depicted by red $\X$'s.

\begin{figure}
\centering
\begin{tabular}{|c|c|c|c|c|c|c|c|}
\hline
 \hspace{0.25cm} & $28$ & $26$  & $23$ & $19$ & $14$ & $8$ & $1$ \\ \hline
  &    & $27$  & $24$ & $20$ & $15$ & $9$ & $2$ \\ \hline
  &    &    & $25$ & $21$ & $16$ & $10$ & $3$ \\ \hline
  &    &    &   & $22$ & $17$ & $11$ & $4$ \\ \hline
  &    &    &   &    & $18$ & $12$ & $5$ \\ \hline
  &   &    &  $\X$ &  $\X$  &   & $13$ & $6$ \\ \hline
 &  &  $\X$&  $\X$ &  $\X$  & $\X$  &    & $7$ \\ \hline
 & $\X$  &  $\X$   &  $\X$  &  $\X$  &  $\X$   & $\X$   &    \\ \hline
\end{tabular}
\caption{For $c = [1234567]$ mark with red $\textcolor{red}{\X}$'s the entries which are guaranteed to be zero by  Proposition~\ref{prop:ZeroRelationsOnMatrix}, and we show the projection $\Pi_c$ with (black) numbers.}\label{fig:ZerosAndPicTamari}
\end{figure}

The function $\nulabel_c$ reduces to $\nulabel_c(i) = i-1$ in the Tamari case, simplifying Lemma \ref{lem:first_entries_distinct_equiv_classes}, which allows Theorems \ref{thm:first_half_relations} and \ref{thm:second_half_relations} to be simplified as well.

\begin{corollary}\label{cor:SummingRelationsTamari}
Let $c = s_1s_2 \cdots s_n$ and let $w$ be a $c$-singleton. Then, for each $y \in [n+1]$ and $0 \leq z \leq y-1$, there is exactly one value in $\{w(1),\ldots,w(y)\}$ which is equivalent to $z$ modulo $y$.  

Consequently, for any point $X = (X(i,j)) \in \aff(c)$, we have \[
\sum_{j \equiv z} \sum_{i=1}^y X(i,j) = 1
\]
where our equivalence in the first sum is modulo $y$. 
\end{corollary}
\begin{proof}
When $y \leq \frac{n+2}{2}$, the first statement follows from Lemma \ref{lem:first_entries_distinct_equiv_classes}, so suppose $y > \frac{n+2}{2}$. We can use the same argument as in part 1, case i of the proof of Lemma \ref{lem:first_entries_distinct_equiv_classes} to show we cannot have distinct values $1 \leq a < b \leq y$ such that $w(a) \equiv w(b) \pmod{y}$ even when $y > \frac{n+2}{2}$. Then, by the pigeonhole principle, for each $z$, there is one $i \in [y]$ such that $w(i) \equiv z \pmod{y}$. The final statement of the corollary again follows from the fact that this statement is true for all generators of $\aff(c)$. 
\end{proof}

The projection $\Pi_c$ is also simple to describe in this case. The entries chosen are exactly those strictly above the main diagonal, with order given by reading the columns from right to left, and in each column reading from top to bottom. See Figure \ref{fig:ZerosAndPicTamari} for the example of $A_7$.

We conclude by giving an example of the unimodular transformation from the projection of $\bir(c)$ to $\ord(\HeapDiagonalReading)$ in the Tamari case for $A_3$. The leftmost matrix is $\mathcal{U}_c$ and the middle is the result of applying $\Pi_c$ to the vertices of $\bir(c)$. The column vectors to the left of the vertical line are the projections of the vectors $\PermMatrix{b_i}$ for $1 \leq i \leq 6$.  The vectors to the right of the vertical line are the projections of the identity permutation, Id, and $s_1s_2s_3$, the two $c$-singletons which are not non-trivial prefixes of~$\DiagonalReadingWord$. In the rightmost matrix, we have the indicator vectors of $\HeapDiagonalReading=\heap(\DiagonalReadingWord) = \heap([121321])$, where again to the left of the vertical line we have all the vectors of the form $o(b_i)$, and the vectors to the right of the vertical line are $o(\text{Id})$ and  $o(s_1s_2s_3)$.

\[
\begin{pmatrix}
1 & 0 & 0 & 0 & 0 & 0\\
1 & 1 & 0 & 0 & 0 & 0\\
1 & 1 & 1 & 0 & 0 & 0\\
1 & 0 & 0 & 1 & 0 & 0\\
0 & 0 & 0 & 1 & 1 & 0\\
1 & 0 & 0 & 1 & 0 & 1
\end{pmatrix}
\begin{pmatrix}
 0 & 0 & 0 & 0 & 0 & 1 &\bigm|& 0 & 0 \\
0 & 0 & 0 & 0 & 1 & 0 &\bigm|& 0 & 0 \\
 0 & 0 & 0 & 1 & 0 & 0 &\bigm|& 0 & 1 \\
 0 & 0 & 1 & 1 & 1 & 0 &\bigm|& 0 & 0 \\
 0 & 1 & 0 & 0 & 0 & 1 &\bigm|& 0 & 1 \\
 1 & 1 & 0 & 0 & 0 & 0 &\bigm|& 0 & 1
\end{pmatrix}
=
\begin{pmatrix}
 0 & 0 & 0 & 0 & 0 & 1 &\bigm|& 0 & 0 \\
 0 & 0 & 0 & 0 & 1 & 1 &\bigm|& 0 & 0 \\
 0 & 0 & 0 & 1 & 1 & 1 &\bigm|& 0 & 1 \\
 0 & 0 & 1 & 1 & 1 & 1 &\bigm|& 0 & 0 \\
 0 & 1 & 1 & 1 & 1 & 1 &\bigm|& 0 & 1 \\
 1 & 1 & 1 & 1 & 1 & 1 &\bigm|& 0 & 1 
\end{pmatrix}\,.
\]

\subsection{Bipartite orientation}
\label{subsec.bipartite}
Let $c$ be the Coxeter element where the odd-indexed simple transpositions all appear before the even-indexed transpositions throughout this section.  That is, if $n$ is odd let $c = s_1s_3 \cdots s_n s_2s_4 \cdots s_{n-1}$ and if $n$ is even let $c = s_1s_3 \cdots s_{n-1} s_2 s_4 \cdots s_n$.  We refer to these as \emph{bipartite} Coxeter elements as for each $n$ the heap $\heap(c)$ consists of only maximal and minimal elements. 
 Figure \ref{fig:longest element heap bipartite A7 and A8} displays the heaps for $\sort(w_0)$ for bipartite Coxeter elements in $A_7$ and $A_8$.

\begin{figure}[htb!]
\begin{minipage}{0.45\textwidth}
\begin{tikzpicture}[scale = 0.8]
\node (s1r1) at (0,0) {$s_1$};
\node (s3r1) at (2,0) {$s_3$};
\node (s5r1) at (4,0) {$s_5$};
\node (s7r1) at (6,0) {$s_7$};
\node (s2r2) at (1,1) {$s_2$};
\node (s4r2) at (3,1) {$s_4$};
\node (s6r2) at (5,1) {$s_6$};

\node (s1r3) at (0,2) {$s_1$};
\node (s3r3) at (2,2) {$s_3$};
\node (s5r3) at (4,2) {$s_5$};
\node (s7r3) at (6,2) {$s_7$};
\node (s2r4) at (1,3) {$s_2$};
\node (s4r4) at (3,3) {$s_4$};
\node (s6r4) at (5,3) {$s_6$};

\node (s1r5) at (0,2+2) {$s_1$};
\node (s3r5) at (2,2+2) {$s_3$};
\node (s5r5) at (4,2+2) {$s_5$};
\node (s7r5) at (6,2+2) {$s_7$};
\node (s2r6) at (1,3+2) {$s_2$};
\node (s4r6) at (3,3+2) {$s_4$};
\node (s6r6) at (5,3+2) {$s_6$};

\node (s1r7) at (0,2+4) {$s_1$};
\node (s3r7) at (2,2+4) {$s_3$};
\node (s5r7) at (4,2+4) {$s_5$};
\node (s7r7) at (6,2+4) {$s_7$};
\node (s2r8) at (1,3+2+2) {$s_2$};
\node (s4r8) at (3,3+2+2) {$s_4$};
\node (s6r8) at (5,3+2+2) {$s_6$};

\draw [black, thick, shorten <=-2pt, shorten >=-2pt] (s1r1) -- (s2r2) -- (s3r1) -- (s4r2) -- (s5r1) -- (s6r2) -- (s7r1) node[pos=0.5, left=5mm]{};
\draw [black, thick, shorten <=-2pt, shorten >=-2pt] (s1r3) -- (s2r2) -- (s3r3) -- (s4r2) -- (s5r3) -- (s6r2) -- (s7r3) node[pos=0.5, left=5mm]{};

\draw [black, thick, shorten <=-2pt, shorten >=-2pt] (s1r3) -- (s2r4) -- (s3r3) -- (s4r4) -- (s5r3) -- (s6r4) -- (s7r3) node[pos=0.5, left=5mm]{};
\draw [black, thick, shorten <=-2pt, shorten >=-2pt] (s1r5) -- (s2r4) -- (s3r5) -- (s4r4) -- (s5r5) -- (s6r4) -- (s7r5) node[pos=0.5, left=5mm]{};

\draw [black, thick, shorten <=-2pt, shorten >=-2pt] (s1r5) -- (s2r6) -- (s3r5) -- (s4r6) -- (s5r5) -- (s6r6) -- (s7r5) node[pos=0.5, left=5mm]{};
\draw [black, thick, shorten <=-2pt, shorten >=-2pt] (s1r7) -- (s2r6) -- (s3r7) -- (s4r6) -- (s5r7) -- (s6r6) -- (s7r7)  node[pos=0.5, left=5mm]{};

\draw [black, thick, shorten <=-2pt, shorten >=-2pt] (s1r7) -- (s2r8) -- (s3r7) -- (s4r8) -- (s5r7) -- (s6r8) -- (s7r7) node[pos=0.5, left=5mm]{};
\end{tikzpicture}
\end{minipage}
\begin{minipage}{0.45\textwidth}
\begin{tikzpicture}[scale = 0.8]
\node (s1r1) at (0,0) {$s_1$};
\node (s3r1) at (2,0) {$s_3$};
\node (s5r1) at (4,0) {$s_5$};
\node (s7r1) at (6,0) {$s_7$};
\node (s2r2) at (1,1) {$s_2$};
\node (s4r2) at (3,1) {$s_4$};
\node (s6r2) at (5,1) {$s_6$};
\node (s8r2) at (7,1) {$s_8$};

\node (s1r3) at (0,2) {$s_1$};
\node (s3r3) at (2,2) {$s_3$};
\node (s5r3) at (4,2) {$s_5$};
\node (s7r3) at (6,2) {$s_7$};
\node (s2r4) at (1,3) {$s_2$};
\node (s4r4) at (3,3) {$s_4$};
\node (s6r4) at (5,3) {$s_6$};
\node (s8r4) at (7,3) {$s_8$};

\node (s1r5) at (0,2+2) {$s_1$};
\node (s3r5) at (2,2+2) {$s_3$};
\node (s5r5) at (4,2+2) {$s_5$};
\node (s7r5) at (6,2+2) {$s_7$};
\node (s2r6) at (1,3+2) {$s_2$};
\node (s4r6) at (3,3+2) {$s_4$};
\node (s6r6) at (5,3+2) {$s_6$};
\node (s8r6) at (7,3+2) {$s_8$};

\node (s1r7) at (0,2+4) {$s_1$};
\node (s3r7) at (2,2+4) {$s_3$};
\node (s5r7) at (4,2+4) {$s_5$};
\node (s7r7) at (6,2+4) {$s_7$};
\node (s2r8) at (1,3+2+2) {$s_2$};
\node (s4r8) at (3,3+2+2) {$s_4$};
\node (s6r8) at (5,3+2+2) {$s_6$};
\node (s8r8) at (7,3+2+2) {$s_8$};

\node (s1r9) at (0,2+4+2) {$s_1$};
\node (s3r9) at (2,2+4+2) {$s_3$};
\node (s5r9) at (4,2+4+2) {$s_5$};
\node (s7r9) at (6,2+4+2) {$s_7$};

\draw [black, thick, shorten <=-2pt, shorten >=-2pt] (s1r1) -- (s2r2) -- (s3r1) -- (s4r2) -- (s5r1) -- (s6r2) -- (s7r1) --
(s8r2)
node[pos=0.5, left=5mm]{};
\draw [black, thick, shorten <=-2pt, shorten >=-2pt] (s1r3) -- (s2r2) -- (s3r3) -- (s4r2) -- (s5r3) -- (s6r2) -- (s7r3) -- 
(s8r2)
node[pos=0.5, left=5mm]{};

\draw [black, thick, shorten <=-2pt, shorten >=-2pt] (s1r3) -- (s2r4) -- (s3r3) -- (s4r4) -- (s5r3) -- (s6r4) -- (s7r3) -- 
(s8r4)
node[pos=0.5, left=5mm]{};
\draw [black, thick, shorten <=-2pt, shorten >=-2pt] (s1r5) -- (s2r4) -- (s3r5) -- (s4r4) -- (s5r5) -- (s6r4) -- (s7r5) -- 
(s8r4)
node[pos=0.5, left=5mm]{};

\draw [black, thick, shorten <=-2pt, shorten >=-2pt] (s1r5) -- (s2r6) -- (s3r5) -- (s4r6) -- (s5r5) -- (s6r6) -- (s7r5) --
(s8r6)
node[pos=0.5, left=5mm]{};
\draw [black, thick, shorten <=-2pt, shorten >=-2pt] (s1r7) -- (s2r6) -- (s3r7) -- (s4r6) -- (s5r7) -- (s6r6) -- (s7r7)  --
(s8r6)
node[pos=0.5, left=5mm]{};

\draw [black, thick, shorten <=-2pt, shorten >=-2pt] (s1r7) -- (s2r8) -- (s3r7) -- (s4r8) -- (s5r7) -- (s6r8) -- (s7r7) -- 
(s8r8)
node[pos=0.5, left=5mm]{};

\draw [black, thick, shorten <=-2pt, shorten >=-2pt] (s1r9) -- (s2r8) -- (s3r9) -- (s4r8) -- (s5r9) -- (s6r8) -- (s7r9) -- 
(s8r8)
node[pos=0.5, left=5mm]{};
\end{tikzpicture}
\end{minipage}
\caption{Left: Heap of the $\rw{1357246}$-sorting word of the longest element $w_0$ in~$A_7$. Right: Heap of the $\rw{13572468}$-sorting word of the longest element $w_0$ in $A_8$.}
\label{fig:longest element heap bipartite A7 and A8}
\end{figure}

Notice that in the bipartite case all even numbers in $[2,n]$ will be lower-barred and all odd numbers in $[2,n]$ will be upper-barred. In particular, from Proposition \ref{prop:pattern avoidance characterization type A}, we have that $w \in A_n$ is a $c$-singleton if and only if $w$ avoids patterns $312$ and $132$ where ``$2$'' is an even number and patterns $213$ and $231$ where ``$2$'' is an odd number.  This set of conditions is referred to as the alternating scheme in \cite{fishburn1996acyclic} and the number of $c$-singletons for each $n$ is enumerated in Theorem~4 of \cite{GR08}. We mark the entries which are always 0 by $\X$ in each permutation matrix for the bipartite $c$-singleton in $A_7$ and $A_8$ below.

\begin{center}
\begin{minipage}{0.45\textwidth}
\begin{tabular}{|c|c|c|c|c|c|c|c|}
\hline
  & $28$ & $\X$  & $25$ & $\X$ & $20$ & $\X$ & $13$ \\ \hline
  &    & $\X$  & $26$ & $\X$ & $21$ & $7$ & $14$ \\ \hline
  &    &    & $27$ & $\X$ & $22$ & $8$ & $15$ \\ \hline
  &    &    &   & $3$ & $23$ & $9$ & $16$ \\ \hline
  &    &    &   &    & $24$ & $10$ & $17$ \\ \hline
  & $4$  &    &   $\X$ &   &   & $11$ & $18$ \\ \hline
$5$ & $1$  &  & $\X$  &    &  $\X$ &    & $19$ \\ \hline
$2$ & $\X$  & $6$  & $\X$ &  $12$  &  $\X$ &    &    \\ \hline
\end{tabular}
\end{minipage}\hfill
\begin{minipage}{0.45\textwidth}
\begin{tabular}{|c|c|c|c|c|c|c|c|c|}
\hline
 & $36$ & $\X$  & $33$ & $\X$ & $28$ & $\X$ & $21$ & $13$ \\ \hline
  &    & $\X$  & $34$ & $\X$ & $29$ & $\X$ & $22$ & $14$ \\ \hline
  &    &    & $35$ & $\X$ & $30$ & $7$ & $23$ & $15$ \\ \hline
  &    &    &   & $\X$ & $31$ & $8$ & $24$ & $16$ \\ \hline
  &    &    &   &    & $32$ & $9$ & $25$ & $17$ \\ \hline
  &   &    & $3$ &   &   & $10$ & $26$ & $18$ \\ \hline
 & $4$  &  & $\X$  &    &  $\X$ &    & $27$ & $19$ \\ \hline
$5$ & $1$  & $11$  & $\X$ &  &  $\X$ &    &   & $20$ \\ \hline
$2$ & $\X$  & $6$  & $\X$ & $12$ &  $\X$ &    & $\X$  &  \\ \hline
\end{tabular}
\end{minipage}\hfill
\end{center}

Now, we calculate some examples from Lemma \ref{lem:first_entries_distinct_equiv_classes}. First, we compute $\nulabel_c$ and $\nulabel_{c^{-1}}$ when $c$ is the bipartite Coxeter element $\rw{1357246}$ in $A_7$. Notice that if $i \leq \frac{n+1}{2}$, $\nulabel_c(i) + \nulabel_{c^{-1}}(i) = 0$ and otherwise $\nulabel_c(i) + \nulabel_{c^{-1}}(i) = n+1$.

\begin{center}
\begin{tabular}{*{9}{|c}|}
\hline
$i$&1&2&3&4&5&6&7&8\\\hline
$\nulabel_c(i)$&0&1&-1&2&6&3&5&4\\ \hline
$\nulabel_{c^{-1}}(i)$&0&-1&1&-2&2&5&3&4\\ \hline
\end{tabular}
\end{center}

We pictorially exhibit the relations from Theorems \ref{thm:first_half_relations} and \ref{thm:second_half_relations} for the bipartite case in $A_7$. In each figure, the sum of all entries with the same suit will be 1. We begin with the cases $y = 2$ (left) and $y = 3$ (right) from Theorem \ref{thm:first_half_relations}.

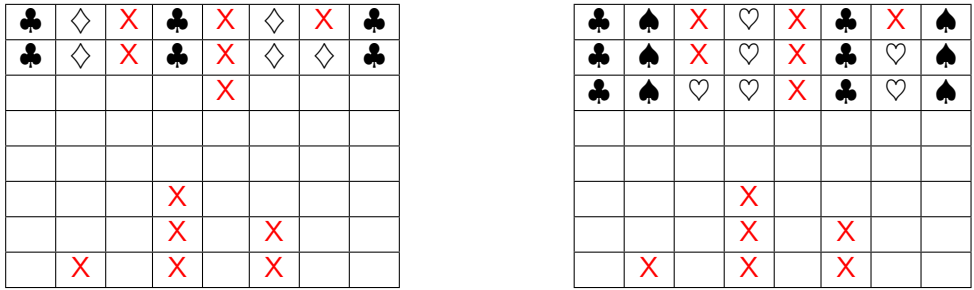
\begin{center}
\begin{minipage}{0.2\textwidth}
\end{minipage}
\begin{minipage}{0.45\textwidth}
\begin{tabular}{|c|c|c|c|c|c|c|c|}
\hline
$\clubsuit$ & $\diamondsuit$ & $\X$  & $\clubsuit$ & $\X$ & $\diamondsuit$ &  $\X$ & $\clubsuit$\\ \hline
$\clubsuit$ & $\diamondsuit$ & $\X$  & $\clubsuit$ & $\X$ & $\diamondsuit$ &  $\diamondsuit$ & $\clubsuit$\\ \hline
&  &  &  & $\X$ & &  &  \\ \hline
  &   & &   & &  & &  \\ \hline
  &    &    &   &    &  &  &  \\ \hline
  &   &   & $\X$  &    &   &  &  \\ \hline
 &   &  &  $\X$ &    & $\X$  &   &  \\ \hline
 & $\X$  &  &  $\X$  &   & $\X$   & & \hspace{0.25cm}   \\ \hline
\end{tabular}
\end{minipage}
\begin{minipage}{0.45\textwidth}
\begin{tabular}{|c|c|c|c|c|c|c|c|}
\hline
$\clubsuit$ & $\spadesuit$ & $\X$  & $\heartsuit$ & $\X$ & $\clubsuit$ &  $\X$ & $\spadesuit$\\ \hline
$\clubsuit$ & $\spadesuit$ & $\X$  & $\heartsuit$ & $\X$ & $\clubsuit$ &  $\heartsuit$ & $\spadesuit$\\ \hline
$\clubsuit$ & $\spadesuit$ & $\heartsuit$  & $\heartsuit$ & $\X$ & $\clubsuit$ &  $\heartsuit$ & $\spadesuit$\\ \hline
  &   & &   & &  & &  \\ \hline
  &    &    &   &    &  &  &  \\ \hline
  &   &   & $\X$  &    &   &  &  \\ \hline
 &   &  &  $\X$ &    & $\X$  &   &  \\ \hline
 & $\X$  &  &  $\X$  &   & $\X$   & & \hspace{0.25cm}   \\ \hline
\end{tabular}
\end{minipage}
\end{center}

Finally, from Theorem \ref{thm:first_half_relations} when $y = 4$, we get the following relations for $A_7$. These relations along with column relations imply the relations for $y = 4$ from Theorem \ref{thm:second_half_relations}.

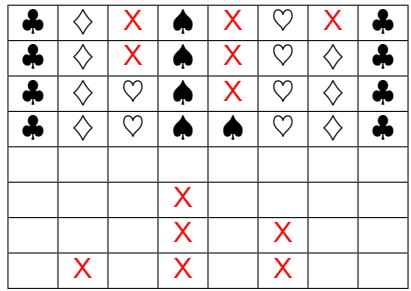
\begin{center}
\begin{tabular}{|c|c|c|c|c|c|c|c|}
\hline
$\clubsuit$ & $\diamondsuit$& $\X$  & $\spadesuit$ & $\X$ & $\heartsuit$ &  $\X$ &  $\clubsuit$\\ \hline
$\clubsuit$ & $\diamondsuit$ & $\X$  & $\spadesuit$& $\X$ & $\heartsuit$ & $\diamondsuit$ &  $\clubsuit$\\ \hline
$\clubsuit$ & $\diamondsuit$ & $\heartsuit$ & $\spadesuit$ & $\X$ & $\heartsuit$ & $\diamondsuit$ & $\clubsuit$ \\ \hline
$\clubsuit$  & $\diamondsuit$  & $\heartsuit$ & $\spadesuit$ & $\spadesuit$ & $\heartsuit$ & $\diamondsuit$ & $\clubsuit$ \\ \hline
  &    &    &   &    &  &  &  \\ \hline
  &   &   & $\X$  &    &   &  &  \\ \hline
 &   &  &  $\X$ &    & $\X$  &   &  \\ \hline
 & $\X$  &  &  $\X$  &   & $\X$   & & \hspace{0.25cm}   \\ \hline
\end{tabular}
\end{center}

Lastly, we give the relations for $y = 2$ (left) and $y = 3$ (right) from Theorem \ref{thm:second_half_relations}.

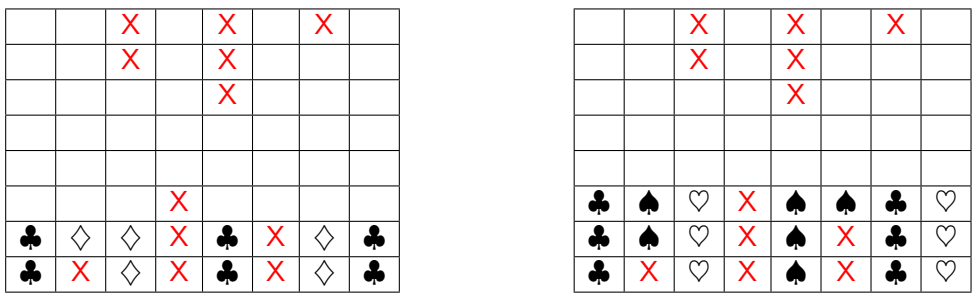
\begin{center}
\begin{minipage}{0.2\textwidth}
\end{minipage}
\begin{minipage}{0.45\textwidth}
\begin{tabular}{|c|c|c|c|c|c|c|c|}
\hline
\hspace{0.25cm} &  & $\X$  &  & $\X$ &  &  $\X$ &  \\ \hline
&   & $\X$  &  & $\X$ &  &  &  \\ \hline
 &    &    &  & $\X$ & &  &  \\ \hline
  &    &    &   &  &  &  & \\ \hline
  &    &    &   &    &  &  &  \\ \hline
 &  &   & $\X$  &   &   &  & \\ \hline
$\clubsuit$ & $\diamondsuit$  & $\diamondsuit$ &  $\X$ &  $\clubsuit$  & $\X$  & $\diamondsuit$  &  $\clubsuit$\\ \hline
$\clubsuit$ & $\X$  & $\diamondsuit$ &  $\X$  &  $\clubsuit$ & $\X$   & $\diamondsuit$ & $\clubsuit$   \\ \hline
\end{tabular}
\end{minipage}
\begin{minipage}{0.45\textwidth}
\begin{tabular}{|c|c|c|c|c|c|c|c|}
\hline
\hspace{0.25cm} &  & $\X$  &  & $\X$ &  &  $\X$ &  \\ \hline
&   & $\X$  &  & $\X$ &  &  &  \\ \hline
 &    &    &  & $\X$ & &  &  \\ \hline
  &    &    &   &  &  &  & \\ \hline
  &    &    &   &    &  &  &  \\ \hline
$\clubsuit$  & $\spadesuit$  & $\heartsuit$  & $\X$  &  $\spadesuit$  & $\spadesuit$  & $\clubsuit$ & $\heartsuit$ \\ \hline
$\clubsuit$ & $\spadesuit$  & $\heartsuit$ &  $\X$ &  $\spadesuit$  & $\X$  & $\clubsuit$  &  $\heartsuit$\\ \hline
$\clubsuit$ & $\X$  & $\heartsuit$ &  $\X$  &  $\spadesuit$ & $\X$   & $\clubsuit$ & $\heartsuit$   \\ \hline
\end{tabular}
\end{minipage}
\end{center}

For our example, the ``diagonal reading word" of $H_c$, as defined in \eqref{eq:Def:Rc}, is \[\DiagonalReadingWord = \rw{1}\rw{321}\rw{54321}\rw{7654321}\rw{765432}\rw{7654}\rw{76}.\]  Using $\DiagonalReadingWord$, we can compute $b_1,b_2,\dots, b_{28}$.  Then the unimodular transformation matrix $\mathcal{U}_c$ from the projection of $\bir(c)$ to $\ord(\HeapDiagonalReading)$ defined in Proposition~\ref{prop.square} is below.
\scriptsize
\setcounter{MaxMatrixCols}{28}
\begingroup % keep the change local
\setlength\arraycolsep{4.5pt}
\[
\begin{pmatrix}
1 & 0 & 0 & 0 & 0 & 0 & 0 & 0 & 0 & 0 & 0 & 0 & 0 & 0 & 0 & 0 & 0 & 0 & 0 & 0 & 0 & 0 & 0 & 0 & 0 & 0 & 0 & 0 \\
0 & 1 & 0 & 0 & 0 & 0 & 0 & 0 & 0 & 0 & 0 & 0 & 0 & 0 & 0 & 0 & 0 & 0 & 0 & 0 & 0 & 0 & 0 & 0 & 0 & 0 & 0 & 0 \\
0 & 0 & 1 & 0 & 0 & 0 & 0 & 0 & 0 & 0 & 0 & 0 & 0 & 0 & 0 & 0 & 0 & 0 & 0 & 0 & 0 & 0 & 0 & 0 & 0 & 0 & 0 & 0 \\
1 & 0 & 0 & 1 & 0 & 0 & 0 & 0 & 0 & 0 & 0 & 0 & 0 & 0 & 0 & 0 & 0 & 0 & 0 & 0 & 0 & 0 & 0 & 0 & 0 & 0 & 0 & 0 \\
0 & 1 & 0 & 0 & 1 & 0 & 0 & 0 & 0 & 0 & 0 & 0 & 0 & 0 & 0 & 0 & 0 & 0 & 0 & 0 & 0 & 0 & 0 & 0 & 0 & 0 & 0 & 0 \\
0 & 1 & 0 & 0 & 0 & 1 & 0 & 0 & 0 & 0 & 0 & 0 & 0 & 0 & 0 & 0 & 0 & 0 & 0 & 0 & 0 & 0 & 0 & 0 & 0 & 0 & 0 & 0 \\
0 & 0 & 0 & 0 & 0 & 0 & 1 & 0 & 0 & 0 & 0 & 0 & 0 & 0 & 0 & 0 & 0 & 0 & 0 & 0 & 0 & 0 & 0 & 0 & 0 & 0 & 0 & 0 \\
0 & 0 & 0 & 0 & 0 & 0 & 1 & 1 & 0 & 0 & 0 & 0 & 0 & 0 & 0 & 0 & 0 & 0 & 0 & 0 & 0 & 0 & 0 & 0 & 0 & 0 & 0 & 0 \\
0 & 0 & 0 & 0 & 0 & 0 & 1 & 1 & 1 & 0 & 0 & 0 & 0 & 0 & 0 & 0 & 0 & 0 & 0 & 0 & 0 & 0 & 0 & 0 & 0 & 0 & 0 & 0 \\
0 & 0 & 0 & 0 & 0 & 0 & 1 & 1 & 1 & 1 & 0 & 0 & 0 & 0 & 0 & 0 & 0 & 0 & 0 & 0 & 0 & 0 & 0 & 0 & 0 & 0 & 0 & 0 \\
0 & 0 & 0 & 0 & 0 & 0 & 1 & 1 & 1 & 1 & 1 & 0 & 0 & 0 & 0 & 0 & 0 & 0 & 0 & 0 & 0 & 0 & 0 & 0 & 0 & 0 & 0 & 0 \\
0 & 1 & 0 & 0 & 0 & 1 & 0 & 0 & 0 & 0 & 0 & 1 & 0 & 0 & 0 & 0 & 0 & 0 & 0 & 0 & 0 & 0 & 0 & 0 & 0 & 0 & 0 & 0 \\
0 & 0 & 0 & 0 & 0 & 0 & 0 & 0 & 0 & 0 & 0 & 0 & 1 & 0 & 0 & 0 & 0 & 0 & 0 & 0 & 0 & 0 & 0 & 0 & 0 & 0 & 0 & 0 \\
0 & 0 & 0 & 0 & 0 & 0 & 0 & 0 & 0 & 0 & 0 & 0 & 1 & 1 & 0 & 0 & 0 & 0 & 0 & 0 & 0 & 0 & 0 & 0 & 0 & 0 & 0 & 0 \\
0 & 0 & 0 & 0 & 0 & 0 & 0 & 0 & 0 & 0 & 0 & 0 & 1 & 1 & 1 & 0 & 0 & 0 & 0 & 0 & 0 & 0 & 0 & 0 & 0 & 0 & 0 & 0 \\
0 & 0 & 0 & 0 & 0 & 0 & 0 & 0 & 0 & 0 & 0 & 0 & 1 & 1 & 1 & 1 & 0 & 0 & 0 & 0 & 0 & 0 & 0 & 0 & 0 & 0 & 0 & 0 \\
0 & 0 & 0 & 0 & 0 & 0 & 0 & 0 & 0 & 0 & 0 & 0 & 1 & 1 & 1 & 1 & 1 & 0 & 0 & 0 & 0 & 0 & 0 & 0 & 0 & 0 & 0 & 0 \\
0 & 0 & 0 & 0 & 0 & 0 & 0 & 0 & 0 & 0 & 0 & 0 & 1 & 1 & 1 & 1 & 1 & 1 & 0 & 0 & 0 & 0 & 0 & 0 & 0 & 0 & 0 & 0 \\
0 & 0 & 0 & 0 & 0 & 0 & 0 & 0 & 0 & 0 & 0 & 0 & 1 & 1 & 1 & 1 & 1 & 1 & 1 & 0 & 0 & 0 & 0 & 0 & 0 & 0 & 0 & 0 \\
0 & 0 & 0 & 0 & 0 & 0 & 0 & 0 & 0 & 0 & 0 & 0 & 1 & 0 & 0 & 0 & 0 & 0 & 0 & 1 & 0 & 0 & 0 & 0 & 0 & 0 & 0 & 0 \\
0 & 0 & 0 & 0 & 0 & 0 & 1 & 0 & 0 & 0 & 0 & 0 & 0 & 0 & 0 & 0 & 0 & 0 & 0 & 1 & 1 & 0 & 0 & 0 & 0 & 0 & 0 & 0 \\
0 & 0 & 0 & 0 & 0 & 0 & 0 & 0 & 0 & 0 & 0 & 0 & 0 & 0 & 0 & 0 & 0 & 0 & 0 & 1 & 1 & 1 & 0 & 0 & 0 & 0 & 0 & 0 \\
0 & 0 & 0 & 0 & 0 & 0 & 0 & 0 & 0 & 0 & 0 & 0 & 0 & 0 & 0 & 0 & 0 & 0 & 0 & 1 & 1 & 1 & 1 & 0 & 0 & 0 & 0 & 0 \\
0 & 0 & 0 & 0 & 0 & 0 & 0 & 0 & 0 & 0 & 0 & 0 & 0 & 0 & 0 & 0 & 0 & 0 & 0 & 1 & 1 & 1 & 1 & 1 & 0 & 0 & 0 & 0 \\
0 & 0 & 0 & 0 & 0 & 0 & 0 & 0 & 0 & 0 & 0 & 0 & 1 & 0 & 0 & 0 & 0 & 0 & 0 & 1 & 0 & 0 & 0 & 0 & 1 & 0 & 0 & 0 \\
0 & 0 & 0 & 0 & 0 & 0 & 0 & 0 & 0 & 0 & 0 & 0 & 1 & 1 & 0 & 0 & 0 & 0 & 0 & 0 & 0 & 0 & 0 & 0 & 1 & 1 & 0 & 0 \\
0 & 0 & 0 & 0 & 0 & 0 & 1 & 1 & 0 & 0 & 0 & 0 & 0 & 0 & 0 & 0 & 0 & 0 & 0 & 0 & 0 & 0 & 0 & 0 & 1 & 1 & 1 & 0 \\
0 & 0 & 0 & 0 & 0 & 0 & 0 & 0 & 0 & 0 & 0 & 0 & 1 & 0 & 0 & 0 & 0 & 0 & 0 & 1 & 0 & 0 & 0 & 0 & 1 & 0 & 0 & 1
\end{pmatrix}
\]
\endgroup
\normalsize
This is a unitriangular matrix, and thus has determinant 1.

\section{Future Work}
\label{Sec:future}

In Definition~\ref{def:cBirkhoff}, we defined $\bir(c)$ to be the convex hull of the permutation matrices $X(w)$ of all $c$-singletons $w$; and we proved that $\bir(c)$ is unimodularly equivalent to the order polytope $\ord(\heap(\rw{u}))$ of $\heap(\rw{u})$, where $\rw{u}=\sort(w_0)$ which is a reduced word for $w_0$. Since the $c$-singletons are in bijection with the order ideals of $\heap(\rw{u})$ (see the bijections $\MapPermutation$ and $f$ defined in Proposition~\ref{prop:I=Heap([r])} and Proposition~\ref{prop:vertex-bijection}), the polytope $\bir(c)$ is equal to the convex hull of
\begin{equation}
\label{eq:sec8.1}
\{ X(\MapPermutation(I)) \ \vert \ I \text{ is an order ideal of } \heap(\rw{u}) \}. 
\end{equation}

Inspired by \eqref{eq:sec8.1}, we could similarly define a subpolytope of the Birkhoff polytope for an arbitrary reduced word. This leads us to the following open question.

\begin{question}\label{question:generalization in type A}
Let $w\in A_n$ and $\rw{u}$ be a reduced word for $w$. Is the order polytope of $\heap(\rw{u})$ unimodularly equivalent to the convex hull of the permutation matrices corresponding to order ideals of $\heap(\rw{u})$? 
\end{question}

We used SageMath \cite{sagemath} to compare several polytope invariants including dimensions, (normalized) volumes, $f$-vectors for the two polytopes in all cases in $A_1,A_2, A_3$ and $A_4$, and found that they are always equal except for two counterexamples of $\rw{u}$: $\rw{3432312343}$ and $\rw{2123243212}$. While not conclusive, this data suggests that our main result, Theorem~\ref{thm:main}, can potentially be extended much further.  In the next example, we explore the counterexample of $\rw{2123243212}$.

\begin{example}\label{ex:counterexample in A4}
Consider the reduced word $R=\rw{2123243212}$ for $w_0 \in A_4$. 
The heap $H=\heap(R)$ of $R$ is illustrated in Figure \ref{fig:counterexample2123243212}. Since $H$ has ten elements, the order polytope $\ord(H)$ is of dimension~$10$ (see Section \ref{sec:order polytope}). 
\begin{figure}[htb!]
\begin{tikzpicture}[scale=0.9]
\node (s2a) at (1,1) {$s_2$};
\node (s4a) at (3,5) {$s_4$};
\node (s1b) at (0,2) {$s_1$};
\node (s2b) at (1,3) {$s_2$};
\node (s3b) at (2,4) {$s_3$};
\node (s3c) at (2,6) {$s_3$};
\node (s2c) at (1,5) {$s_2$};
\node (s2d) at (1,7) {$s_2$};
 \node (s1e) at (0,8) {$s_1$};
 \node (s2e) at (1,9) {$s_2$};

 \draw [thick, shorten <=-2pt, shorten >=-2pt] (s3c) -- (s2d) node[pos=0.5, left=5mm]{};
 \draw [thick, shorten <=-2pt, shorten >=-2pt] (s2d) -- (s1e) node[pos=0.5, left=5mm]{};
  \draw [thick, shorten <=-2pt, shorten >=-2pt] (s2e) -- (s1e) node[pos=0.5, left=5mm]{};
 \draw [thick, shorten <=-2pt, shorten >=-2pt] (s4a) -- (s3c) node[pos=0.5, left=5mm]{};
\draw [thick, shorten <=-2pt, shorten >=-2pt] (s2c) -- (s3c) node[pos=0.5, left=5mm]{};

\draw [thick, shorten <=-2pt, shorten >=-2pt] (s1b) -- (s2b) node[pos=0.5, left=5mm]{};
\draw [thick, shorten <=-2pt, shorten >=-2pt] (s2b) -- (s3b) node[pos=0.5, left=5mm]{};

\draw [thick, shorten <=-2pt, shorten >=-2pt] (s2a) -- (s1b); 
\draw [thick, shorten <=-2pt, shorten >=-2pt] (s4a) -- (s3b) -- (s2c);
\end{tikzpicture}
\caption{Heap of $\rw{2123243212}$, a reduced word for the longest element $w_0$ in $A_4$}
\label{fig:counterexample2123243212}
\end{figure}

On the other hand, consider the subpolytope $B$ of the Birkhoff polytope in $A_4$ defined by taking the convex hull of the permutations corresponding to the order ideals of $H$. In contrast to $\ord(H)$, the dimension of $B$ is at most $9$. To see this, first we list the twelve permutations. 

\begin{center}
\begin{tabular}{ccccc}
$\begin{pmatrix}
1&\mathbf{0}&0&\mathbf{0}&0\\
0&1&0&0&0\\
0&0&1&0&0\\
0&0&\mathbf{0}&1&0\\
0&\mathbf{0}&\mathbf{0}&\mathbf{0}&1\\
\end{pmatrix}$&
$\begin{pmatrix}
1&\mathbf{0}&0&\mathbf{0}&0\\
0&0&1&0&0\\
0&1&0&0&0\\
0&0&\mathbf{0}&1&0\\
0&\mathbf{0}&\mathbf{0}&\mathbf{0}&1\\
\end{pmatrix}$&
$\begin{pmatrix}
0&\mathbf{0}&1&\mathbf{0}&0\\
1&0&0&0&0\\
0&1&0&0&0\\
0&0&\mathbf{0}&1&0\\
0&\mathbf{0}&\mathbf{0}&\mathbf{0}&1\\
\end{pmatrix}$&
$\begin{pmatrix}
0&\mathbf{0}&1&\mathbf{0}&0\\
0&1&0&0&0\\
1&0&0&0&0\\
0&0&\mathbf{0}&1&0\\
0&\mathbf{0}&\mathbf{0}&\mathbf{0}&1\\
\end{pmatrix}$\\
Id & $s_2$ & $s_2s_1$&$s_2s_1s_2$\\
$\begin{pmatrix}
0&\mathbf{0}&1&\mathbf{0}&0\\
0&1&0&0&0\\
0&0&0&1&0\\
1&0&\mathbf{0}&0&0\\
0&\mathbf{0}&\mathbf{0}&\mathbf{0}&1\\
\end{pmatrix}$&
$\begin{pmatrix}
0&\mathbf{0}&1&\mathbf{0}&0\\
0&0&0&1&0\\
0&1&0&0&0\\
1&0&\mathbf{0}&0&0\\
0&\mathbf{0}&\mathbf{0}&\mathbf{0}&1\\
\end{pmatrix}$&
$\begin{pmatrix}
0&\mathbf{0}&1&\mathbf{0}&0\\
0&1&0&0&0\\
0&0&0&1&0\\
0&0&\mathbf{0}&0&1\\
1&\mathbf{0}&\mathbf{0}&\mathbf{0}&0\\
\end{pmatrix}$&
$\begin{pmatrix}
0&\mathbf{0}&1&\mathbf{0}&0\\
0&0&0&1&0\\
0&1&0&0&0\\
0&0&\mathbf{0}&0&1\\
1&\mathbf{0}&\mathbf{0}&\mathbf{0}&0\\
\end{pmatrix}$\\
$s_2s_1s_2s_3$ & $s_2s_1s_2s_3s_2$& $s_2s_1s_2s_3s_4$ &  $s_2s_1s_2s_3s_2s_4$ & \\
$\begin{pmatrix}
0&\mathbf{0}&1&\mathbf{0}&0\\
0&0&0&1&0\\
0&0&0&0&1\\
0&1&\mathbf{0}&0&0\\
1&\mathbf{0}&\mathbf{0}&\mathbf{0}&0\\
\end{pmatrix}$&
$\begin{pmatrix}
0&\mathbf{0}&1&\mathbf{0}&0\\
0&0&0&0&1\\
0&0&0&1&0\\
0&1&\mathbf{0}&0&0\\
1&\mathbf{0}&\mathbf{0}&\mathbf{0}&0\\
\end{pmatrix}$&
$\begin{pmatrix}
0&\mathbf{0}&0&\mathbf{0}&1\\
0&0&1&0&0\\
0&0&0&1&0\\
0&1&\mathbf{0}&0&0\\
1&\mathbf{0}&\mathbf{0}&\mathbf{0}&0\\
\end{pmatrix}$&
$\begin{pmatrix}
0&\mathbf{0}&0&\mathbf{0}&1\\
0&0&0&1&0\\
0&0&1&0&0\\
0&1&\mathbf{0}&0&0\\
1&\mathbf{0}&\mathbf{0}&\mathbf{0}&0\\
\end{pmatrix}$\\
$s_2s_1s_2s_3s_2s_4s_3$ & $s_2s_1s_2s_3s_2s_4s_3s_2$& $s_2s_1s_2s_3s_2s_4s_3s_2s_1$&
$s_2s_1s_2s_3s_2s_4s_3s_2s_1s_2$\\
\end{tabular}
\end{center}

We notice that none of the permutation matrices has a 1 in entries $(1,2), (1,4),(4,3),(5,2),(5,3),$ or $(5,4)$. Moreover, we see that each matrix has exactly one nonzero entry out of the following set $\{(1,1),(2,1),(3,1),(4,2),(4,3)\}$. Therefore,  we get the following relations on the affine span of the vectors given by the permutations above:\[
X(1,2) = 0 \quad X(1,4) = 0 \quad X(4,3) = 0 \quad X(5,2) = 0 \quad X(5,3) = 0 \quad X(5,4) = 0
\]
\[
X(1,1) + X(2,1) + X(3,1) + X(4,2) + X(4,3) = 1
\]

Of the ten row and column relations from Lemma \ref{lem:RowAndCol}, we have that any set of 9 is linearly independent but the last would be a linear combination of the 9. Each of these extra relations is linearly independent with a set of 9 row and column relations. This implies that the affine span of these vectors lives in at most a $25 - (9 + 6 + 1) = 9$ dimensional subspace of $\mathbb{R}^{25}$. In particular, the convex hull of these vectors cannot be unimodularly equivalent to the order polytope of $\heap([2123243212])$ as the latter is 10-dimensional.
\end{example}

We note that the answer to Question~\ref{question:generalization in type A} 
does depend on the reduced word $\rw{u}$, not just the permutation $w$.  For example, our main theorem shows that the answer to Question~\ref{question:generalization in type A} is yes when $\rw{u}=\sort(w_0)$ but Example~\ref{ex:counterexample in A4} shows that the answer to Question~\ref{question:generalization in type A} is no when $\rw{u}=[2123243212]$, even though both are reduced words for $w_0\in A_4$.

Our second future research direction is to consider other types.

\begin{question}
Does Theorem~\ref{thm:main} extend beyond type $A$?
\end{question}

Experiments in SageMath suggest that the answer is generally no.  In particular, this does not seem to hold for type $D$.  However, preliminary computations in SageMath are promising for type $B$.

Finally, in the paper that inspired our work, \cite{DS18}, Davis and Sagan explored both pattern-avoiding Birkhoff polytopes $B_n(\Pi)$ and pattern-avoiding permutahedra $P_n(\Pi)$.  In this paper we defined $c$-Birkhoff polytopes $\bir(c)$, which coincides with $B_n(\Pi)$ when $c=s_1\dots s_{n-1}$ and $\Pi=\{132,312\}$.  Similarly, we can define $c$-permutahedra as $$P(c)=\{(a_1,\dots,a_n)\ |\ a_1\dots a_n\text{ is a }c\text{-singleton}\}$$ so that $P(c)$ and $P_n(\Pi)$ again coincide when $c=s_1\dots s_{n-1}$ and $\Pi=\{132,312\}$.

\begin{question}
Could any of Davis and Sagan's results in \cite{DS18} be recovered or extended by investigating $c$-singleton permutahedra?
\end{question}

\subsection*{Acknowledgements}
We thank the Minnesota Research Workshop in Algebraic Combinatorics 2022 (MRWAC 2022) organizers for bringing us together. 
This project benefited from the input and feedback of many people including  
Emily Barnard,
Rob Davis,
Jesus De Loera, 
Galen Dorpalen-Barry, 
William Dugan, 
Peter J{\o}rgensen, 
Elizabeth Kelley, 
Jean-Philippe Labb\'{e}, 
Fu Liu, 
Emily Meehan,
Alejandro Morales,
David Nkansah,
Sasha Pevzner, 
Nathan Reading,
Bruce Sagan, 
Amit Shah,
Frank Sottile, 
Jessica Striker, 
and Sylvester Zhang. 
We are also grateful to the anonymous reviewer whose suggestions helped improve and clarify this paper.
Esther Banaian was supported by Research Project 2 from the Independent Research Fund Denmark (grant no. 1026-00050B). 

\printbibliography

\end{document}